\documentclass[reqno]{amsart}

\usepackage[all]{xy}
\usepackage{graphicx}

\usepackage{amssymb}
\usepackage{amsmath}
\usepackage{mathrsfs}
\usepackage{epsfig}
\usepackage{amscd}
\usepackage{graphicx, color}

\newcommand{\red}{\textcolor{red}}

\def\E{\ifmmode{\mathbb E}\else{$\mathbb E$}\fi} 
\def\N{\ifmmode{\mathbb N}\else{$\mathbb N$}\fi} 
\def\R{\ifmmode{\mathbb R}\else{$\mathbb R$}\fi} 
\def\Q{\ifmmode{\mathbb Q}\else{$\mathbb Q$}\fi} 
\def\C{\ifmmode{\mathbb C}\else{$\mathbb C$}\fi} 
\def\H{\ifmmode{\mathbb H}\else{$\mathbb H$}\fi} 
\def\Z{\ifmmode{\mathbb Z}\else{$\mathbb Z$}\fi} 
\def\P{\ifmmode{\mathbb P}\else{$\mathbb P$}\fi} 
\def\T{\ifmmode{\mathbb T}\else{$\mathbb T$}\fi} 
\def\SS{\ifmmode{\mathbb S}\else{$\mathbb S$}\fi} 
\def\DD{\ifmmode{\mathbb D}\else{$\mathbb D$}\fi} 

\newcommand{\e}{\varepsilon}

\newcommand{\del}{\partial}
\newcommand{\Cont}{{\operatorname{Cont}}}
\newcommand{\Hom}{{\operatorname{Hom}}}

\newcommand{\ben}{\begin{enumerate}}
\newcommand{\een}{\end{enumerate}}
\newcommand{\be}{\begin{equation}}
\newcommand{\ee}{\end{equation}}
\newcommand{\bea}{\begin{eqnarray}}
\newcommand{\eea}{\end{eqnarray}}
\newcommand{\beastar}{\begin{eqnarray*}}
\newcommand{\eeastar}{\end{eqnarray*}}
\newcommand{\bc}{\begin{center}}
\newcommand{\ec}{\end{center}}

\theoremstyle{theorem}
\newtheorem{thm}{Theorem}[section]
\newtheorem{cor}[thm]{Corollary}
\newtheorem{lem}[thm]{Lemma}
\newtheorem{prop}[thm]{Proposition}

\theoremstyle{definition}
\newtheorem{defn}[thm]{Definition}
\newtheorem{rem}[thm]{Remark}

\newtheorem{exm}[thm]{Example}
\newtheorem{hypo}[thm]{Hypothesis}

\newtheorem*{thm*}{Theorem}

\numberwithin{equation}{section}

\def\R{{\mathbb R}}

\def\E{{\mathbb E}}
\def\Z{{\mathbb Z}}
\def\C{{\mathbb C}}
\def\R{{\mathbb R}}
\def\P{{\mathbb P}}

\def\N{{\mathbb N}}

\def\11{{\mathbb I}}

\def\delbar{{\overline \partial}}

\def\C{\mathbb{C}}
\def\Z{\mathbb{Z}}

\def\T{\mathbb{T}}
\def\L{\mathbb{L}}

\def\Q{\mathbb{Q}}

\def\E{\ifmmode{\mathbb E}\else{$\mathbb E$}\fi} 
\def\N{\ifmmode{\mathbb N}\else{$\mathbb N$}\fi} 
\def\R{\ifmmode{\mathbb R}\else{$\mathbb R$}\fi} 
\def\Q{\ifmmode{\mathbb Q}\else{$\mathbb Q$}\fi} 
\def\C{\ifmmode{\mathbb C}\else{$\mathbb C$}\fi} 
\def\H{\ifmmode{\mathbb H}\else{$\mathbb H$}\fi} 
\def\Z{\ifmmode{\mathbb Z}\else{$\mathbb Z$}\fi} 
\def\P{\ifmmode{\mathbb P}\else{$\mathbb P$}\fi} 
\def\SS{\ifmmode{\mathbb S}\else{$\mathbb S$}\fi} 
\def\DD{\ifmmode{\mathbb D}\else{$\mathbb D$}\fi} 

\def\R{{\mathbb R}}

\def\E{{\mathbb E}}
\def\Z{{\mathbb Z}}
\def\C{{\mathbb C}}
\def\R{{\mathbb R}}

\def\N{{\mathbb N}}

\def\delbar{{\overline \partial}}

\def\e{\varepsilon}

\def\CE{{\mathcal E}}
\def\CF{{\mathcal F}}

\def\CH{{\mathcal H}}
\def\CI{{\mathcal I}}
\def\CJ{{\mathcal J}}

\def\CL{{\mathcal L}}

\def\CN{{\mathcal N}}

\def\CQ{{\mathcal Q}}
\def\CR{{\mathcal R}}

\def\CT{{\mathcal T}}
\def\CU{{\mathcal U}}
\def\CV{{\mathcal V}}

\def\grad#1{\,\nabla\!_{{#1}}\,}

\def\darr#1{\raise1.5ex\hbox{$\leftrightarrow$}
\mkern-16.5mu #1}

\def\roughly#1{\raise.3ex\hbox{$#1$\kern-.75em
\lower1ex\hbox{$\sim$}}}

\def\opname#1{\mathop{\kern0pt{\rm #1}}\nolimits}

\def\dim{\opname{dim}}

\def\grad{\opname{grad}}

\def\span{\operatorname{span}}

\begin{document}

\quad \vskip1.375truein

\def\mq{\mathfrak{q}}
\def\mp{\mathfrak{p}}
\def\mH{\mathfrak{H}}
\def\mh{\mathfrak{h}}
\def\ma{\mathfrak{a}}
\def\ms{\mathfrak{s}}
\def\mm{\mathfrak{m}}
\def\mn{\mathfrak{n}}
\def\mz{\mathfrak{z}}
\def\mw{\mathfrak{w}}
\def\Hoch{{\tt Hoch}}
\def\mt{\mathfrak{t}}
\def\ml{\mathfrak{l}}
\def\mT{\mathfrak{T}}
\def\mL{\mathfrak{L}}
\def\mg{\mathfrak{g}}
\def\md{\mathfrak{d}}
\def\mr{\mathfrak{r}}

\title[Exponential convergence for Morse-Bott case]{Analysis of contact Cauchy-Riemann maps II:
canonical neighborhoods and exponential convergence for the Morse-Bott case}

\author{Yong-Geun Oh, Rui Wang}
\address{Center for Geometry and Physics, Institute for Basic Science (IBS),
77 Cheongam-ro, Nam-gu, Pohang, Korea 37673
\& Department of Mathematics, POSTECH, Pohang, Korea, 37673}
\email{yongoh1@postech.ac.kr}
\address{University of California, Irvine, 340 Rowland Hall (Bldg.\# 400),
Irvine, CA 92697-3875, USA}
\email{ruiw10@math.uci.edu}
\thanks{This work is supported by the IBS project \# IBS-R003-D1}

\date{}

\begin{abstract} This is a sequel to the papers \cite{oh-wang1}, \cite{oh-wang2}.
In \cite{oh-wang1}, the authors introduced a canonical
affine connection on $M$ associated to the contact triad $(M,\lambda,J)$.
In \cite{oh-wang2}, they used the connection to establish a priori $W^{k,p}$-coercive estimates
for maps $w: \dot \Sigma \to M$ satisfying $\delbar^\pi w= 0, \, d(w^*\lambda \circ j) = 0$
\emph{without involving symplectization}. We call such a pair $(w,j)$  a contact instanton.
In this paper, we first prove a canonical neighborhood theorem of the locus $Q$ foliated by
closed Reeb orbits of a Morse-Bott contact form.
Then using a general framework of the three-interval method,
we establish exponential decay estimates for contact instantons $(w,j)$ of the triad $(M,\lambda,J)$,
with $\lambda$ a Morse-Bott contact form and
$J$ a CR-almost complex structure adapted to $Q$,
under the condition that the asymptotic charge of $(w,j)$ at the associated puncture vanishes.

We also apply the three-interval method to the symplectization case and provide an alternative
approach via tensorial calculations to exponential decay estimates in the Morse-Bott case
for the pseudoholomorphic curves on the symplectization of contact manifolds.
This was previously established by Bourgeois \cite{bourgeois} (resp. by Bao \cite{bao}),
by using special coordinates, for the cylindrical (resp. for the asymptotically cylindrical) ends.
The exponential decay result for the Morse-Bott case is an essential ingredient in the set-up of
the moduli space of pseudoholomorphic curves which plays a central role in contact homology and
symplectic field theory (SFT).
\end{abstract}


\maketitle

\tableofcontents

\section{Introduction}
\label{sec:intor}

Let $(M,\xi)$ be a contact manifold.
Each contact form $\lambda$ of $\xi$, i.e., a one-form with $\ker \lambda = \xi$, canonically induces a splitting
$$
TM = \R\{X_\lambda\} \oplus \xi.
$$
Here $X_\lambda$ is the Reeb vector field of $\lambda$,
which is uniquely determined by the equations
$$
X_\lambda \rfloor \lambda \equiv 1, \quad X_\lambda \rfloor d\lambda \equiv 0.
$$
We denote by $\Pi=\Pi_\lambda: TM \to TM$ the idempotent, i.e., an endomorphism satisfying
$\Pi^2 = \Pi$ such that $\ker \Pi = \R\{X_\lambda\}$ and $\operatorname{Im} \Pi = \xi$.
Denote by $\pi=\pi_\lambda: TM \to \xi$ the associated projection.

\begin{defn}[Contact Triad]
We call the triple $(M,\lambda,J)$ a contact triad of $(M,\xi)$ if
$\lambda$ is a contact form of $(M,\xi)$,
and $J$ is an endomorphism of $TM$ with $J^2 = -\Pi$ which we call
$CR$-almost complex structure, such that the triple $(\xi, d\lambda|_\xi,J|_\xi)$ defines a
Hermitian vector bundle over $M$.
\end{defn}
As long as no confusion arises, we abuse our notation $J$ also for its restriction to $\xi$.

In \cite{oh-wang2}, the authors of the present paper called the pair $(w,j)$ a contact instanton, if
$(\Sigma, j)$ is a (punctured) Riemann surface and $w:\Sigma\to M$ satisfies the following equations
\be\label{eq:contact-instanton}
\delbar^\pi w = 0, \quad d(w^*\lambda \circ j) = 0.
\ee
A priori coercive $W^{k,2}$-estimates for $w$ with $W^{1,2}$-bound was established
\emph{without involving symplectization}. Moreover, the study of $W^{1,2}$ (or the derivative) bound
and the definition of relevant energy is carried out by the first-named author in \cite{oh:energy}.

Furthermore, for the punctured domains $\dot\Sigma$ equipped with cylindrical metric near the puncture,
the present authors proved the result of asymptotic subsequence \emph{uniform} convergence to a Reeb orbit
(which must be closed
when the corresponding charge is vanishing) under the assumption that the $\pi$-harmonic energy is finite
and the $C^0$-norm of derivative $dw$ is bounded. (Refer \cite[Section 6]{oh-wang2} for precise statement and
Section \ref{sec:pseudo} in the current paper for its review.)
Based on this subsequence uniform convergence result, the present authors previously proved $C^\infty$ exponential decay
in \cite{oh-wang2} when the contact form is nondegenerate.
The proof is based on the so-called \emph{three-interval argument} which is essentially different
from the proofs for exponential convergence in existing literatures, e.g., from those
in \cite{HWZ1, HWZ2, HWZplane} which use the method of differential inequality.

The present paper is a sequel to the paper \cite{oh-wang2} and
the main purpose thereof is to generalize the exponential
convergence result to the Morse-Bott case. In Part \ref{part:exp} of the
current paper, we systematically develop the above mentioned three-interval method as a general framework
and establish the result for Morse-Bott contact forms.
(Corresponding results for pseudo-holomorphic curves in symplectizations
were provided by various authors including \cite{HWZ3, bourgeois, bao}
and we suggest readers to compare our method with theirs.)

In general, the exponential
convergence result is an important ingredient in the set-up of the Fredholm theory
and in the relevant gluing construction. In contact geometry, the moduli spaces of pseudo-holomorphic curves
with noncompact sources are used in defining the contact homology or setting up
the framework of the symplectic field theory (SFT) (see e.g. \cite{SFT} for an introduction).
In this regard, the Morse-Bott case provides important computable examples
in contact geometry and in SFT. (See \cite{bourgeois} for some examples of such computations
based on the Morse-Bott framework of contact homology.)
However, there are various subtleties in
describing the structure of the Morse-Bott moduli spaces and
the corresponding contact homology for the contact forms of Morse-Bott type,
which have not been rigorously set up yet.
One of the purposes of current paper is to provide a careful geometric description
of the locus of closed Reeb orbits and the corresponding tensorial proof of exponential decay results.
Moreover, the abstract framework of the three-interval method we develop in this paper for the exponential decay
proof can be easily applied to other evolution type of equations, and provides a general `black box'
for the exponential decay.

The proof of the exponential decay result consists of two parts, one geometric and the other analytic. Part \ref{part:coordinate} is
devoted to unveil the geometric structure, the \emph{pre-contact structure}, carried by the loci $Q$ of
the closed Reeb orbits of a Morse-Bott contact form $\lambda$ (see Section \ref{subsec:clean} for precise definition).
We prove a canonical neighborhood theorem of any pre-contact manifold which is the contact analogue to
Gotay's on presymplectic manifolds \cite{gotay}, which we call the \emph{contact thickening}
of a pre-contact manifold. By using this neighborhood theorem, we obtain a canonical splitting of
the tangent bundle $TM$ in terms of the pre-contact structure of $Q$ and its thickening. Then we
introduce the class consisting of $J$'s \emph{adapted to $Q$} (refer Section \ref{sec:adapted} for definition)
besides the standard compatibility requirement to $d\lambda$. At last we split the derivative $dw$ of contact instanton $w$
into various components and study them separately.
In this way, we are given the geometric framework which gets us ready to conduct the three-interval method
provided in Part \ref{part:exp}.
Part \ref{part:exp} is then devoted to applying the enhanced version of the three-interval framework
in proving the exponential convergence for the Morse-Bott case, which generalizes the one presented
for the nondegenerate case in \cite{oh-wang2}.

Now we outline the main results in the present paper in more details.
\subsection{Structure of the locus of closed Reeb orbits}
\label{subsec:clean}

Assume $\lambda$ is a fixed contact form of the contact manifold $(M, \xi)$.
For a closed Reeb orbit $\gamma$ of period $T > 0$, one can write $\gamma(t) = \phi^t(\gamma(0))$, where
$\phi^t= \phi^t_{X_\lambda}$ is the flow of the Reeb vector field $X_\lambda$.

Denote by $\Cont(M, \xi)$ the set of all contact one-forms with respect to the contact structure $\xi$,
and by $\CL(M)=C^\infty(S^1, M)$ the space of loops $z: S^1 = \R /\Z \to M$.
Consider the bundle $\mathcal{L}$ over the product
$(0,\infty) \times \CL(M) \times \Cont(M,\xi)$ whose fiber
at $(T, z, \lambda)$ is $C^\infty(z^*TM)$.
The assignment
$$
\Upsilon: (T,z,\lambda) \mapsto \dot z - T \,X_\lambda(z)
$$
defines a section of the bundle, where $(T,z)$ is a pair with a loop $z$ parameterized over
the unit interval $S^1=[0,1]/\sim$ defined by $z(t)=\gamma(Tt)$ for
a Reeb orbit $\gamma$ of period $T$. Notice that $(T, z, \lambda)\in \Upsilon^{-1}(0)
:=\frak{Reeb}(M,\xi)$ if and only if there exists some Reeb orbit $\gamma: \R \to M$ with period $T$, such that
$z(\cdot)=\gamma(T\cdot)$.
Denote by
$$
\frak{Reeb}(M,\lambda) : = \{(T,z) \mid (T,z,\lambda) \in \frak{Reeb}(M,\xi)\}
$$
for each $\lambda \in \Cont(M,\xi)$. From the formula of a $T$-periodic orbit $(T,\gamma)$,
$
T = \int_\gamma \lambda
$.
it follows that the period varies smoothly over $\gamma$.

The general Morse-Bott condition (Bott's notion  \cite{bott} of clean critical submanifold in general) for $\lambda$ corresponds to the statement that
every connected component of $\frak{Reeb}(M,\lambda)$ is a smooth submanifold
of $ (0,\infty) \times \CL(M)$ and its tangent space at
every pair $(T,z) \in \frak{Reeb}(M,\lambda)$ therein coincides with $\ker d_{(T,z)}\Upsilon$.
Denote by $Q$ the locus of closed Reeb orbits contained in a fixed
connected component of $\frak{Reeb}(M,\lambda)$. Throughout this paper, we also call $Q$
a Morse-Bott submanifold when we want to emphasize its manifold structure.

However, when one tries to set up the moduli space of contact instantons
for Morse-Bott contact forms, more requirements are needed and we recall the definition that Bourgeois
adopted in \cite{bourgeois}.
{(Strictly speaking, we also need to
take suitable completions of $\CL(M)$ and $\Cont(M,\xi)$
but we ignore this point which does not play any role in our main discussion.)

\begin{defn}[Equivalent to Definition 1.7 \cite{bourgeois}]\label{defn:morse-bott-intro}
A contact form $\lambda$ is called be of Morse-Bott type, if it satisfies the following:
\begin{enumerate}
\item
Every connected component of $\frak{Reeb}(M,\lambda)$ is a smooth submanifold
of $ (0,\infty) \times \CL(M)$ with its tangent space at
every pair $(T,z) \in \frak{Reeb}(M,\lambda)$ therein coincides with $\ker d_{(T,z)}\Upsilon$;
\item The locus $Q$ is embedded;
\item The 2-form $d\lambda|_Q$ associated to the locus $Q$ of closed Reeb orbits has constant rank.
\end{enumerate}
\end{defn}
Here Condition (1) corresponds to
Bott's notion of Morse-Bott critical manifolds which we name as \emph{standard Morse-Bott type}.
While $\frak{Reeb}(M,\lambda)$ is a
smooth submanifold, the orbit locus $Q \subset M$ of $\frak{Reeb}(M,\lambda)$ is
in general only an immersed submanifold and could have multiple sheets along the locus of multiple orbits.
Therefore we impose Condition (2).
In general, the restriction of the two-form $d\lambda$ to $Q$ has varying rank. It is still not clear whether
the exponential estimates we derive in this paper holds in this general context because our proof
strongly relies on the existence of canonical model of neighborhoods of $Q$. For this reason, we
also impose Condition (3).
We remark that Condition (3) means that the 2-form $d\lambda|_Q$ becomes a presymplectic form.

Depending on the type of the presymplectic form,
we say that $Q$ is of pre-quantization type if the rank of $d\lambda|_Q$ is maximal,
and is of isotropic type if the rank of $d\lambda|_Q$ is zero. The general case is
a mixture of these two. In particular when $\dim M = 3$, such $Q$ must be either of
prequantization type or of isotropic type. This is the case dealt with in
\cite{HWZ3}. The general case considered in \cite{bourgeois}, \cite{behwz} includes the mixed type.

\begin{defn}[Pre-Contact Form] We call one-form $\theta$ on a manifold $Q$ a \emph{pre-contact} form
if $d\theta$ has constant rank, i.e., if $d\theta$ is a presymplectic form.
\end{defn}

While the notion of presymplectic manifolds is well-established in symplectic geometry,
this contact analogue seems to have not been used
in literature, at least formally, as far as we know.

With this terminology introduced, we prove the following theorem.

\begin{thm}[Theorem \ref{thm:morsebottsetup}]\label{thm:morsebottsetup-intro}
Let $\lambda$ be a Morse-Bott type contact form of contact manifold $(M, \xi)$ as defined above.
Let $Q$ be an associated Morse-Bott submanifold of closed Reeb orbits.
Suppose that $Q$ is embedded and $d\lambda|_Q = i_Q^*(d\lambda)$ has constant rank. Then $Q$ carries
\begin{enumerate}
\item a locally free $S^1$-action generated by the Reeb vector field $X_\lambda|_Q$;
\item the pre-contact form $\theta$ given by $\theta = i_Q^*\lambda$ and the splitting
\be\label{eq:kernel-dtheta0}
\ker d\theta = \R\{X_\theta\} \oplus H,
\ee
such that the distribution  $H = \ker d\theta \cap \xi|_Q$ is integrable;
\item
an $S^1$-equivariant symplectic vector bundle $(E,\Omega) \to Q$ with
$$
E = (TQ)^{d\lambda}/\ker d\theta, \quad \Omega = [d\lambda]_E.
$$
\end{enumerate}
\end{thm}

Here we use the fact that there exists a canonical embedding
\be\label{eq:EtoNQM}
E = (TQ)^{d\lambda}/\ker d\theta \hookrightarrow T_QM/ TQ = N_QM,
\ee
and $d\lambda|_{(TQ)^{d\lambda}}$ canonically induces a bilinear form $[d\lambda]_E$
on $E = (TQ)^{d\lambda}/\ker di_Q^*\lambda$ by symplectic reduction.

\begin{defn} Let $(Q,\theta)$ be a pre-contact manifold
equipped with the splitting \eqref{eq:kernel-dtheta0}. We call such a triple $(Q,\theta,H)$
a \emph{Morse-Bott contact set-up}.
\end{defn}

Denote by $\CF$ and $\CN$ the foliations associated to the distribution $\ker d\theta$ and $H$,
respectively. We also denote by $T\CF$, $T\CN$ the associated foliation tangent bundles and
$T^*\CN$ the foliation cotangent bundle of $\CN$.

We prove the following canonical model theorem which describes a
natural way of thickening of Morse-Bott contact set-up $(Q,\theta,H)$ whenever a
symplectic vector bundle $E \to Q$ is given.

\begin{thm}\label{thm:splitting1} Let $(Q,\theta,H)$ be a Morse-Bott contact set-up.
Let a symplectic vector bundle $(E,\Omega) \to Q$ be given. Then
the bundle $F = T^*\CN \oplus E$
carries a canonical contact form $\lambda_{F;G}$ defined as in \eqref{eq:lambdaF}, for
each choice of complement $G$ such that $TQ = T\CF \oplus G$. Furthermore for two such choices of
$G, \, G'$, two induced contact structures are naturally isomorphic.
\end{thm}
Based on this theorem, we denote any such $\lambda_{F;G}$ just by $\lambda_F$
suppressing $G$ from its notation.
This normal form provides a general class of contact
manifolds equipped with a contact form of Morse-Bott type.

Finally we prove the following canonical neighborhood
theorem for $Q \subset M$ with $Q$ defined above for any
Morse-Bott contact form $\lambda$ of contact manifold $(M,\xi)$.

\begin{thm}[Theorem \ref{thm:neighborhoods2}]\label{thm:neighbohood}
Let $Q$ be the submanifold foliated by closed Reeb orbits of
Morse-Bott type contact form $\lambda$ of contact manifold $(M,\xi)$, and
$(Q,\theta)$ and $(E,\Omega)$ be the associated pair defined above.
Let $(F, \lambda_F)$ be the model contact manifold with $F = T^*\CN \oplus E$
and $\lambda_F$ be the contact form on $U_F \subset F$ given in \eqref{eq:lambdaF}.

Then there exist neighborhoods $\CU$ of $Q$ and $U_F$ of the zero section $o_F$
and a diffeomorphism $\psi: U_F \to \CU$ and a function $f: U_F \to \R$ such that
\be\label{eq:psi*lambda}
\psi^*\lambda = f\, \lambda_F, \, f|_{o_F} \equiv 1, \, df|_{o_F}\equiv 0
\ee
and
\be\label{eq:ioFpsi}
i_{o_F}^*\psi^*\lambda = \theta, \quad (\psi^*d\lambda|_{VTF})|_{o_F} = 0\oplus \Omega
\ee
where we use the canonical identification of $VTF|_{o_F} \cong T^*\CN \oplus E$ on the
zero section $o_F \cong Q$.
\end{thm}
\begin{rem}\label{rem:behwz}
We would like to remark that while the bundles $E$ and $TQ/T\CF$ carry canonical fiberwise
symplectic form and so carry canonical orientations
induced by $d\lambda$, the bundle $T\CN$ may not be orientable in general along a
Reeb orbit corresponding to an orbifold point in $P = Q/\sim$.
\end{rem}

\subsection{The three-interval method of exponential estimates}
\label{subsec:three-interval}

For the study of the asymptotic behavior of finite $\pi$-energy solutions
of contact instanton $w: \dot \Sigma \to M$ near a Morse-Bott submanifold $Q$,
we introduce the following class of $CR$-almost complex structures.

\begin{defn}[Definition \ref{defn:adapted}] Let $Q \subset M$ be a
Morse-Bott submanifold foliated by closed Reeb orbits of $\lambda$.
Suppose $J$ defines a contact triad $(M,\lambda,J)$.
We say a $CR$-almost complex structure $J$ for $(M,\xi)$ is adapted to
the submanifold $Q$ or simply is $Q$-adapted if $J$ satisfies
\be\label{eq:JTNT}
J(TQ) \subset TQ + JT\CN.
\ee
\end{defn}

Note that this condition is vacuous for the nondegenerate case,
but for the general Morse-Bott case, the class of
adapted $J$ is strictly smaller than the one of general $CR$-almost complex structures of
the triad. The set of $Q$-adapted $J$'s is contractible and the proof is given in Appendix \ref{sec:appendix-adapted}.
 As far as the applications to contact topology are concerned, requiring this
condition is not any restriction but seems to be necessary for the analysis of
contact-instanton maps or of the pseudoholomorphic maps in the symplectization
(or in the symplectic manifolds with contact-type boundary.)

Let $w: \dot\Sigma \rightarrow M$ be a contact instanton map, i.e.,
satisfying \eqref{eq:contact-instanton} at a cylindrical end
$[0, \infty)\times S^1$, which now can be written as
\be\label{eq:contact-instanton2}
\pi \frac{\del w}{\del \tau} + J \pi \frac{\del w}{\del t} = 0, \quad
d(w^*\lambda \circ j) = 0,
\ee
for $(\tau, t)\in [0, \infty)\times S^1$. We put the following basic hypotheses
for the study of exponential convergence.
\begin{hypo}\label{hypo:basic-intro}[Hypothesis \ref{hypo:basic}]
\begin{enumerate}
\item \emph{Finite $\pi$-energy}:
 $E^\pi(w): = \frac{1}{2} \int_{[0, \infty)\times S^1} |d^\pi w|^2 < \infty$;
 \item \emph{Finite derivative bound}:
 $\|dw\|_{C^0([0, \infty)\times S^1)} \leq C < \infty$;
\item \emph{Non-vanishing asymptotic action}:\\
$
\CT :=  \frac{1}{2}\int_{[0,\infty) \times S^1} |d^\pi w|^2
+ \int_{\{0\}\times S^1}(w|_{\{0\}\times S^1})^*\lambda\neq 0
$;
\item \emph{Vanishing asymptotic charge}:
$
\CQ:=\int_{\{0\}\times S^1}((w|_{\{0\}\times S^1})^*\lambda\circ j)=0$.
\end{enumerate}
\end{hypo}
Under these hypotheses, we establish the following $C^\infty$ uniform exponential convergence of
$w$ to a closed Reeb orbit $z$ of period $T=\CT$. This result was already known in
\cite{HWZ3, bourgeois,bao} in the context of pseudo-holomorphic curves $u = (w,a)$ in symplectizations.
However we emphasize that our proof presented here, which uses the three-interval framework, is
 different from the ones \cite{HWZ3,bourgeois,bao} even in the symplectization case.
Furthermore when we deal with the case of symplectization,
our method completely separates the estimates of $w$ from that of $a$'s.
(See Section \ref{sec:asymp-cylinder}.)

\begin{thm}\label{thm:expdecay} Assume $(M, \lambda)$ is a Morse-Bott contact manifold and
$w$ is a contact instanton satisfying the Hypothesis \ref{hypo:basic-intro}
at the given end. Then there exists a closed Reeb orbit $z$ with period $T=\CT$ and positive
constant $\delta$ determined by $z$, such that
$$\|d(w(\tau, \cdot), z(T\cdot))\|_{C^0(S^1)}<C e^{-\delta \tau},$$
and
\beastar
&&\left\|\pi \frac{\del w}{\del\tau}(\tau, \cdot)\right\|_{C^0(S^1)}<Ce^{-\delta\tau}, \quad
\left\|\pi \frac{\del w}{\del t}(\tau, \cdot)\right\|_{C^0(S^1)}<Ce^{-\delta\tau}\\
&&\left\|\lambda(\frac{\del w}{\del\tau})(\tau, \cdot)\right\|_{C^0(S^1)}<Ce^{-\delta\tau}, \quad
\left\|\lambda(\frac{\del w}{\del t})(\tau, \cdot)-T\right\|_{C^0(S^1)}<Ce^{-\delta\tau}\\
&&\left\|\nabla^l dw(\tau, t)\right\|_{C^0(S^1)}<C_le^{-\delta\tau} \quad \text{for any}\quad l\geq 1,
\eeastar
where $d$ is the distance function induced from the triad metric on $M$ and $C, C_{l}$ are positive constants which only depend on $l$.
\end{thm}

Now comes the outline of the strategy of our proof of exponential convergence in the present paper.
Mundet i Riera and Tian in \cite{mundet-tian} elegantly used a discrete method of three-interval arguments
in proving exponential decay under the assumption of
$C^0$-convergence already established. However, for most cases of interests, the $C^0$-convergence
is not a priori given in the beginning but it is often the case that the $C^0$-convergence
can be obtained only after one proves the exponential convergence of derivatives.
(See the proofs of, for example,  \cite{HWZ1, HWZ2, HWZplane}, \cite{HWZ3}, \cite{bourgeois}, \cite{bao}).
To obtain the exponential estimates of derivatives, researchers conduct some bruit-force calculations in deriving the needed
differential inequality, and then proceed from there towards the final result.
Such calculation, especially in coordinates, becomes quite complicated for the Morse-Bott situation
and hides the geometry that explains why such a differential inequality could be expected.

Our proof is divided into two parts by writing $w=(u, s)$  in the
normalized contact triad $(U_F,\lambda_F, J_0)$ (see Definition \ref{defn:normaltriad}) with $U_F \subset F \to Q$
for any given compatible $J$ adapted to $Q$, where $J_0$ is canonical normalized
$CR$-almost complex structure associated to $J$.
We also decompose $s=(\mu, e)$  in terms of the splitting $F = T^*\CN \oplus E$.
In this decomposition, the $L^2$-exponential estimates for the $e$-component is an easy consequence of
the three-interval method which we formulate above in a general abstract framework
(see Theorem \ref{thm:three-interval} for the precise statement). This estimate belongs to the
standard realm of exponential decay proof for the asymptotically cylindrical elliptic equations.

However the study of $L^2$-exponential estimates for $(u,\mu)$ does not directly belong to this
standard realm. Although we still apply similar three-interval method for the study of $L^2$-exponential convergence,
its study is much more subtle than that of the normal component due to the presence of
non-trivial kernel of the asymptotic operator $B_\infty$ of the linearization.
To handle the $(u,\mu)$-component, we formulate the following general theorem
from the abstract framework of the three-interval argument, and
refer readers to Section \ref{sec:three-interval}, \ref{subsec:exp-horizontal} for the precise statement and
its proof.

\begin{thm}
Assume $\xi(\tau, t)$ is a section of some vector bundle on $\R\times S^1$ which satisfies the equation
$$
\nabla^\pi_\tau\zeta+J\nabla^\pi_t\zeta+S\zeta=L(\tau, t) \quad \text{ with } |L|<Ce^{-\delta_0\tau}
$$
of Cauchy-Riemann type (or more generally any elliptic PDE of evolution type), where $S$ is a bounded symmetric operator.

Suppose that there exists a sequence $\{\bar\zeta_k\}$ (e.g., by performing a suitable
rescaling of $\zeta$) such that at least one subsequence converges to a non-zero
section $\bar\zeta_\infty$ of a (trivial) Banach bundle over a fixed finite interval, say on $[0,3]$,
that satisfies the ODE
$$
\frac{D \bar\zeta_\infty}{d\tau}+B_\infty \bar\zeta_\infty =0
$$
on the associated Banach space.

Then provided $\|\zeta(\tau, \cdot)\|_{L^2(S^1)}$ converges to zero as
$\tau$ goes to $\infty$, $\|\zeta(\tau, \cdot)\|_{L^2(S^1)}$ decays
exponentially fast with the rate $\delta >0 $ for any constant $\delta < \min\{\lambda_0,\delta_0\}$
where $\lambda_0$ is the smallest absolute value of non-zero eigenvalues of $B_\infty$.
\end{thm}

\begin{rem}For the special case when $B_\infty$ has only trivial kernel,
the result can be regarded as the discrete analogue of the differential inequality method
used by Robbin-Salamon in \cite{robbin-salamon}.
\end{rem}

In this framework, our exponential convergence proof is based on intrinsic geometric tensor calculations
which is coordinate-free. As a result, our proof make it manifest that (roughly) the exponential
decay occurs whenever the geometric PDE has bounded energy at cylindrical ends
and the limiting equation is of linear evolution type
$
\frac{\del \overline \zeta_\infty}{\del\tau}+B_\infty \overline \zeta_\infty =0
$,
where $B_\infty$ is an elliptic operator with discrete spectrum.
If $B_\infty$ has trivial kernel, the conclusion follows rather immediately from
the three-interval argument. Even when $B_\infty$ contains non-trivial kernel,
the exponential decay would still follow as long as some geometric condition, like the Morse-Bott assumption in
the current case of our interest, enables one to extract
some non-vanishing solution of the limit equation
$\frac{\del \overline\zeta_\infty}{\del\tau}+B_\infty \overline \zeta_\infty =0$
that arises in the course of three-interval arguments.
Moreover the decay rate $\delta> 0$ is always provided by the minimal eigenvalue of $B_\infty$.

Now we roughly explain how the non-vanishing limiting solution mentioned above is obtained in the current situation:
First, the canonical neighborhood provided in Part \ref{part:coordinate} is used to split the contact
instanton equations into the vertical and horizontal parts. By this way, only the horizontal equation could be
involved with the kernel of $B_\infty$ which by the Morse-Bott condition has
nice geometric structure in the sense that the kernel can be excluded by looking
higher derivative instead of the map itself.
Then, to further see the limit of the derivative is indeed non-vanishing, we apply the
geometric decomposition the derivative and study the center of mass on the Morse-Bott
submanifold $Q$. The details are presented in Section \ref{subsec:exp-horizontal}
and Section \ref{subsec:centerofmass}.

\part{Contact Hamiltonian geometry and canonical neighborhoods}\label{part:neighborhoods}
\label{part:coordinate}

The main purpose of this part is to prove a canonical neighborhood theorem for the loci of closed Reeb orbits when the contact form $\lambda$ of a contact manifold $(M, \xi)$ is of Morse-Bott type. The results of this part
provides geometric preparation for the study of asymptotic exponential convergence
of contact instanton at a puncture of the domain Riemann surface.

The outline of Part 1 in section wise is as follows.
\begin{itemize}
\item In Section 2, we review some basic facts related to contact forms of a contact manifold. We first set up a natural isomorphism between $TM$ and $T^*M$
using the contact form $\lambda$.
This is a contact analogue of the isomorphism between tangent bundle and
cotangent bundle for  symplectic manifolds.
Then we derive explicit formulae of the Reeb vector field $R_{f\lambda}$ and
the contact projection $\pi_{f\lambda}$ in terms of $X_\lambda$, $\pi_\lambda$ and $f$
respectively.
\item In Section 3, we introduce the definition of Morse-Bott contact forms.
Then we study the canonical pre-contact structure associated to the loci of closed Reeb orbits under Morse-Bott assumption.
\item In section 4, we introduce the notion of contact thickening of pre-contact structure. It is the contact analogue
of the symplectic thickening for pre-symplectic structure constructed in \cite{gotay}, \cite{oh-park}.
\item In Section 5, we prove a canonical neighborhood theorem of the loci of
closed  Reeb orbits $Q$ under the Morse-Bott assumption.
\item  In Section 6, we derive the linearization formula of a Reeb orbit in the normal form.
\item In Section 7, we express the derivative $dw = (du, \nabla_{du} f)$ of
any smooth map $w = (u,f)$ from a (punctured) surface into the normal neighborhood $F$ of $Q$
in terms of the splitting
$$
TU_F = TQ \oplus F = TQ \oplus (E \oplus JT\CF), \quad TQ = T\CF \oplus G.
$$
\item In Section 8 and Section 9,  we introduce the class of adapted CR-almost complex structure and prove its abundance.

\end{itemize}

\section{Basics on contact forms}

We recall some basic facts on the contact geometry and
contact Hamiltonian dynamics especially in relation to the perturbation of
contact forms for a given contact manifold $(M,\xi)$.

\subsection{$\lambda$-dual vector fields and $\lambda$-dual one-forms}
\label{subsec:some}

Let $(M,\xi)$ be a contact manifold and $\lambda$ be a contact form with $\ker \lambda = \xi$. Consider
its associated decomposition
\be\label{eq:decomp-TM}
TM = \R\{X_\lambda\} \oplus \xi
\ee
and denote by $\pi=\pi_\lambda: TM \to \xi$ the associated projection.
This decomposition canonically
induces the corresponding dual decomposition
\be\label{eq:decomp-T*M}
T^*M = \xi^\perp \oplus (\R\{X_\lambda\})^\perp
\ee
where $(\cdot)^\perp$ is the annihilator of $(\cdot)$. This
gives rise to a decomposition
\be\label{eq:alpha-decomp}
\alpha = \alpha(X_\lambda)\, \lambda + \alpha \circ \pi_\lambda.
\ee

Then we have the following general lemma whose proof immediately follows from \eqref{eq:decomp-T*M}.
\begin{lem}\label{lem:decompose}
For any given one-form $\alpha$, there exists a unique $Y_\alpha \in \xi$ such that
$$
\alpha = Y_\alpha \rfloor d\lambda + \alpha(X_\lambda) \lambda.
$$
\end{lem}

\begin{defn}[$\lambda$-Dual Vector Field and One-Form] Let $\lambda$ be a given contact form of $(M,\xi)$.
We define the \emph{$\lambda$-dual vector field} of a one-form $\alpha$ to be
$$
\flat_\lambda(\alpha): = Y_\alpha + \alpha(X_\lambda)\, X_\lambda.
$$
Conversely for any given vector field $X$ we define its $\lambda$-dual one-form by
$$
\sharp_\lambda(X)= X \rfloor d\lambda + \lambda(X)\, \lambda.
$$
\end{defn}
For the simplicity of notation, we will denote $\alpha_X: = \sharp_\lambda(X)$.
By definition, we have the identity
\be\label{eq:obvious-id}
\lambda(\flat_\lambda(\alpha)) = \alpha(X_\lambda).
\ee
The following proposition is immediate from the definitions of the dual vector field and the dual
one-forms.

\begin{prop}\label{prop:inverse} The map $\flat_\lambda: \Omega^1(M) \to \frak X(M), \, \alpha \mapsto \flat_\lambda(\alpha)$
and the map $\sharp^\lambda: \frak X(M) \to \Omega^1(M), \, X \mapsto \alpha_X$ are inverse to each other.
In particular any vector field  can be written as $\flat_\lambda(\alpha)$ for a unique one-form $\alpha$ and
any one-form can be written as $\alpha_X$ for a unique vector field $X$.
\end{prop}

By definition, we have $\flat_\lambda(\lambda) = X_\lambda$  which
corresponds to the $\lambda$-dual to the contact form $\lambda$ itself for which $Y_\lambda =0$
by definition.
Obviously when an exact one-form $\alpha$ is given, the choice of $h$ with $\alpha = dh$ is
unique modulo addition by a constant (on each connected component of $M$).

To equip readers with some feeling on the above decomposition which is
not common in the literature, we now provide the coordinate expression of $\flat_\lambda(\alpha)$ and $\alpha_X$
in the Darboux chart $(q_1,\cdots, q_n, p_1, \cdots, p_n,z)$ with respect to the canonical
one-form $\lambda_0 = dz - \sum_{i=1}^n p_i\, dq_i$ on $\R^{2n+1}$ or more generally on the
one-jet bundle $J^1(N)$ of a smooth $n$-manifold $N$.
\emph{However this coordinate expression will not be used in the rest of the present paper.}

We recall that for this contact form,
the associated Reeb vector field is nothing but
$$
X_{\lambda_0} = \frac{\del}{\del z}.
$$
We start with the expression of $\flat_\lambda(\alpha)$ for a given one-form
$$
\alpha = \alpha_0 \, dz + \sum_{i=1}^n a_i\, dq_i + \sum_{j=1}^n b_j\, dp_j.
$$
We denote
$$
\flat_\lambda(\alpha) = v_0 \frac{\del}{\del z} + \sum_{i=1}^n v_{i;q} \frac{\del}{\del q_i}
+ \sum_{j=1}^n v_{j;p} \frac{\del}{\del p_j}.
$$
A direct computation using the defining equation of $\flat_\lambda(\alpha)$ leads to
\begin{prop}\label{prop:expressioninDarbouxchart}
Consider the standard contact structure $\lambda = dz - \sum_{i=1}^np_i\, dq_i$
on $\R^{2n+1}$.  Then for the given one-form $\alpha = \alpha_0 dz + \sum_{i=1}^na_i\, dq_i + \sum_{j=1}^nb_j \, dp_j$,
\be\label{eq:vis}
\flat_\lambda(\alpha) = \left(\alpha_0 + \sum_{k=1}^n p_k\, b_k\right) \frac{\del}{\del z} +\sum_{i=1}^n b_i\frac{\del}{\del q_i}
+ \sum_{j=1}^n (-a_j - p_j\, \alpha_0) \frac{\del}{\del p_j}.
\ee
Conversely, for given $X =  v_0 \frac{\del}{\del z} + \sum_{i=1}^n v_{i;q} \frac{\del}{\del q_i} + \sum_{j=1}^n v_{j;p} \frac{\del}{\del p_j}$, we obtain
\bea\label{eq:alphais}
\alpha_X & = & \left(v_0 - \sum_{j=1}^n p_j\, v_{j;q}\right)\,dz \nonumber \\
&{}& \quad - \sum_{i=1}^n \left(  v_{i;p} + \left((v_0 - \sum_{j=1}^n p_j\, v_{j;q})p_i\right)\right)  dq_i
+\sum_{j=1}^n v_{j;q}\, dp_j.
\eea
\end{prop}
\begin{proof} Here we first recall the basic identity \eqref{eq:obvious-id}.

By definition, $\flat_\lambda(\alpha)$ is determined by the equation
\be\label{eq:alpha-lambda0}
\alpha = \flat_\lambda(\alpha) \rfloor \sum_{i=1}^n dq_i \wedge dp_i
+ \lambda(\flat_\lambda(\alpha))\, \left(dz - \sum_{i=1}^n p_i\, dq_i\right)
\ee
in the current case.
A straightforward computation leads to the formula \eqref{eq:vis}.
Then \eqref{eq:alphais} can be derived either by inverting this formula or by
using the defining equation of $\alpha_X$, which is further reduced to
$$
\alpha_X = X \rfloor d\lambda + \lambda(X)\, \lambda = X \rfloor \sum_{i=1}^n dq_i \wedge dp_i
+ (dz - \sum_{j=1}^n p_j\, dq_j)(X)\, \lambda.
$$
We omit the details of the computation.
\end{proof}

\begin{exm}
Again consider the canonical one-form $dz - \sum_{i=1}^n p_i\, dq_i$. Then we obtain the
following coordinate expression as a special case of \eqref{eq:vis}
\be\label{eq:hamvis}
\flat_\lambda(dh) = \left(\frac{\del h}{\del z} + \sum_{i=1}^np_i\frac{\del h}{\del p_i}\right)\frac{\del}{\del z}
+ \sum_{i=1}^n \frac{\del h}{\del p_i}\frac{\del}{\del q_i}
+ \sum_{i=1}^n\left(- \frac{\del h}{\del q_i} - p_i\frac{\del h}{\del z}\right)\frac{\del}{\del p_i}.
\ee
\end{exm}

\subsection{Perturbation of contact forms of $(M,\xi)$}
\label{subsec:perturbed-forms}

In this section, we exploit the discussion on the $\lambda$-dual vector fields and
express the Reeb vector fields $X_{f\lambda}$ and the projection $\pi_{f\lambda}$
associated to the contact form $f \lambda$ for a positive function $f > 0$,
in terms of those associated to the given contact form $\lambda$ and
the $\lambda$-dual vector fields of $df$.

Recalling Lemma \ref{lem:decompose}, we can write
$$
dg = Y_{dg} \rfloor d\lambda + dg(X_\lambda) \lambda
$$
with $Y_{dg} \in \xi$ in a unique way for any smooth function $g$. Then by definition, we have
$Y_{dg} = \pi_\lambda(\flat_\lambda(dg))$.

We first compute the following useful explicit formula
for the associated Reeb vector fields $X_{f\lambda}$ in terms of $X_\lambda$
and $Y_{dg}$.

\begin{prop}[Perturbed Reeb Vector Field]\label{prop:eta} Denote $Y_{dg}  = \pi_\lambda(\flat_\lambda(dg))$
as above. Then we have
$$
X_{f\lambda} = \frac{1}{f}(X_{\lambda} + Y_{dg}), \quad g = \log f.
$$
\end{prop}
\begin{proof}
It turns out to be easier to
consider $f\, X_{f\lambda}$ which we compute below. First we have
\be\label{eq:fXflambda}
f\,X_{f\lambda}= c \cdot X_\lambda+\eta
\ee
with respect to the splitting $TM = \R\{X_\lambda\} \oplus \xi$ for some constant $c$
and $\eta \in \xi$. We evaluate
\beastar
c = \lambda(fX_{f\lambda})=(f\lambda)(X_{f\lambda})= 1.
\eeastar
It remains to derive the formula for $\eta$. Using the formula
\beastar
d(f\lambda)= f d\lambda+df\wedge \lambda
\eeastar
and $\lambda(\eta) = 0$,
we compute
\beastar
\eta\rfloor d\lambda&=&(fX_{f\lambda})\rfloor d\lambda\\
&=& X_{f\lambda}\rfloor d(f\lambda)-X_{f\lambda}\rfloor(df\wedge\lambda)\\
&=&-X_{f\lambda}\rfloor(df\wedge\lambda)\\
&=&-X_{f\lambda}(f)\lambda+\lambda(X_{f\lambda})df\\
&=&-\frac{1}{f}(X_\lambda+\eta)(f)\lambda+\frac{1}{f}\lambda(X_\lambda+\eta)df\\
&=&-\frac{1}{f}X_\lambda(f)\lambda-\frac{1}{f}\eta(f)\lambda+\frac{1}{f}df.
\eeastar
Take value of $X_\lambda$ for both sides, we get
$
\eta(f)=0,
$
and hence
$$
\eta\rfloor d\lambda=-\frac{1}{f}X_\lambda(f)\lambda+\frac{1}{f}df.
$$
Setting $g =\log f$, we can rewrite this into
$$
\eta\rfloor d\lambda=-dg(X_\lambda)\lambda+dg.
$$
In other words, we obtain
\be
dg = \eta \rfloor d\lambda + dg(X_\lambda)\lambda.\label{eq:eta1}
\ee
Therefore by Lemma \ref{lem:decompose}, we have obtained
$\eta = Y_{dg}$. Substituting this into \eqref{eq:fXflambda} and dividing it
by $f$, we have finished the proof.
\end{proof}

Next we compare the contact projections
$\pi_{\lambda}$ with $\pi_{f\lambda}$ associated to $\lambda$ and $f\lambda$
respectively.

\begin{prop}[Perturbed Contact Projection]\label{prop:xi-projection}
Let $(M,\xi)$ be a contact manifold and let $\lambda$ be a contact form
i.e., $\ker \lambda = \xi$. Let $f$ be a positive smooth function and $f \, \lambda$
be its associated contact form. Denote by $\pi_{\lambda}$ and $\pi_{f\, \lambda}$ be
their associated $\xi$-projection. Then
\be\label{eq:xi-projection}
\pi_{f\, \lambda}(Z) = \pi_{\lambda}(Z)- \lambda(Z) Y_{dg}
\ee
for the function $g = \log f$.
\end{prop}
\begin{proof}
We compute
\beastar
\pi_{f\lambda}(Z)&=& Z - f\lambda(Z) X_{f\lambda}
= Z-\lambda(Z)(f X_{f\lambda})\\
&=&Z-\lambda(Z) X_{\lambda}+(\lambda(Z) X_{\lambda}-\lambda(Z)(fX_{f\lambda}))\\
&=&\pi_{\lambda}Z+\lambda(Z)(X_{\lambda}-f X_{f\lambda})\\
&=&\pi_{\lambda}Z-\lambda(Z) Y_{dg}.
\eeastar
This finishes the proof.
\end{proof}

\subsection{Linearization formula for the perturbed contact form}

We next study the relationship between the linearization of
$\Upsilon_\lambda (z) = \dot z - T X_{f \lambda}(z)$ which we denote by
$$
D^\pi \Upsilon(z)(Z) = \nabla_t^\pi Z - T\nabla_Z X_{f\, \lambda}
$$
with respect to the triad connection of $(M,\lambda,J)$ (see Proposition 7.6 in \cite{oh-wang2})
for a given function $f$. Substituting
$$
X_{f\, \lambda}  = \frac{1}{f}(X_{\lambda} + Y_{dg})
$$
into this formula, we derive

\begin{lem}[Linearization]\label{lem:DUpsilon}
Let $\nabla$ be the triad connection of $(M, f \lambda,J)$. Then
for any vector field $Z$ along a Reeb orbit $z = (\gamma(T\cdot),o_{\gamma(T\cdot)})$,
\be\label{eq:DUpsilon}
D^\pi \Upsilon(z)(Z)
=  \nabla_t^\pi Z - T\left(\frac{1}{f}\nabla_Z X_\lambda + Z[1/f] X_\lambda \right)
- T\left(\frac{1}{f} \nabla_Z Y_{dg} + Z[1/f] Y_{dg}\right)
\ee
\end{lem}
\begin{proof} Let $\nabla$ be the triad connection of $(M,f\lambda,J)$. Then
by definition its torsion $T$ satisfies the axiom $T(X_\lambda,Z) = 0$ for any vector field $Z$ on $M$
(see Theorem 1 in \cite{oh-wang1}). Using this property, as in section 7 of \cite{oh-wang2},
we compute
\beastar
D^\pi \Upsilon(z)(Z) & = &  \nabla_t^\pi Z - T\nabla_Z X_{f\lambda} \\
& = & \nabla_t^\pi Z - T\nabla_Z\left(\frac{1}{f}(X_\lambda + Y_{dg})\right) \\
& = & \nabla_t^\pi Z -TZ[1/f](X_\lambda + Y_{dg}) - \frac{T}{f}\nabla_Z(X_\lambda + Y_{dg})\\
& = & \nabla_t^\pi Z - T\left(\frac{1}{f}\nabla_Z X_\lambda + Z[1/f] X_\lambda \right)
- T\left(\frac{1}{f} \nabla_Z Y_{dg} + Z[1/f] Y_{dg}\right).
\eeastar
which finishes the proof.
\end{proof}

We note that when $f\equiv 1$, the above formula reduces to the standard formula
$$
D^\pi \Upsilon(z)(Z)
=  \nabla_t^\pi Z - T\nabla_Z X_\lambda
$$
which is  further reduced to
$$
D^\pi \Upsilon(z)(Z) = \nabla_t^\pi Z - \frac{T}{2}(\CL_{X_\lambda}J) J Z
$$
for any contact triad $(M,\lambda,J)$.
(See section 7 \cite{oh-wang2} for some discussion on this formula.)

\section{The locus foliated by closed Reeb orbits}
\label{sec:neighbor}

\subsection{Definition of Morse-Bott contact form}
\label{subsec:morse-bott}

Let $(M,\xi)$ be a contact manifold and $\lambda$ be a contact form of $\xi$.
We would like to study the linearization of the equation $\dot x  = X_\lambda(x)$
along a closed Reeb orbit.
Let $\gamma$ be a closed Reeb orbit of period $T > 0$. In other words,
$\gamma: \R \to M$ is a solution of $\dot \gamma = X_\lambda(\gamma)$ satisfying $\gamma(t+T) = \gamma(t)$.

By definition, we can write $\gamma(t) = \phi^t_{X_\lambda}(\gamma(0))$
for the Reeb flow $\phi^t= \phi^t_{X_\lambda}$ of the Reeb vector field $X_\lambda$.
In particular $p = \gamma(0)$ is a fixed point of the diffeomorphism
$\phi^T$ when $\gamma$ is a closed Reeb orbit of period $T$.
Since $\CL_{X_\lambda}\lambda = 0$, the contact diffeomorphism $\phi^T$ canonically induces the isomorphism
$$
\Psi_{z;p} : = d\phi^T(p)|_{\xi_p}: \xi_p \to \xi_p
$$
which is the linearized Poincar\'e return map $\phi^T$ restricted to $\xi_p$ via the splitting
$$
T_p M=\xi_p\oplus \R\cdot \{X_\lambda(p)\}.
$$

\begin{defn} Let $\gamma$ be a closed Reeb orbit with period $T> 0$ and
denote by $z:S^1 \to M$ the map defined by $z(t) = \gamma(Tt)$.
We say a $T$-closed Reeb orbit $(T,z)$ is \emph{nondegenerate}
if the linearized return map $\Psi_{z;p}:\xi_p \to \xi_p$ with $p = \gamma(0)$ has no eigenvalue 1.
\end{defn}

Denote $\Cont(M, \xi)$ the set of contact one-forms with respect to the contact structure $\xi$
and $\CL(M)=C^\infty(S^1, M)$ the space of loops $z: S^1 = \R /\Z \to M$.
 We would like to consider the bundle
$\mathcal{L}$ over the product $(0,\infty) \times \CL(M) \times \Cont(M,\xi)$ whose fiber
at $(T, z, \lambda)$ is given by the space $C^\infty(z^*TM)$ of sections of the bundle $z^*TM \to S^1$.
We consider the assignment
$$
\Upsilon: (T,z,\lambda) \mapsto \dot z - T \,X_\lambda(z)
$$
which is a section. Then $(T, z, \lambda)\in \Upsilon^{-1}(0)
:=\mathfrak{Reeb}(M,\xi)$ if and only if there exists some closed Reeb orbit $\gamma: \R \to M$ with period $T$, such that
$z(\cdot)=\gamma(T\cdot)$.

We first start with the standard notion, applied to the set of
closed Reeb orbits, of Morse-Bott critical manifolds introduced by Bott in \cite{bott}:

\begin{defn} We call a contact form $\lambda$ \emph{standard Morse-Bott type} if
every connected component of $\frak{Reeb}(M,\lambda)$ is a smooth submanifold
of $ (0,\infty) \times \CL(M)$ with its tangent space at
every pair $(T,z) \in \frak{Reeb}(M,\lambda)$ therein coincides with $\ker d_{(T,z)}\Upsilon$.
\end{defn}

The following is an immediate consequence of this definition.

\begin{lem}\label{lem:T-constant} Suppose $\lambda$ is of standard Morse-Bott type, then
on each connected component of $\mathfrak{Reeb}(M,\lambda)$, the period remains constant.
\end{lem}
\begin{proof}
Let $(T_0, z_0)$ and $(T_1, z_1)$ be two elements in the same
connected component of $\mathfrak{Reeb}(M,\lambda)$.
We connect them by a smooth one-parameter family
$(T_s, z_s)$ for $0 \leq s \leq 1$. Since $\dot{z_s}=T_sX_\lambda(z_s)$ and then
$T_s=\int_{S^1}z^*_s\lambda$, it is enough to prove
$$
\frac{d}{ds} T_s =\frac{d}{ds} \int_{S^1}z^*_s\lambda\equiv 0.
$$
We compute
\beastar
\frac{d}{ds} z^*_s\lambda &=& z_s^*(d (z_s' \rfloor \lambda) +z_s' \rfloor d\lambda)\\
&=&d(z_s^*(z_s' \rfloor \lambda))+z_s^*(z_s' \rfloor d\lambda)\\
&=&d(z_s^*(z_s' \rfloor \lambda)).
\eeastar
Here we use $z_s'$ to denote the derivative with respect to $s$.
The last equality comes from the fact that $\dot{z_s}$ is parallel to $X_\lambda$.
Therefore we obtain by Stokes formula that
$$
\frac{d}{ds} \int_{S^1}z^*_s\lambda = \int_{S^1} d(z_s^*(z_s' \rfloor \lambda)) = 0
$$
and finish the proof.
\end{proof}
Now we prove
\begin{lem}\label{lem:localfree} Let $\lambda$ be standard Morse-Bott type.
Fix a connected component $\CR \subset \frak{Reeb}(M,\lambda)$
and denote by $Q \subset M$ the locus of the corresponding closed Reeb orbits. Then $Q$ is a smooth immersed submanifold
which carries a natural locally free $S^1$-action induced by the Reeb flow over one period.
\end{lem}
\begin{proof} Consider the evaluation map $ev_{\CR}: \mathfrak{Reeb}(M,\lambda) \to M$ defined by
$ev_{\CR}(T,z) = z(0)$.  It is easy to prove that the map is a local immersion and so $Q$ is an
immersed submanifold. Since the closed Reeb orbits have constant period $T>0$ by Lemma \ref{lem:T-constant},
the action is obviously locally free and so $ev_{\CR}$ is an immersion
and so $Q$ is immersed in $M$. This finishes the proof.
\end{proof}

However $Q$ may not be embedded in general along the locus of multiple orbits.

Partially following \cite{bourgeois}, from now on in the rest of the paper,
\emph{we always assume $Q$ is embedded and compact}.
Denote $\omega_Q : = i_Q^*d\lambda$ and
$$
\ker \omega_Q =\{e\in TQ \mid \omega(e, e')= 0 \quad \text{for any} \quad e'\in TQ\}.
$$
We warn readers that the Morse-Bott condition does not imply that the form $\omega_Q$
has constant rank, and hence the dimension of this kernel may vary pointwise on $Q$.
However if it does, $\ker \omega_Q$ defines an integrable distribution and so
defines a foliation, denoted by $\CF$, on $Q$.
Since $Q$ is also foliated by closed Reeb orbits and $\CL_{X_\lambda}d\lambda=0$,
it follows that $\CL_{X_{\lambda}}\omega_Q=0$ when we restrict everything on $Q$. Therefore
each leaf of the foliation consisting of closed Reeb orbits. Motivated by this, we also
impose the condition that \emph{the two-form $\omega_Q$ has constant rank}.

\begin{defn}[Compare with Definition 1.7 \cite{bourgeois}]
\label{defn:morse-bott}
We say that the contact form $\lambda$ is of Morse-Bott type if it satisfies the following:
\begin{enumerate}
\item
Every connected component of $\frak{Reeb}(M,\lambda)$ is a smooth submanifold
of $ (0,\infty) \times \CL(M)$ with its tangent space at
every pair $(T,z) \in \frak{Reeb}(M,\lambda)$ therein coincides with $\ker d_{(T,z)}\Upsilon$.
\item $Q$ is embedded.
\item $\omega_Q$ has constant rank on $Q$.
\end{enumerate}
\end{defn}

\subsection{Structure of the locus of closed Reeb orbits}
\label{subsec:structure}

Let $\lambda$ be a Morse-Bott contact form of $(M,\xi)$ and $X_\lambda$
its Reeb vector field. Let $Q$ be as in Definition \ref{defn:morse-bott}. In general, $Q$
carries a natural locally free $S^1$-action induced by the Reeb flow $\phi_{X_\lambda}^T$
(see Lemma \ref{lem:localfree}).
Then by the general theory of compact Lie group actions (see \cite{helgason} for example), the action has a finite
number of orbit types which have their minimal periods, $T/m$ for some integer $m \geq 1$.
The set of orbit spaces $Q/S^1$ carries natural
orbifold structure at each multiple orbit with its isotropy group $\Z/m$ for some $m$.

\begin{rem}\label{rem:noneffective}
Here we would like to mention that the $S^1$-action induced by
$\phi_{X_\lambda}^T$ on $Q$ may not be effective:
It is possible that the connected component $\frak{R}$ of $\frak{Reeb}(M,\lambda)$
can consist entirely of multiple orbits.
\end{rem}

Now we fix a connected component of $Q$ and just denote it by $Q$ itself.
Denote $\theta = i_Q^*\lambda$.
We note that the two-form $\omega_Q = d\theta$ is assumed to
have constant rank on $Q$ by the definition of Morse-Bott contact form in Definition \ref{defn:morse-bott}.

The following is an immediate consequence of the definition but exhibits a
particularity of the null foliation of the presymplectic manifold $(Q,\omega_Q)$
arising from the locus of closed Reeb orbits. We note that $\ker \omega_Q$ carries a natural splitting
$$
\ker \omega_Q = \R\{X_\lambda\} \oplus (\ker \omega_Q \cap \xi|_Q).
$$
\begin{lem}\label{lem:integrableCN}
The distribution $(\ker \omega_Q) \cap \xi|_Q$
on $Q$ is integrable.
\end{lem}
\begin{proof} Let $X,\, Y$ be vector fields on $Q$ such that $X,\, Y \in (\ker \omega_Q) \cap \xi|_Q$.
Then $[X, Y]\in \ker \omega_Q$ since $\omega_Q$ is a closed two-form whose
null distribution is integrable.
At the same time, we compute
$i_Q^*\lambda([X,Y]) = X[\lambda(Y)] - Y[\lambda(X)] -\omega_Q(X,Y) = 0$
where the first two terms vanish since $X, \, Y \in \xi$ and the third vanishes because
$X \in \ker \omega_Q$. This proves $[X,Y] \in \ker \omega_Q \cap \xi|_Q$,
which finishes the proof.
\end{proof}

Therefore $\ker \omega_Q \cap \xi|_Q$ defines another foliation $\CN$ on $Q$, and hence
\be\label{eq:TCF}
T\CF = \mathbb{R}\{X_\lambda\} \oplus  T\CN.
\ee
Note that this splitting is $S^1$-invariant.

We now recall some basic properties of presymplectic manifold \cite{gotay}
and its canonical neighborhood theorem.
Fix an $S^1$-equivariant splitting of $TQ$
\be\label{eq:splitting}
TQ = T\CF \oplus G =\mathbb{R}X_\lambda\oplus T\CN \oplus G
\ee
by choosing an $S^1$-invariant complementary subbundle $G \subset TQ$.
This splitting is not unique but its choice will not matter for the coming discussions.
The null foliation carries a natural {\it transverse symplectic form}
in general \cite{gotay}. Since the distribution $T\CF \subset TQ$ is preserved by
Reeb flow, it generates the $S^1$-action thereon  in the current context. We denote by
$$
p_{T\CN;G}: TQ \to T\CN, \quad p_{G;G}:TQ \to G
$$
the projection to $T\CN$ and to $G$ respectively with respect to the splitting \eqref{eq:splitting}.
We denote by $T^*\CN \to Q$ the associated foliation cotangent bundle, i.e., the
dual bundle of $T\CN$.

We now consider the isomorphism
\be\label{eq:tildedlambda}
\widetilde{d\lambda|_\xi}: \xi \to \xi^*
\ee
and fix a splitting $T_QM = TQ \oplus N_QM$ with $T_QM = TM|_Q$
so that $N_QM \subset \xi_Q$:
This is possible since $\R\{X_\lambda\} \subset TQ$. We can also choose the splitting so that
it is $S^1$-equivariant. (See the proof of Proposition \ref{prop:neighbor-F} below.)

This leads to the further splitting
\be\label{eq:xiQ}
\xi|_Q = T\CN \oplus G \oplus N_QM
\ee
combined with \eqref{eq:splitting}, which in turn leads to
\be\label{eq:xi*}
\xi^*|_Q = (G \oplus N_QM)^\perp \oplus (T\CN)^\perp
\ee
where $(\cdot)^\perp$ denotes the annihilator of $(\cdot)$.
In particular, it induces an isomorphism
\be\label{eq:T*CN}
T^*\CN \cong (G \oplus N_QM)^\perp \subset \xi^*|_Q.
\ee
Now we consider the embedding $T^*\CN \to \xi$ defined by
the inverse of \eqref{eq:tildedlambda}, which we denote by
\be\label{eq:TNdual}
(T\CN)^{\#_{d\lambda}} = (\widetilde{d\lambda})^{-1}(T^*\CN).
\ee

Next we consider the symplectic normal bundle $(TQ)^{d\lambda} \subset T_QM$
defined by
\be\label{eq:TQdlambda}
(TQ)^{d\lambda} = \{v \in T_qM \mid d\lambda(v, w) = 0 ,  \forall w \in T_qQ\}.
\ee
We define a vector bundle
\be\label{eq:E}
E = (TQ)^{d\lambda}/T\CF,
\ee
and then have the natural embedding
\be\label{eq:embed-E}
E = (TQ)^{d\lambda}/T\CF \hookrightarrow T_QM/TQ = N_QM
\ee
induced by the inclusion map $(TQ)^{d\lambda} \hookrightarrow T_QM$.
The following is straightforward to check.

\begin{lem} The $d\lambda|_E$ induces a nondegenerate 2-form and so $E$ carries a
fiberwise symplectic form, which we denote by $\Omega$.
\end{lem}

We now consider the exact sequence
$$
0 \to E \to N_QM \to N_QM/E \to 0
$$
induced by \eqref{eq:embed-E}. The sequence is $S^1$-equivariant with respect to the natural
$S^1$-action thereon. We have the canonical isomorphism
$$
N_QM/E \cong \frac{T_QM}{TQ + (TQ)^{d\lambda}}
$$
which is $S^1$-equivariant. Then we have an $S^1$-equivariant splitting
\be\label{eq:TQM}
T_QM = (TQ + (TQ)^{d\lambda}) \oplus (T\CN)^{\#_{d\lambda}}
\ee
where $(T\CN)^{\#_{d\lambda}}$ is the $d\lambda$-dual \eqref{eq:TNdual}.
This also induces an embedding
\be\label{eq:embed-T*N}
T^*\CN \to (T\CN)^{\#_{d\lambda}} \hookrightarrow T_QM \to N_QM
\ee
which is also $S^1$-equivariant.

We now denote $F: = T^*\CN \oplus E \to Q$. The following proposition provides a
local model of the neighborhood of $Q \subset M$.

\begin{prop}\label{prop:neighbor-F} We
fix the splittings \eqref{eq:splitting} and \eqref{eq:xiQ}.
Then the sum of \eqref{eq:embed-T*N} and \eqref{eq:embed-E}
defines an isomorphism $T^*\CN \oplus E \to N_QM$ depending only on
the splittings.
\end{prop}
\begin{proof}
A straightforward dimension counting shows that the bundle map indeed is an
isomorphism.
\end{proof}

Identifying a neighborhood of $Q \subset M$ with a neighborhood of the zero section of $F$
and pulling back the contact form $\lambda$ to $F$, we may assume that our contact form $\lambda$ is
defined on a neighborhood of $o_F \subset F$. We also identify with $T^*\CN$ and $E$ as
their images in $N_QM$.

\begin{prop}\label{prop:S1-bundle} The $S^1$-action on $Q$
canonically induces the $S^1$-invariant vector bundle structure on $E$ such that
the form $\Omega$ is invariant under the $S^1$-action on $E$.
\end{prop}
\begin{proof}
The action of $S^1$ on $Q$ by $t\cdot q=\phi^t(q)$
canonically induces a $S^1$ action on $T_QM$ by $t\cdot v=(d\phi^t)(v)$,
for $v\in T_QM$.
Hence it gives rise to the following identity
\be
t^*d\lambda=d\lambda\label{eq:S1action}
\ee
 since the Reeb flow preserves $\lambda$.
We first show it is well-defined on $E\to Q$, i.e., if $v\in (T_qQ)^{d\lambda}$,
then $t\cdot v\in (T_{t\cdot q}Q)^{d\lambda}$.
In fact, by using \eqref{eq:S1action}, for $w\in T_{t\cdot q}Q$,
$$
d\lambda(t\cdot v, w)=\left((\phi^t)^*d\lambda\right)\left(v, (d\phi^t)^{-1}(w)\right)
=d\lambda\left(v, (d\phi^t)^{-1}(w)\right).
$$
This vanishes, since $Q$ consists of closed Reeb orbits and thus $d\phi^t$ preserves $TQ$.

Secondly, the same identity \eqref{eq:S1action} further indicates that this $S^1$ action preserves
$\Omega$ on fibers, i.e., $t^*\Omega=\Omega$, and we are done with the proof of this proposition.
\end{proof}

Summarizing the above discussion, we have concluded that
the base $Q$ is decorated by
the one-form $\theta: = i_Q^*\lambda$ on the base $Q$ and the bundle $E$ is decorated by
the fiberwise symplectic 2-form $\Omega$. They satisfy
the following additional properties:
\begin{enumerate}
\item $Q = o_F$ carries an $S^1$-action which is locally free.
In particular $Q/S^1$ is a smooth orbifold.
\item The one-form $\theta$ is $S^1$-invariant, and $d\theta$ is a presymplectic form.
\item The bundle $E$ carries an $S^1$-action that preserves the fiberwise 2-form $\Omega$
and hence induces a $S^1$-invariant symplectic vector bundle structure on $E$.
\item The bundle $F = T^*\CN \oplus E \to Q$ carries the direct sum $S^1$-equivariant
vector bundle structure compatible to the $S^1$-action on $Q$.
\end{enumerate}

We summarize the above discussions into the following theorem.

\begin{thm}\label{thm:morsebottsetup} Consider the locus $Q$ of closed Reeb orbits of
a Morse-Bott type contact form $\lambda$.
Let $(TQ)^{\omega_Q} \subset TQ$ be the null distribution of $\omega_Q = i_Q^*d\lambda$ and $\CF$ be
the associated characteristic foliation.
Then the restriction of $\lambda$ to $Q$ induces the following geometric structures:
\begin{enumerate}
\item $Q = o_F$ carries an $S^1$-action which is locally free.
In particular $Q/S^1$ is a smooth orbifold.
Fix an $S^1$-invariant splitting \eqref{eq:splitting}.
\item We have the natural identification
\be\label{eq:NQM}
N_QM \cong T^*\CN \oplus E = F,
\ee
as an $S^1$-equivariant vector bundle, where
$$
E: = (TQ)^{d\lambda}/T\CF
$$
is the symplectic normal bundle.
\item The two-form $d\lambda$ restricts to a nondegenerate skew-symmetric
two-form on $G$, and induces a fiberwise symplectic form
$
\Omega
$
on $E$ defined as above.
\end{enumerate}
\end{thm}

We say that $Q$ is of pre-quantization type if the rank of $d\lambda|_Q$ is maximal
and is of isotropic type if the rank of $d\lambda|_Q$ is zero. The general case will be
a mixture of the two.
\begin{rem}\label{rem:dim3}
In particular when $\dim M = 3$, such $Q$ must be either of
prequantization type or of isotropic type. This is the case that is considered in
\cite{HWZ3}. The general case considered in \cite{bourgeois} and \cite{behwz} includes the mixed type.
\end{rem}

\section{Contact thickening of Morse-Bott contact set-up}
\label{sec:thickening}

Motivated by the isomorphism in Theorem \ref{thm:morsebottsetup}, we consider the pair $(Q,\theta)$ and
the symplectic vector bundle $(E,\Omega) \to Q$ that satisfy the above properties.
We assume that $Q$ is compact and connected.

In the next section, we will associate the model contact form on the direct sum
$$
F = T^*\CN \oplus E
$$
and prove a canonical neighborhood theorem of the locus of closed Reeb orbits of
general contact manifold $(M,\lambda)$ such that the zero section of $F$ corresponds to $Q$.

To state our canonical neighborhood theorem, we need to first
identify the relevant geometric structure of the canonical neighborhoods.
For this purpose, introduction of the following notion is useful.

\begin{defn}[Pre-Contact Form] We call a one-form $\theta$ on a manifold $Q$ a \emph{pre-contact} form
if $d\theta$ has constant rank.
\end{defn}

\subsection{The $S^1$-invariant pre-contact manifold $(Q,\theta)$}
\label{subsec:Q}

First, we consider the pair $(Q,\theta)$ such that $Q$ carries a nontrivial
$S^1$-action preserving the one-form $\theta$. After taking the quotient of $S^1$ by some finite subgroup,
we may assume that the action is effective. We will also assume that the action
is locally free. Then by the general theory of group actions of
compact Lie group (see \cite{helgason} for example),
the action is free on a dense open subset and has only a finite number
of different types of isotropy groups. In particular the quotient
$P: = Q/S^1$ becomes a presymplectic orbifold with a finite number of
different types of orbifold points. We denote by $Y$ the vector field generating the
$S^1$-action, i.e., the $S^1$-action is generated by its flows.

We require that the circle action preserves $\theta$, i.e., $\CL_Y\theta = 0$.
Since the action is locally free and free on a dense open subset of $Q$,
we can normalize the action so that
\be\label{eq:thetaX=1}
\theta(Y) \equiv 1.
\ee
We denote this normalized vector field by $X_\theta$. We would like to emphasize that
$\theta$ may not be a contact form but can be regarded as the connection form of
the circle bundle $S^1 \to Q \to P$ over the orbifold $P$ in general. Although $P$
may carry non-empty set of orbifold points, the connection form $\theta$ is assumed to
be smooth on $Q$.

Similarly as in Lemma \ref{lem:integrableCN}, we also require
the presence of $S^1$-invariant splitting
$$
\ker d\theta = \R \{X_\theta\} \oplus H
$$
such that the subbundle $H$ is also integrable.

With these terminologies introduced, we can rephrase Theorem \ref{thm:morsebottsetup} as the following.

\begin{thm}\label{thm:morsebottsetup2}
Let $Q$ be the locus foliated by closed Reeb orbits of
a contact manifold $(M,\lambda)$ of Morse-Bott type.
Then $Q$ carries a locally free $S^1$-action and
\begin{enumerate}
\item an $S^1$-invariant pre-contact form $\theta$ given by $\theta = i_Q^*\lambda$,
\item a splitting
\be\label{eq:kernel-dtheta}
\ker d\theta = \R\{X_\theta\} \oplus H,
\ee
such that  the distribution $H = \ker d\theta \cap \xi|_Q$ is integrable,
\item
an $S^1$-equivariant symplectic vector bundle $(E,\Omega) \to Q$ with
$$
E = (TQ)^{d\lambda}/\ker d\theta, \quad \Omega = [d\lambda]_E
$$
\end{enumerate}
\end{thm}
Here we use the fact that there exists a canonical embedding
$$
E = (TQ)^{d\lambda}/\ker d\theta \hookrightarrow T_QM/ TQ = N_QM
$$
and $d\lambda|_{(TQ)^{d\lambda}}$ canonically induces a bilinear form $[d\lambda]_E$
on $E = (TQ)^{d\lambda}/\ker di_Q^*\lambda$ by symplectic reduction.

\begin{defn} Let $(Q,\theta)$ be a pre-contact manifold
equipped with a locally free $S^1$-action generated by a vector field $Y$,
and with a $S^1$-invariant one-form $\theta$ and
the splitting \eqref{eq:kernel-dtheta}. Assume $\theta(Y) \equiv 1$.
We call such a triple $(Q,\theta,H)$
a \emph{Morse-Bott contact set-up}.
\end{defn}

As before, we denote by $\CF$ and $\CN$
the associated foliations on $Q$, and decompose
$$
T\CF = \R\{X_\theta\} \oplus T\CN.
$$

We define a one-form $\Theta_G$ on $T^*\CN$ as follows. For a tangent
$\xi \in T_\alpha(T^*\CN)$, define
\be\label{eq:thetaG}
\Theta_G(\xi): = \alpha(p_{T\CN;G} d\pi_{(T^*\CN)}(\xi))
\ee
using the splitting
$$
TQ = \R\{X_\theta\} \oplus T\CN \oplus G.
$$
By definition, it follows $\Theta_G|_{VT(T^*\CN)} \equiv 0$ and $d\Theta_G(\alpha)$ is nondegenerate
on
$$
\widetilde{T_q\CN} \oplus VT_\alpha T^*\CN \cong T_q\CN \oplus T_q^*\CN
$$
which becomes
the canonical pairing defined on $T_q\CN \oplus T_q^*\CN$ under the identification.

\subsection{The bundle $E$}

We next examine the structure of
the $S^1$-equivariant symplectic vector bundle $(E, \Omega)$.

We denote by $\vec R$ the radial vector field which
generates the family of radial multiplication
$$
(c, e) \mapsto c\, e.
$$
This vector field is invariant under the given $S^1$-action on $E$, and vanishes on the zero section.
By its definition, $d\pi(\vec R)=0$, i.e., $\vec R$ is in the vertical distribution, denoted by $VTE$, of $TE$.

Denote the canonical isomorphism $V_eTE \cong E_{\pi(e)}$ by $I_{e;\pi(e)}$.
It obviously intertwines the scalar multiplication, i.e.,
$$
I_{e;\pi(e)}(\mu\, \xi) = \mu I_{e;\pi(e)}(\xi)
$$
for a scalar $\mu$.
It also satisfies the following identity \eqref{eq:dRc=Ic} with respect to the derivative of the
fiberwise scalar multiplication map $R_c: E \to E$.

\begin{lem} Let $\xi \in V_e TE$. Then
\be\label{eq:dRc=Ic}
I_{c\, e;\pi(c\, e)}(dR_c(\xi)) = c\, I_{e;\pi(e)}(\xi)
\ee
on $E_{\pi(c\, e)} = E_{\pi(e)}$ for any constant $c$.
\end{lem}
\begin{proof} We compute
\beastar
I_{c\, e;\pi(c\, e)}(dR_c(\xi)) & = & I_{c\, e;\pi(c\, e)}\left(\frac{d}{ds}\Big|_{s=0}c(e + s\xi)\right) \\
& = & I_{c\, e;\pi(c\, e)}(R_c(\xi)) = c\, I_{e;\pi(e)}(\xi)
\eeastar
which finishes the proof.
\end{proof}

We then define the fiberwise two-form  $\Omega^v$ on $VTE \to E$ by
$$
\Omega^v_e(\xi_1,\xi_2) = \Omega_{\pi_F(e)}(I_{e;\pi(e)}(\xi_1),I_{e;\pi(e)}(\xi_2))
$$
for $\xi_1, \xi_2\in V_eTE$,
and one-form
$
\vec R \rfloor \Omega^v
$
respectively.

Now we introduce an $S^1$-invariant symplectic connection on $E$ and choose the splitting
$$
TE = HTE  \oplus VTE.
$$
Existence of such an invariant connection follows
e.g., by averaging over the compact group $S^1$. We denote by $\widetilde \Omega$ the resulting
two-form on $E$. We extend the fiberwise form $\Omega$ of $E$ into the differential two-form
$\widetilde \Omega$ on $E$ by setting
$$
\widetilde \Omega_e(\xi,\zeta) = \Omega^v_e(\xi^v, \zeta^v).
$$

Denote by $\vec R$ the radial vector field of $E \to Q$ and consider the one-form
\be\label{eq:lambdaF}
\vec R \rfloor \widetilde \Omega
\ee
which is invariant under the action of $S^1$ on $E$.
\begin{rem}
Suppose $d\lambda_E(\cdot, J_E \cdot)=: g_{E;J_E}$ defines a
Hermitian vector bundle $(\xi_E, g_{E,J}, J_E)$.
Then we can write the radial vector field considered in the previous section as
$$
\vec R(e) = \sum_{i=1}^k r_i \frac{\del}{\del r_i}
$$
where $(r_1,\cdots, r_k)$ is the coordinates of $e$ for
a local frame $\{e_1, \cdots, e_k\}$ of the vector bundle $E$. By definition, we have
\be\label{eq:vecRe}
I_{e;\pi_E(e)}(\vec R(e)) = e.
\ee
Obviously the right hand side expression does not depend on the choice of
local frames.
Let $(E,\Omega,J_E)$ be a Hermitian vector bundle and define
$|e|^2 = g_F(e,e)$. Motivated by the terminology used in \cite{bott-tu}, we call the one-form
\be\label{eq:angularform}
\psi = \psi_\Omega = \frac{1}{r} \frac{\del}{\del r} \rfloor \Omega^v
\ee
the \emph{global angular form} for the Hermitian vector bundle $(E,\Omega,J_E)$.
Note that $\psi$ is defined only on $E \setminus o_E$ although
$\Omega$ is globally defined.
\end{rem}

We state the following lemma.

\begin{lem}\label{lem:ROmega} Let $\Omega$ be as above.
Then,
\begin{enumerate}
\item $\vec R \rfloor d \widetilde \Omega = 0$.
\item
For any non-zero constant $c > 0$, we have
$$
R_{c}^*\widetilde \Omega = c^2\, \widetilde \Omega.
$$
\end{enumerate}
\end{lem}
\begin{proof}
Notice that $\widetilde\Omega$ is compatible with $\Omega$ in the sense of symplectic fibration
and the symplectic vector bundle connection is nothing but the
Ehresmann connection induced by $\widetilde\Omega$, which is a symplectic connection now.
Since $\vec R$ is vertical, the statement (1) immediately follows from the fact that the symplectic connection is vertical closed.

It remains to prove statement (2).
Let $e \in E$ and $\xi_1, \, \xi_2 \in T_e E$.  By definition, we derive
\beastar
(R_c^*\widetilde \Omega)_e(\xi_1,\xi_2)
& = & \widetilde \Omega_{c\,e}(dR_c(\xi_1^v), dR_c(\xi_2^v))\\
& = & \Omega^v_{c\,e}(dR_c(\xi_1^v), dR_c(\xi_2^v)))
= c^2 \widetilde \Omega_e(\xi_1,\xi_2)
\eeastar
where we use the equality \eqref{eq:dRc=Ic} and $\pi_F(c\, e) = \pi(e)$ for the fourth equality.

This proves $R_c^*\widetilde \Omega = c^2 \widetilde \Omega$.
\end{proof}

It follows from Lemma \ref{lem:ROmega} (2) that we get $\CL_{\vec R}\widetilde\Omega=2\widetilde\Omega$. By Cartan's formula, we get
$$
d(\vec R \rfloor \widetilde \Omega) =  2\widetilde \Omega.
$$

\subsection{Canonical contact form and contact structure on $F$}

Let $(Q,\theta,H)$ be a given  Morse-Bott contact set up and $(E,\Omega) \to Q$ be
any $S^1$-equivariant symplectic vector bundle equipped with an $S^1$-invariant symplectic connection
on it.

Now we equip the bundle $F = T^*\CN \oplus E$ with a canonical $S^1$-invariant contact form
on $F$. We denote the bundle projections by $\pi_{E;F}: F \to E$ and $\pi_{T^*\CN;F}: F \to T^*\CN$
of the splitting $F  = T^*\CN \oplus E$ respectively, and provide the direct sum connection on
$F = T^*\CN \oplus E$.

\begin{thm} Let $(Q,\theta,H)$ be a Morse-Bott contact set-up.
Denote by $\CF$ and $\CN$ the foliations associated to the distribution $\ker d\theta$ and $H$,
respectively. We also denote by $T\CF$, $T\CN$ the associated foliation tangent bundles and
$T^*\CN$ the foliation cotangent bundle of $\CN$. Then for any symplectic vector bundle $(E,\Omega) \to Q$ with an $S^1$-invariant symplectic connection,
the following holds:
\begin{enumerate}
\item
The total space of the bundle $F = T^*\CN \oplus E$ carries a canonical contact form $\lambda_{F;G}$ defined as in \eqref{eq:lambdaF}, for
each choice of complement $G$ such that $TQ = T\CF \oplus G$.
\item
For any two such choices of $G, \, G'$, the associated contact forms are
canonically gauge-equivalent by a bundle map $\psi_{GG'}: TQ \to TQ$ preserving $T\CF$.
\end{enumerate}
\end{thm}
\begin{proof} We define a differential one-form on $F$ explicitly by
\be\label{eq:lambdaF}
\lambda_F = \pi_F^*\theta + \pi_{T^*\CN;F}^*\Theta_G +
\frac{1}{2} \pi_{E;F}^*\left(\vec R \rfloor \widetilde \Omega\right).
\ee
Using Lemma \ref{lem:ROmega}, we obtain
\be\label{eq:dlambdaF}
d\lambda_F = \pi_F^*d\theta + \pi_{T^*\CN;F}^*d\Theta_G  + \pi_{E;F}^*\widetilde \Omega
\ee
by taking the differential of \eqref{eq:lambdaF}.

A moment of examination of
this formula gives rise to the following

\begin{prop}\label{lem:deltaforE} There exists some $\delta> 0$ such that
the one-form $\lambda_F$ is a contact form
on the disc bundle $D^{\delta}(F)$, where
$$
D^{\delta}(F) = \{(q,v) \in F \mid \|v\| < \delta \}
$$
such that $\lambda_F|_{o_F} = \theta$ on $Q \cong o_F \subset F$.
\end{prop}
\begin{proof} This immediately follows from the formulae \eqref{eq:lambdaF}
and \eqref{eq:dlambdaF}.
\end{proof}
This proves the statement (1).

For the proof of the statement (2), we first note that the bundle $E$
itself does not depend on the choice of $G$. On the other hand, we put the one-form
$$
\lambda_{F;G} = \pi_F^*\theta + \pi_{T^*\CN;F}^*\Theta_G +
\frac{1}{2}\pi^*_{E;F}\left(\vec R \rfloor \widetilde \Omega\right)
$$
on $E$, which depends on $G$ in general because the one-form $\Theta_G$ does. Furthermore
we recall that the projection map $\pi_{T^*\CN;F}$ also depends on the canonical splitting
\be\label{eq:TQF}
T_Q F = HT_Q F \oplus VT_Q F \cong TQ \oplus F|_Q.
\ee
Now we fix this splitting $T_Q F$ and let $G,\, G'$ be two splittings
of $TQ$
$$
TQ = T\CF \oplus G = T\CF \oplus G'.
$$
Since both $G, \, G'$ are transversal to $T \CF$ in $TQ$,
we can represent $G'$ as the graph of the bundle map $A_G: G \to T\CN$. Then we consider the
bundle isomorphism
$$
\psi_{GG'}: TQ/\R\{X_\lambda\} \to TQ/\R\{X_\lambda\}
$$
defined by
$$
\psi_{GG'} = \left(\begin{matrix} Id_{T\CN} & A_G\\
0 & id_G \end{matrix}\right)
$$
under the splitting $TQ = \R\{X_\lambda\} \oplus T\CN \oplus G$. Then $\psi_{GG'}(G) = \operatorname{Graph} A_G$
and $\psi_{GG'}|_{T\CN} = id_{T\CN}$.
Therefore we have $p_{T\CN;G} = p_{T\CN;G'}\circ \psi_{GG'}$.

We compute
\beastar
\Theta_{G}(\alpha)(\eta) & = &\alpha(p_{T\CN;G}(d\pi_{T^*\CN}(\eta)))\\
& =&\alpha(p_{T\CN;G'}\circ \psi_{GG'}(\eta))
=  \Theta_{G'}(\alpha)(\psi_{GG'}(\eta)).
\eeastar
This proves $\Theta_{G} = \Theta_{G'}\circ \psi_{GG'}$.
\end{proof}

Now we study the contact geometry of $(D^\delta(F),\lambda_F)$. We first note that the two-form
$d\lambda_F$ is a presymplectic form with one dimensional kernel such that
$$
d\lambda_F|_{VTF} = \widetilde \Omega^v|_{VTF}.
$$
Denote by $\widetilde{X}:=(d\pi_{F;H})^{-1}(X)$ the horizontal lifting of the vector field $X$ on $Q$,
where
$$
d\pi_{F;H}:= d\pi_F|_H:HTF \to TQ
$$
is the bijection of the horizontal distribution and $TQ$.

\begin{lem}[Reeb Vector Field] The Reeb vector field $X_F$ of $\lambda_F$ is given by
$$
X_F=\widetilde X_\theta,
$$
where $\widetilde X_\theta$ denotes the horizontal lifting of $X_\theta$ to $F$.
\end{lem}
\begin{proof} We only have to check the defining property $\widetilde X_\theta \rfloor \lambda_F = 1$
and $\widetilde X_\theta \rfloor d\lambda_F = 0$.
We first look at
\beastar
\lambda_F(\widetilde X_\theta)&=& \pi_F^*\theta(\widetilde X_\theta) + \pi_{T^*\CN;F}^*\Theta_G(\widetilde X_\theta)
 + \frac{1}{2} \pi_{E;F}^*\left(\vec R \rfloor \widetilde \Omega\right)(\widetilde X_\theta)\\
 & = & \theta(X_\theta) + 0 + 0 = 1.
 \eeastar
Here $\widetilde{\Omega}(\vec{R}, \widetilde X_\theta) = 0$
by definition of $\widetilde\Omega$  since $\widetilde X_\theta^v = 0$.
Then we calculate
\beastar
\widetilde X_\theta\rfloor d\lambda_F&=&\widetilde X_\theta\rfloor (\pi_F^*d\theta
+\widetilde{\Omega}+\pi_{T^*\CN;F}^*d\Theta_G)\\
&=& \widetilde X_\theta\rfloor \pi_F^*d\theta+\widetilde X_\theta\rfloor \widetilde{\Omega}+\widetilde X_\theta\rfloor \pi_{T^*\CN;F}^*d\Theta_G = 0.
\eeastar
We only need to explain why the last term $\widetilde X_\theta\rfloor \pi_{T^*\CN;F}^*d\Theta_G$ vanishes.
In fact $pr_{\CN;G} d\pi_F(\widetilde X_\theta) = pr_{\CN; G}(X_\theta) = 0$. Using this,
the definition of $\Theta_G$ and the $S^1$-equivariance of the vector bundle $F \to Q$ and the fact that $\widetilde X_\theta$ is
the vector field generating the $S^1$-action, we derive
\beastar
\widetilde X_\theta\rfloor \pi_{T^*\CN;F}^*d\Theta_G =  0
\eeastar
by a straightforward computation.
This finishes the proof.
\end{proof}

Now we calculate the contact structure $\xi_F$.

\begin{lem}[Contact Distribution] \label{lem:decomp-VW}
At each point $(\alpha,e) \in U_F \subset F$, we
define two subspaces of $T_{(\alpha,e)}F$
$$
V:=\{\xi_V \in T_{(\alpha,e)}F \mid \xi_V =-\pi_{T^*\CN;F}^*\Theta_G(\eta)X_F+ \widetilde\eta, \, \eta \in \ker \theta \}
$$
and
$$
W:=\{\xi_W \in T_{(\alpha,e)}F \mid \xi_W :=-\frac{1}{2}\pi_{E;F}^*\widetilde \Omega(e,v)X_F+v, \,  v\in VTF\},
$$
Then
$
\xi_F=V\oplus W.
$
\end{lem}
\begin{proof}
By straightforward calculation, both $V$ and $W$ are subspaces of $\xi_F = \ker \lambda_F$.

For any $\xi\in \xi_F$, we decompose $\xi=\xi^{\text{h}}+\xi^{\text{v}}$
using the decomposition $TF=HTF\oplus VTF$.
Since $\pi_{T^*\CN;F}^*\Theta_G(\widetilde\eta) = \Theta_G(\eta)$ and we can write $\xi^{\text{v}} = I_{(e;\pi(e))}(v)$
for a unique $v \in E_{\pi(e)}$. Therefore we need to find $b \in \R$, $\eta \in \ker \theta$
so that for the horizontal vector $\xi^{\text{h}}=\widetilde{(\eta+b X_\theta)}$
\beastar
\lambda_F(\xi) & = & 0 \\
- \pi_{T^*\CN;F}^*\Theta_G(\eta)X_F + \widetilde\eta & \in & V\\
-\frac{1}{2} \pi_{E;F}^*\widetilde \Omega (e, v)X_F+v & \in & W.
\eeastar
Then
\beastar
\xi^h&=&\widetilde{(\eta+ b X_\theta)}\\
&=&\widetilde\eta+bX_F
\eeastar
which determines $\eta \in T_{\pi(e)}\CN \oplus G_{\pi(e)}$ uniquely. We need to determine $b$.

Since
\beastar
0=\lambda_F(\xi)&=&\lambda_F(\widetilde\eta+bX_F+\xi^{\text{v}})\\
&=&b+\lambda_F(\widetilde\eta)+\lambda_F(\xi^{\text{v}})\\
&=&b+ \pi_{T^*\CN;F}^*\Theta_G (\widetilde\eta)+\frac{1}{2} \pi_{E;F}^*\widetilde \Omega(e,v)
\eeastar
Then we set $\xi_W = -\frac{1}{2} \pi_{E;F}^*\widetilde \Omega(\vec{R}, v)X_F+ v$ for $v$ such that $I_{e;\pi(e)}(v) = \xi^v$
and then finally choose $b = - \pi_{T^*\CN;F}^*\Theta_G(\widetilde\eta)$ so that
$\xi_V: = - \pi_{T^*\CN;F}^*\Theta_G(\widetilde\eta) X_F + \widetilde \eta$.
Therefore we have proved $\xi_F=V+W$.

 To see it is a direct sum, assume
 $$
 -(\pi_{T^*\CN;F}^*\Theta_G)(\widetilde{\eta})X_F+\widetilde\eta-\frac{1}{2}\pi_{E;F}^*\widetilde \Omega(e, v)X_F+v=0,
 $$
 for some $\eta\in \xi_\lambda$ and $v\in VTM$.
 Apply $d\pi$ to both sides, and it follows that
 $$
 -\big((\pi_{T^*|CN;F}^*\Theta_G)(\widetilde{\eta})+\frac{1}{2} \pi_{E;F}^*\widetilde \Omega(e, v)X_\theta\big)+\eta=0.
 $$
Hence $\eta=0$, and then $v=0$ follows since $X_F$ is in horizontal part.
This finishes the proof.
\end{proof}

\section{Canonical neighborhoods of the locus of closed Reeb orbits}
\label{sec:normalform}

Now let $Q$ be the submanifold of $(M,\lambda)$ that is
foliated by the closed Reeb orbits of $\lambda$ with constant period $T$.
Consider the  Morse-Bott contact set-up $(Q,\theta,H)$ defined as before
and the  symplectic vector bundle $(E,\Omega)$ associated to $Q$.

Now in this section, we prove the following canonical neighborhood
theorem as the converse of Theorem \ref{thm:morsebottsetup2}.

\begin{thm}[Canonical Neighborhood Theorem]\label{thm:neighborhoods2}
Let $Q$ be the submanifold of closed Reeb orbits of
Morse-Bott type contact form $\lambda$, and $(Q,\theta)$ and $(E,\Omega)$
be the associated pair and $F = T^*\CN \oplus E$ defined above.
Then there exist neighborhoods $U$ of $Q$ and $U_F$ of the zero section $o_F$,
and a diffeomorphism $\psi: U_F \to U$ and a function $f: U_F \to \R$ such that
$$
\psi^*\lambda = f\, \lambda_{F;G} \, f|_{o_F} \equiv 1, \, df|_{o_F}\equiv 0
$$
and
$$
i_{o_F}^*\psi^*\lambda = \theta, \quad (\psi^*d\lambda|_{VTF})|_{o_F} = 0\oplus \Omega
$$
where we use the canonical identification of $VTF|_{o_F} \cong T^*\CN \oplus E$ on the
zero section $o_F \cong Q$.
\end{thm}

We first identify the local pair $(\CU, Q) \cong (U_F,Q)$ by a
diffeomorphism $\phi: \CU \to U_F$ such that
$$
\phi|_Q = \textrm{id}_Q, \quad d\phi(N_QM) = T^*\CN \oplus E.
$$
Such a diffeomorphism obviously exists by definition of $E$ and $T^*\CN$ via
the normal exponential map with respect to any metric $g$ (defined on $\CU$)
that satisfies the following property:
We note that we have the associated short exact sequences
\bea
&{}& 0 \to T\CF \to TQ \to TQ/T\CF \to 0\label{eq:sequence1}\\
&{}& 0 \to TQ \to T_QM \to N_QM \to 0\label{eq:sequence2}\\
&{}& 0 \to E \to N_QM \to N_QM/E \to 0\label{eq:sequence3}
\eea
which are $S^1$-equivariant with respect to the above mentioned
natural induced $S^1$-action on $Q$.
We take $S^1$-equivariant splittings of \eqref{eq:sequence2},
\eqref{eq:sequence3} in addition to that of \eqref{eq:sequence1}
used in Theorem \ref{thm:splitting1}. We then choose an $S^1$-equivariant
metric on the vector bundle $N_QM \cong F$ whose associated normal exponential
map of $Q \cong o_F$ respects the above chosen splittings.

From now on, we will sometimes denote $U_F$ by $F$ in the following context
if there is no danger of confusion.
Now $F$ carries two contact forms $\psi^*\lambda, \, \lambda_F$ with $\psi = \phi^{-1}$ and they are the same
on the zero section $o_F$. With this preparation, we will derive Theorem \ref{thm:neighborhoods2} by the following
general submanifold version of Gray's theorem.

\begin{thm}\label{thm:normalform} Let $M$ be an odd dimensional manifold with two contact
forms $\lambda_0$ and $\lambda_1$ on it.
Let $Q$ be a closed manifold of closed Reeb orbits of $\lambda_0$ in $M$ and
\be\label{eq:equalities}
\lambda_0|_{T_QM} =\lambda_1|_{T_QM}, \, d\lambda_0|_{T_QM} = d\lambda_1|_{T_QM}
\ee
where we denote $T_QM = TM|_Q$.
Then there exists a diffeomorphism $\phi$ from a neighborhood $\mathcal{U}$ to $\mathcal{V}$ such that
\be\label{eq:phi}
\phi|_Q = \text{\rm id}_Q,
\ee
and a function $f>0$ such that
$$
\phi^*\lambda_1 = f \cdot\lambda_0,
$$
and
\be\label{eq:f}
f|_Q\equiv 1,\quad df|_{T_QM}\equiv 0.
\ee
\end{thm}
\begin{proof}
By the assumption on $\lambda_0, \, \lambda_1$, there exists a small tubular neighborhood of $Q$ in $M$, denote by $\mathcal{U}$,
such that the isotopy $\lambda_t=(1-t)\lambda_0+t\lambda_1$, $t\in [0,1]$, are contact forms in $\mathcal{U}$:
this follows from the requirement \eqref{eq:equalities}. Moreover, we have
$$
\lambda_t|_{T_QM} \equiv \lambda_0|_{T_QM}(=\lambda_1|_{T_QM}), \quad \text{ for any } t\in [0,1].
$$
Then the standard Moser's trick will finish up the proof. For reader's convenience, we provide the details here.

We are looking for a family of diffeomorphisms onto its image $\phi_t: \CU' \to \CU$
for some smaller open subset $\CU' \subset \overline{\CU'} \subset \CU$
such that
$$
\phi_t|_Q = \textrm{id}_Q, \quad d\phi_t|_{T_QM} = \textrm{id}_{T_QM}
$$
for all $t \in [0,1]$, together with a family of functions $f_t>0$
defined on $\overline{\CU'}$ such that
$$
\phi_t^*\lambda_t = f_t\cdot \lambda_0 \quad \text{on } \, \CU'
$$
for $0 \leq t \leq 1$. We will further require $f_t \equiv 1$ on $Q$ and $df_t|_Q \equiv 0$.

Since $Q$ is a closed manifold, it is enough to look for the vector fields $Y_t$ generating $\phi_t$ via
\be\label{eq:ddtphit}
\frac{d}{dt}\phi_t=Y_t\circ \phi_t, \quad \phi_0=id,
\ee
that satisfies
$$
\begin{cases}
\phi_t^*\left(\frac{d}{dt}\lambda_t+\CL_{Y_t}\lambda_t\right)
= \frac{f_t'}{f_t} \phi_t^*\lambda_t\\
Y_t|_Q\equiv0, \quad \nabla Y_t|_{T_QM} \equiv 0.
\end{cases}
$$
By Cartan's magic formula, the first equation gives rise to
\be\label{eq:cartan}
d(Y_t\rfloor\lambda_t)+Y_t\rfloor d\lambda_t=\CL_{Y_t}\lambda_t=(\frac{f_t'}{f_t}\circ \phi_t^{-1})\lambda_t-\alpha,
\ee
where
$$
\alpha=\lambda_1-\lambda_0 (\equiv \frac{d \lambda_t}{dt}).
$$
Now, we need to show that there exists $Y_t$ such that
$\frac{d}{dt}\lambda_t+\CL_{Y_t}\lambda_t$ is proportional to $\lambda_t$.
Actually, we can make our choice of $Y_t$ unique if we restrict ourselves
to those tangent to $\xi_t = \ker \lambda_t$ by Lemma \ref{lem:decompose}.

We require $Y_t\in \xi_t$ and then \eqref{eq:cartan} becomes
\be\label{eq:alphatYt}
\alpha = - Y_t \rfloor d\lambda_t +(\frac{f_t'}{f_t}\circ \phi_t^{-1})\lambda_t.
\ee
This in turn determines $\phi_t$ by integration.
Since $\alpha|_Q =(\lambda_1-\lambda_0)|_Q=0$ and $\lambda_t|_{T_QM} = \lambda_0|_{T_QM}$,
(and hence $f_t \equiv 1$ on $Q$), it follows that $Y_t=0$ on $Q$. Therefore by compactness of
$[0,1] \times Q$, the domain of existence of the ODE $\dot x = Y_t(x)$
includes an open neighborhood of $[0,1] \times Q \subset \R \times M$ which we may
assume is of the form $(-\epsilon, 1+ \epsilon) \times \CV$.

Now going back to \eqref{eq:alphatYt}, we find that the coefficient
$\frac{f_t'}{f_t}\circ \phi_t^{-1}$ is uniquely determined.
We evaluate $\alpha = \lambda_1 - \lambda_0$ against the vector fields $X_t := (\phi_t)_*X_{\lambda_0}$, and get
\be\label{eq:logft}
\frac{d}{dt}\log f_t = \frac{f_t'}{f_t}=(\lambda_1(X_t)-\lambda_0(X_t))\circ \phi_t,
\ee
which determines $f_t$ by integration with the initial condition $f_0 \equiv 1$.

It remains to check the additional properties \eqref{eq:phi}, \eqref{eq:f}.
We set
$$
h_t = \frac{f_t'}{f_t}\circ \phi_t^{-1}.
$$
\begin{lem}
$$
dh_t|_{T_QM} \equiv 0
$$
\end{lem}
\begin{proof} By \eqref{eq:logft}, we obtain
$$
dh_t = d(\lambda_1(X_t))-d(\lambda_0(X_t)) = \CL_{X_t}(\lambda_1 - \lambda_0) - X_t \rfloor d(\lambda_1 - \lambda_0).
$$
Since $X_t = X_{\lambda_0} = X_{\lambda_1}$ on $Q$, $X_t \in \xi_1 \cap \xi_0$ and so the
second term vanishes.

For the first term, consider $p \in Q$ and $v \in T_pM$. Let $Y$ be a locally defined vector field
with $Y(p) = v$. Then we compute
$$
\CL_{X_t}(\lambda_1 - \lambda_0)(Y)(p) = \CL_{X_t}((\lambda_1 - \lambda_0)(Y))(p) - (\lambda_1 - \lambda_0)(\CL_{X_t}Y)(p).
$$
The second term of the right hand side vanishes since $\lambda_1 = \lambda_0$ on $T_pM$ for $p \in Q$. For the first one, we note
$X_t$ is tangent to $Q$ for all $t$ and $(\lambda_1 - \lambda_0)(Y) \equiv 0$ on $Q$ by the hypothesis
$\lambda_0 = \lambda_1$ on $T_QM$. Therefore the first term also vanishes. This finishes the proof.
\end{proof}

Now we set $g_t = \log f_t$. Since $\phi_t$ is a
diffeomorphism and $\phi_t(Q) \subset Q$, this implies $dg'_t = 0$ on $Q$ for all $t$.
By integrating $dg'_t = 0$ with $dg'_0 = 0$ along $Q$ over time $t = 0$ to $t=1$, which
implies $dg_t = 0$ along $Q$ (meaning $dg_t|_{T_QM} = 0$), i.e., $df_t = 0$ on $Q$.
This completes the proof of Theorem \ref{thm:normalform}.
\end{proof}

Applying this theorem to $\lambda$ and $\lambda_F$ on $F$ with $Q$ as the zero section $o_F$,
we can wrap-up the proof of Theorem \ref{thm:neighborhoods2}

\begin{proof}[Proof of Theorem \ref{thm:neighborhoods2}]
The requirement \eqref{eq:psi*lambda} and the first of \eqref{eq:ioFpsi}
are immediate translations of Theorem \ref{thm:normalform}.
For the second requirement in \eqref{eq:ioFpsi}, we compute
$$
\psi^*d\lambda = df \wedge \lambda_F + f\, d\lambda_F.
$$
By using $df|_{o_F}= 0$ and $f|_{o_F} = 1$, we derive
$$
\psi^*d\lambda|_{VTF} = d\lambda_F|_{VTF} = \Omega
$$
on $o_F$. This then finishes the second an hence finishes the proof.
\end{proof}

\begin{defn}[Normal Form of Contact Form]\label{defn:normalneighborhood} We call $(U_F,f \, \lambda_F)$ the normal
form of the contact form $\lambda$ associated to the Morse-Bott submanifold $Q$ of closed Reeb orbits.
\end{defn}

Note that the contact structures associated to $\psi^*\lambda$ and $\lambda_F$ are the same which is given by
$$
\xi_F = \ker \lambda_F = \ker \psi^*\lambda.
$$
This proves the following normal form theorem of the contact structure $(M,\xi)$
in a neighborhood of $Q$.

\begin{prop} Suppose that $Q \subset M$ be a submanifold of closed Reeb orbits of
$\lambda$. Then there exists a contactomorphism from
a neighborhood $\CU \supset Q$ to a neighborhood of the zero section of
$F$ equipped with $S^1$-equivariant contact structure $\lambda_F$.
\end{prop}

\begin{defn}[Normal Form of Contact Structure] We call $(F,\xi_F)$ the normal form of
$(M,\xi)$ associated to the Morse-Bott submanifold $Q$ of closed Reeb orbits.
\end{defn}
However, the Reeb vector fields of $\psi^*\lambda$ and $\lambda_F$ coincide
only along the zero section in general.

In the rest of the paper, we will work with $F$ and for the general contact form $\lambda$
that satisfies
\be\label{eq:F-lambda}
\lambda|_{o_F} \equiv \lambda_F|_{o_F}, \quad d\lambda|_{VTF}|_{o_F} = \Omega.
\ee
In particular $o_F$ is also the locus of closed Reeb orbits  with the same period $T$
of a Morse-Bott contact form $\lambda$.
We re-state the above normal form theorem in this context.

\begin{prop} Let $\lambda$ be any contact form in a neighborhood of $o_F$ on $F$
satisfying \eqref{eq:F-lambda}. Then there exist an open embedding $\psi: \CU \to F$
and a function $f$ on $\CU$ for open neighborhoods $\CU, \, F$ of $o_F$ such that
$\psi|_{oF} = \textrm{id}_{o_F}$,
$\psi^*\lambda = f\, \lambda_F$ with $f|_{o_F} \equiv 1$ and $df|_{o_F} \equiv 0$.
\end{prop}

We denote $\xi$ and $X_\lambda$ the corresponding contact structure and Reeb vector field of $\lambda$,
and $\pi_\lambda$, $\pi_{\lambda_F}$ the corresponding projection from
$TF$ to $\xi$ and $\xi_F$.

\section{Linearization of closed Reeb orbit on the normal form}
\label{sec:linearization-orbits}

In this section, we systematically examine the decomposition of the linearization map of
closed Reeb orbits in terms of the coordinate expression of
the loops $z$ in $F$ in this normal neighborhood.

For a given map $z: S^1 \to F$, we denote by $x := \pi_F \circ z$. Then
we can express
$$
z(t) = (x(t), s(t)), \quad t \in S^1
$$
where $s(t) \in F_{x(t)}$, i.e., $s$ is the section os $x^*F$.

We regard this decomposition as the map
$$
\CI: C^\infty(S^1, F) \to \CH^F_{S^1}
$$
where $\CH^F$ is the infinite dimensional vector bundle
$$
\CH^F_{S^1} = \bigcup_{x \in C^\infty(S^1,F)} \CH^F_{S^1,x}
$$
where $\CH^F_{S^1,x}$ is the vector space  given by
$$
\CH^F_{S^1,x} = \Omega^0(x^*F)
$$
the set of smooth sections of the pull-back vector bundle $x^*F$. This provides
a coordinate description of $C^\infty(S^1,F)$ in terms of $\CH^F_{S^1}$. We denote the corresponding
coordinates $z = (u_z,s_z)$ when we feel necessary to make the dependence of $(x,s)$ on $z$
explicit.

We fix an $S^1$-invariant connection on $F$ and the associated splitting
\be\label{eq:TF}
TF = HTF \oplus VTF
\ee
which is defined to be the direct sum of the connection of $T^*\CN$ and the $S^1$-invariant
connection on the symplectic vector bundle $(E,\Omega)$.
Then we express
$$
\dot z = \left(\begin{matrix} \widetilde{\dot x} \\
\nabla_t s \end{matrix}\right).
$$
Here we regard $\dot x$ as a $TQ$-valued one-form on $S^1$ and
$\nabla_t s$ is defined to be
$$
\nabla_{\dot x} s = (x^*\nabla)_{\frac{\del}{\del t}} s
$$
 which we regard as an element of $F_{x(t)}$. Through identification of
$H_s TF$ with $T_{\pi_F(s)} Q$ and $V_s TF$ with $F_{\pi(s)}$  or more precisely through the identity
$$
\widetilde{\dot x} \circ I_{s;x} = \dot x,
$$
we will just write
$$
\dot z = \left(\begin{matrix} \dot x \\
\nabla_t s \end{matrix}\right).
$$

Recall that $o_F$ is foliated by the closed Reeb orbits of $\lambda$ which also form
the fibers of the prequantization bundle $Q \to P$.

For a given Reeb orbit $z = (x,s)$, we denote $x(t) = \gamma(T\cdot t)$ where
$\gamma$ is a Reeb orbit of period $T$ of the contact form $\theta$ on $Q$ which is
nothing but a fiber of the prequantization $Q \to P$.
We then decompose
$$
D\Upsilon(z)(Z) = (D\Upsilon(z)(Z))^v + (D\Upsilon(z)(Z))^h.
$$
Then the assignment $Z \mapsto (D\Upsilon(z)(Z))^v$ defines an operator
from $\Gamma(z^*VTF)$ to $\Gamma(z^*VTF)$. We remark that since $VTF \subset \xi$,
we have
\be\label{eq:Dvertical}
D\Upsilon(z)(Z) = D^\pi\Upsilon(z)(Z)
\ee
for any vertical vector field $Z$.

Composing with the map $I_{z;x}$,
we have obtained an operator from $\Omega^0(x^*F)$ to $\Omega^0(x^*F)$.
We denote this operator by
\be\label{eq:Dupsilon}
D\upsilon(x): \Omega^0(x^*F) \to \Omega^0(x^*F).
\ee
Using $X_F = \widetilde X_{\lambda,Q}$ and $\nabla_Y X_F = 0$ for any vertical vector field $Y$,
we derive the following proposition from Lemma \ref{lem:DUpsilon}.
This will be important later for our exponential estimates.

\begin{prop}\label{prop:DupsilonE} Let $D\upsilon = D\upsilon(x)$ be the operator defined above.
Define the vertical Hamiltonian vector field $X^\Omega_g$ by
$$
X^\Omega_g \rfloor \Omega = dg|_{VTF}.
$$
Then
\be\label{eq:nablaeY}
D\upsilon = \nabla_t^F - T\, D^v X_g^{\Omega}(z)
\ee
where $z = (x,o_{x})$.
\end{prop}
\begin{proof}
Consider a vertical vector field $Z \in VTF$ along a Reeb orbit $z$ as above
and regard it as the section of $z^*F$ defined by
$$
s_Z(t) = I_{z;x}(Z(t))
$$
where $\gamma$ is a Reeb orbit with period $T$ of $X_{\lambda,Q}$ on $o_F\cong Q$.
Recall the formula
\bea\label{eq:DUpsilonE}
D\Upsilon(z)(Z) &= & D^\pi \Upsilon(z)(Z) \nonumber \\
& = & \nabla_t^\pi Z - T \left(\frac{1}{f}\nabla_Z X_{\lambda_F} + Z[1/f] X_{\lambda_F} \right)
- T \left(\frac{1}{f} \nabla_Z Y_{dg} + Z[1/f] Y_{dg}\right) \nonumber\\
&{}&
\eea
from \eqref{eq:Dvertical}, and Lemma \ref{lem:DUpsilon} which we apply to the vertical vector field $Z$ for
the contact manifold $(U_F,\lambda_F)$.

We recall $f \equiv 1$ on $o_F$ and $df \equiv 0$ on $TF|_{o_F}$.
Therefore we have $Z[1/f] = 0$. Furthermore recall $X_F = \widetilde X_{\lambda,Q}$ and
$$
\nabla_Z X_F = D^v X_F (s_Z) = D^v \widetilde X_\theta(s_Z) = 0
$$
on $o_F$. On the other hand, by definition, we derive
$$
 I_{z;x}\left(\frac{1}{f}\nabla_Z Y_{dg}\right) =
D^v X_g^\Omega (s_Z).
$$
By substituting this into \eqref{eq:DUpsilonE} and composing with $I_{z;x}$,
we have finished the proof.
\end{proof}

By construction, it follows that the vector field along $z$ defined by
$$
t \mapsto \phi_{X_\theta}^t(v), \quad t \in [0,1]
$$
for any $v \in T_{z(0)} Q$ lie in $\ker D\Upsilon(z)$.
By the Morse-Bott hypothesis, this set of vector fields exhausts
$
\ker D\Upsilon(z).
$
We denote by $\delta > 0$ the gap between $0$ and the first non-zero eigenvalue of
$D\Upsilon(z)$. Then we obtain the following

\begin{cor}\label{cor:gap} Let $z = (x_z, o_{x_z})$ be a Reeb orbit. Then
for any section $s \in \Omega^0(x^*F)$, we have
\be\label{eq:Dupsilon-gap}
\|\nabla_t^F s - T D^v X_g^{\Omega}(z)(s)\|^2 \geq \delta^2 \|s\|_2^2.
\ee
\end{cor}

This inequality plays a crucial role in the study of exponential convergence of
contact instantons in the Morse-Bott context studied later in the present paper.

\section{Normal coordinates of $dw$ in $(U_F,f \lambda_F)$}
\label{sec:coord}

We fix the splitting $TF = HTF \oplus VTF$ given in \eqref{eq:TF} and
consider the decomposition of $w = (u,s)$ according to the splitting.
For a given map $w: \dot \Sigma \to F$, we denote by $u := \pi_F \circ w$. Then
we can express
$$
w(z) = (u(z), s(z)), \quad z \in \Sigma
$$
where $s(z) \in F_{u(z)}$, i.e., $s$ is the section of $u^*F$.

We regard this decomposition as the map
$$
\CI: \CF(\Sigma, F) \to \CH^F_\Sigma
$$
where $\CH^F$ is the infinite dimensional vector bundle
$$
\CH^F_\Sigma = \bigcup_{u \in \CF(\Sigma,F)} \CH^F_{\Sigma,u}
$$
where $\CH^F_{\Sigma,u}$ is the vector space  given by
$$
\CH^F_{\Sigma,u} = \Omega^0(u^*F)
$$
the set of smooth sections of the pull-back vector bundle $u^*F$. This provides
a coordinate description of $\CF(\Sigma,F)$ in terms of $\CH^F_\Sigma$. We denote the corresponding
coordinates $w = (u_w,s_w)$ when we feel necessary to make the dependence of $(u,s)$ on $w$
explicit.

In terms of the splitting \eqref{eq:TF}, we express
$$
dw = \left(\begin{matrix} \widetilde{du} \\
\nabla_{du} s \end{matrix}\right).
$$
Here we regard $du$ as a $TQ$-valued one-form on $\dot\Sigma$ and
$\nabla_{du} s$ is defined to be
$$
\nabla_{du(\eta)} s(z) = (u^*\nabla)_{\eta} s
$$
for a tangent vector $\eta \in T_z\Sigma$, which we regard as an element of $F_{u(z)}$. Through identification of
$HTF_s$ with $T_{\pi_F(s)} Q$ and $VTF_s$ with $F_{\pi_F(s)}$  or more precisely through the identity
$$
I_{w;u}(\widetilde{du}) = du,
$$
we will just write
$$
dw = \left(\begin{matrix} du \\
\nabla_{du} s \end{matrix}\right)
$$
from now on, unless it is necessary to emphasize the fact that $dw$ a priori has values
in $TF = HTF \oplus VTF$, not $TQ \oplus F$.

To write them in terms of the coordinates $w = (u,s)$,  we first derive the
formula for the projection $d^\pi w = d^{\pi_{\lambda}}w$ with $\lambda = f\, \lambda_F$.
For this purpose, we recall the formula for $X_{f\lambda_F}$ from Proposition \ref{prop:eta} subsection
\ref{subsec:perturbed-forms}
$$
X_{f\lambda_F} = \frac{1}{f}(X_\lambda + Y_{dg}), \quad Y_{dg}: = \pi_{\lambda_F}(\flat_{\lambda_F}(dg))
$$
for $g = \log f$. We decompose
$$
Y_{dg} = (Y_{dg})^v + (Y_{dg})^h
$$
into the vertical and the horizontal components. This leads us to the decomposition
\be\label{eq:XflambdaF-decompo}
f\, X_{f\lambda_F}  = (Y_{dg})^h + X_{\lambda_F} + (Y_{dg})^v
\ee
in terms of the splitting
$$
TF = \widetilde{(\xi_\lambda\cap TQ)} \oplus \R\{X_F\} \oplus VTF, \quad HTF = \widetilde{(\xi_\lambda\cap TQ)} \oplus \R\{X_F\}.
$$
Recalling $d\lambda_F = \pi_F^*d\theta + \pi_{T^*\CN;F}^*d\Theta_G  + \pi_{E;F}^*\widetilde \Omega$, and since $d\Theta_G$ vanishes on $VTF$, we have derived

\begin{lem} At each $s \in F$,
\be\label{eq:Hamiltonian}
(Y_{dg})^v(s) = X_{g|_{F_{\pi(s)}}}^{\Omega^v(s)}.
\ee
\end{lem}

Now we are ready to derive an important formula that will play a
crucial role in our exponential estimates in later sections. Recalling the canonical
isomorphism
$$
I_{s;\pi_F (s)}; VTF_s \to F_{\pi(s)}
$$
we introduced in section 2, we define the following \emph{vertical derivative}

\begin{defn}\label{defn:vertical-derive}
Let $X$ be a vector field on $F \to Q$. The \emph{vertical derivative},
denoted by $D^v X: F \to F$ is the map defined by
\be\label{eq:DvX}
D^vX(q)(f): = \frac{d}{dr}\Big|_{r=0} I_{r f;\pi(r f)}(X^v(r f))
\ee
\end{defn}

\begin{prop} Let $(E,\Omega, J_E)$ be the Hermitian vector bundle
for $\Omega$ defined as before.
Let $g = \log f$ and $X_g^{d\lambda_E}$ be the contact Hamiltonian
vector field as above. Then we have
$$
J_E D^v Y_{dg} = \operatorname{Hess}^v g(q,o_q).
$$
In particular, $J_E D^v X_g^{d\lambda_E}: E \to E$
is a symmetric endomorphism with respect to the metric
$g_E = \Omega(\cdot, J_E\cdot)$.
\end{prop}
\begin{proof} Let $q \in Q$ and $e_1, \, e_2 \in E_q$. We compute
\beastar
\langle D^v Y_{dg}(q)e_1, e_2 \rangle & = & \Omega (D^v Y_{dg}(q) e_1,J_E e_2) \\
d\lambda_E(D^v Y_{dg}(q) e_1,J_E e_2) & = &
\Omega \left(\frac{d}{dr}\Big|_{r=0} I_{re_1;q}((Y_{dg})^v (re_1)), J_E e_2\right) \\
& = &  \Omega \left(\frac{d}{dr}\Big|_{r=0} I_{re_1;q}(X_g^{\Omega}(re_1)), J_E e_2\right).
\eeastar
Here $\frac{d}{dr}\Big|_{r=0} X_g^{\Omega}(re_1)$ is nothing but
$$
DX_g^{\Omega}(q)(e_1)
$$
where $DX_g^{\Omega}(q)$ is the linearization of the Hamiltonian vector field of $g|_{E_{q}}$ of the
symplectic inner product $\Omega(q)$ on $E_q$. Therefore it lies at the symplectic
Lie algebra $sp(\Omega)$ and so satisfies
\be\label{eq:DXg-symp}
\Omega(DX_g^{\Omega}(q)(e_1), e_2) + \Omega(e_1, DX_g^{\Omega}(q)(e_2)) = 0
\ee
which is equivalent to saying that $J_E DX_g^{\Omega}(q)$ is symmetric with
respect to the inner product $g_E = \Omega(\cdot, J_E \cdot)$. But we also have
$$
J_E DX_g^{\Omega}(q) = D \grad_{g_E(q)} g|_{E_q} = \operatorname{Hess}^v g(q).
$$

On the other hand, \eqref{eq:DXg-symp} also implies
$$
\Omega(DX_g^{\Omega}(q)(J_E e_1), e_2) - \Omega(DX_g^{\Omega}(q)(e_2), J_E e_1) = 0
$$
with $e_1$ replaced by $J e_1$ therein. The first term becomes
$$
\langle DX_g^{\Omega}(q)(e_1), e_2 \rangle
$$
and the second term can be written as
\beastar
\Omega(DX_g^{\Omega}(q)(J_E e_2), e_1)& = & - \Omega(J_E e_2, DX_g^{\Omega}(q)(e_1))\\
& = & \Omega(e_2, J_E DX_g^{\Omega}(q)(e_1)) \\
& = & \langle e_2,  DX_g^{\Omega}(q)(e_1)\rangle.
\eeastar
Combining the two, we have finished the proof.
\end{proof}

\section{$CR$-almost complex structures adapted to $Q$}
\label{sec:adapted}

\emph{We would like to emphasize that we have not involved any
almost complex structure yet. Now we involve $J$ in our discussion.}

Let $J$ be any $CR$-almost complex structure compatible to $\lambda$ in that
$(M,\lambda,J)$ defines a contact triad and denote by $g$ the triad metric.
Then we can realize the normal bundle $N_QM = T_QM /TQ$ as the metric normal bundle
$$
N^g_QM = \{ v \in T_QM \mid d\lambda(v, J w) = 0, \forall w \in TQ\}.
$$
We start with the following obvious lemma

\begin{lem} Consider the foliation $\CN$ of $(Q,\omega_Q)$, where $\omega_Q = i_Q^*d\lambda$.
Then $JT\CN$ is perpendicular to $TQ$ with respect to the triad metric of $(M,\lambda,J)$.
In particular $JT\CN \subset N^g_QM$.
\end{lem}
\begin{proof} The first statement follows from the property that
$T\CN$ is isotropic with respect to $\omega_Q$.
\end{proof}

Now, we introduce the concept of almost complex structures adapted to the locus $Q$ of closed Reeb orbits
of $M$.

\begin{defn}\label{defn:adapted} Let $Q \subset M$ be the locus of closed Reeb orbits of
Morse-Bott contact form $\lambda$.
Suppose $J$ defines a contact triad $(M,\lambda,J)$.
We say a $CR$-almost complex structure $J$ for $(M,\xi)$ is adapted to
the submanifold $Q$ if $J$ satisfies
\be\label{eq:JTQ}
J (TQ) \subset TQ + J T\CN.
\ee
\end{defn}

\begin{prop}\label{prop:adapted} The set of adapted $J$ relative to $Q$ is nonempty and is a contractible
infinite dimensional manifold.
\end{prop}
\begin{proof}
For the existence of a $J$ adapted to $Q$, we recall the splitting
\beastar
T_q Q & = & \R\{X_\lambda(q)\}\oplus T_q \CN \oplus G_q,\\
T_q M & \cong & (\R\{X_\lambda(q)\}\oplus T_q\CN \oplus G_q) \oplus(T_q^*\CN \oplus E_q) \\
& = & \R\{X_\lambda(q)\}\oplus (T_q\CN \oplus T_q^*\CN) \oplus G_q \oplus E_q
\eeastar
on each connected component of $Q$. Therefore we
can find $J$ so that it is compatible on $T\CN \oplus T^*\CN$ with respect to
$-d\Theta_G|_{T\CN \oplus T^*\CN}$, and compatible on $G$
with respect to $\omega_Q$ and on $E$ with respect to $\Omega$. It follows that
any such $J$ is adapted to $Q$. This proves the first statement.

The proof of the second statement will be postponed until Appendix.
\end{proof}

We note that each summand $T_q\CN \oplus T_q^*\CN$,
$G_q$ and $E_q$ in the above splitting of $T_QM$ is symplectic with respect to $d\lambda$.

We recall the embeddings $T^*\CN$ and $E$ into $N_QM$
and the identification $N_Q M \cong T^*\CN \oplus E$ discussed in Subsection
\ref{subsec:structure}.

\begin{lem}\label{lem:J-identify} For any adapted $J$, the identification of the normal bundle
\be\label{eq:NgQM}
N_Q M \to N^g_Q M; \quad [v] \mapsto \widetilde{d\lambda}(-J v)
\ee
naturally induces the following identifications:
\begin{enumerate}
\item
$
T^*\CN\cong JT\CN.
$
\item
$
\text{\rm Image}(E \hookrightarrow N_QM)= (TQ)^{d\lambda}\cap (JT\CN)^{d\lambda}.
$
\end{enumerate}
\end{lem}
\begin{proof}
(1) follows by looking at the metric $\langle\cdot, \cdot\rangle=d\lambda(\cdot, J\cdot)$.
Now restrict it to $(TQ)^{d\lambda}$,
we can identify $E$ with the complement of $T\CF$ with respect to this metric,
which is just $(TQ)^{d\lambda}\cap (JT\CN)^{d\lambda}$.
\end{proof}

\begin{lem}\label{lem:JE} For any adapted $J$, $J E\subset E$ in the sense of the identification of $E$ with
the subbundle of $T_QM$ given in the above lemma.
\end{lem}
\begin{proof}
Take $v\in (TQ)^{d\lambda}\cap (JT\CN)^{d\lambda}$,
then for any $w\in TQ$,
$$
d\lambda(Jv, w)=-d\lambda(v, Jw)=0
$$
since $JTQ\subset TQ + JT\CN$. Hence $Jv\in (TQ)^{d\lambda}$.

For any $w\in T\CN$,
$$
d\lambda(Jv, Jw)=d\lambda(v, w)=0
$$
since $v\in (TQ)^{d\lambda}$ and $w\in TQ$. Hence $Jv\in (JT\CN)^{d\lambda}$,
and we are done.
\end{proof}

\begin{rem}\label{rem:adapted}
\begin{enumerate}
\item
We would like to mention that in the nondegenerate case
the adaptedness is automatically satisfied by any compatible $CR$-almost complex structure
$J \in \CJ(M,\lambda)$, because in that case $P$ is a point and $HTF = \R \cdot \{X_F\}$ and
$VTF = TF = \xi_F$.
\item
However for the general Morse-Bott case, the set of adapted $CR$-almost complex
structure is strictly smaller than $\CJ(M,\lambda)$. It appears that for the proof of exponential
convergence result of closed Reeb orbits in the Morse-Bott case, this additional restriction of
$J$ to those adapted to the Morse-Bott submanifold of closed Reeb orbits in the above sense facilitates geometric computation
considerably. (Compare our computations with those given in \cite{bourgeois}, \cite{behwz}.)
\item When $T\CN = \{0\}$, $(Q,\lambda_Q)$ carries the structure of prequantization $Q \to P = Q/S^1$.
\end{enumerate}
\end{rem}

We specialize to the normal form $(U_F, f \lambda_F)$.
We note that the complex structure $J_F: \xi_F \to \xi_F$ canonically induces one on the vector bundle $VTF \to F$
$$
J_F^v; VTF \to VTF
$$
satisfying $(J_F^v)^2 = -id_{VTF}$. For any given $J$ adapted to $o_F \subset F$,
it has the decomposition
$$
J_{U_F}|_{o_F} = \left(\begin{matrix} \widetilde J_G & 0 & 0 & 0 \\
0 & 0 & I & D \\
C & -I & 0 & 0\\
0 & 0 & 0 & J_E
\end{matrix} \right)
$$
on the zero section with respect to the splitting
$$
TF|_{o_F} \cong \R\{X_F\} \oplus G \oplus (T\CN \oplus T^*\CN) \oplus E.
$$
Here we note that $C \in \Hom(G,T^*\CN)$ and $D \in \Hom(E,T\CN)$, which depend on
$J$. Indeed it is easy to see $C = 0 = D$ from a consideration of the equation $J_U^2 = -Id$.

Using the splitting
\beastar
TF & \cong & \R\{X_{\lambda_F}\} \oplus \widetilde G \oplus (\widetilde{T\CN} \oplus VT\CN^*) \oplus VTF \\
& \cong &
\R\{X_{\lambda_F}\} \oplus \widetilde G  \oplus (T\CN \oplus T^*\CN)\oplus E
\eeastar
on $o_F \cong Q$, we lift $J_F$ to a $\lambda_F$-compatible almost
complex structure on the total space $F$, which we denote by $J_0$.
We note that the triad $(F,\lambda_F,J_0)$ is naturally $S^1$-equivariant by the $S^1$-action induced by
the Reeb flow on $Q$.

\begin{defn}[Normalized Contact Triad $(F,\lambda_F,J_0)$]\label{defn:normaltriad}
We call the $S^1$-invariant contact triad $(F,\lambda_F,J_0)$ the normalized
contact triad adapted to $Q$.
\end{defn}

Now we are ready to give the proof of the following.

\begin{prop}\label{prop:connection} Consider the
contact triad $(U_F,\lambda_F,J_0)$ for an adapted $J$ and its associated
triad connection. Then the zero section $o_F \cong Q$ is
totally geodesic and so naturally induces an affine connection on $Q$.
Furthermore the induced connection on $Q$ preserves $T\CF$ and the splitting
$$
T\CF = \R\{X_{\lambda_F}\} \oplus T\CN.
$$
\end{prop}
\begin{proof}
We note that $\lambda_F$ is invariant under the reflection of
the vector bundle $F \to Q$ by definition \eqref{eq:lambdaF} of $\lambda_F$  and
so is $J_0$ by the construction given above. Therefore the triad metric of $(U_F,\lambda_F,J_0)$
is invariant under the reflection. This implies that the associated triad connection, which
preserves the triad metric by one of the definition properties \cite{oh-wang1},
makes the zero section totally geodesic since it is the fixed point set of the
reflection which is an isometry with respect to the triad metric.
Therefore it canonically restricts to
an affine connection on $o_F \cong Q$.

It remains to show that this connection preserves the splitting $T\CF = \R\{X_{\lambda,Q}\} \oplus T\CN$.
For the simplicity of notation, we denote $\lambda_F = \lambda$ in the rest of this proof.

Let $q \in Q$ and $v \in T_qQ$. We pick a vector field $Z$ that is
tangent to $Q$ and $S^1$-invariant and satisfies $Z(q) = v$. Such a vector field exists because
 $\CF$ is the null foliation of $\omega_Q = i_Q^*d\lambda$,
and $Q$ carries the $S^1$-action induced by the Reeb flow of $\lambda$.
If $Z$ is a multiple of $X_\lambda$, then we can choose $Z = c\, X_\lambda$ for some
constant and so $\nabla_Z X_{\lambda} = 0$ by the axiom $\nabla_{X_\lambda}X_\lambda = 0$
of contact triad connection. Then for $Y \in \xi \cap T\CF$, we compute
$$
\nabla_{X_\lambda} Y = \nabla_Y X_\lambda + [X_\lambda,Y] \in \xi
$$
by an axiom of the triad connection.  On the other hand, for any $Z$ tangent to $Q$, we derive
$$
d\lambda(\nabla_{X_\lambda} Y, Z) = -d\lambda(Y,\nabla_{X_\lambda}Z) = 0
$$
since $Q = o_F$ is totally geodesic and so $\nabla_{X_\lambda}Z \in TQ$. This proves that
$\nabla_{X_\lambda} T\CN \subset T\CN$.

For $v \in T_q Q \cap \xi_q$, we have
$
\nabla_Z {X_\lambda} \in \xi \cap TQ.
$
On the other hand,
$$
\nabla_Z {X_\lambda} = \nabla_{X_\lambda} Z + [Z,X_\lambda] = \nabla_{X_\lambda} Z
$$
since $[Z,X_\lambda]  = 0$ by the $S^1$-invariance of $Z$. Now let $W \in T\CN$ and compute
$$
\langle \nabla_Z {X_\lambda}, W \rangle = d\lambda(\nabla_Z {X_\lambda}, J_0 W).
$$
On the other hand $J_0 W \in T^*\CN \subset T_{o_F}F$ since $\CF$ is (maximally)
isotropic with respect to $d\lambda$. Therefore we obtain
$$
d\lambda( \nabla_Z {X_\lambda}, J_0 W) = -\pi_{T^*\CN;F}^*\Theta_G(q,0,0)(\nabla_Z {X_\lambda}, J_0 W) = 0.
$$
This proves $\nabla_Z {X_\lambda}$ is perpendicular to $T\CN$ with respect to the triad metric
and so must be parallel to $X_\lambda$. On the other hand,
$$
\langle \nabla_Z W, {X_\lambda} \rangle = - \langle \nabla_Z {X_\lambda}, W \rangle  = 0
$$
and hence if $W \in T\CN$, it must be perpendicular to $X_\lambda$. Furthermore we have
$$
d\lambda(\nabla_Z W, V) = -d\lambda(W,\nabla_Z V) =0
$$
for any $V$ tangent to $Q$ since $\nabla_Z V \in TQ$ as $Q$ is totally geodesic. This proves
$\nabla_ZW$ indeed lies in $\xi \cap T\CF = T\CN$, which finishes the proof.
\end{proof}

\part{Exponential estimates for contact instantons: Morse-Bott case}
\label{part:exp}

In this part, we develop the three-interval method of
proving $C^\infty$ exponential convergence to closed Reeb orbits
of any (charge vanishing) contact instanton with finite $\pi$-energy and bounded gradient  of
any Morse-Bott contact form, and use it  at each puncture of domain Riemann surface.

The contents of this part are as follows:
\begin{itemize}
\item In Section \ref{sec:pseudo}, we briefly review the subsequence convergence result
for  contact instantons with finite $\pi$-energy and bounded gradient. This is the starting point for
applying the three-interval method introduced in Section \ref{sec:three-interval} and afterwards;
\item In Section \ref{sec:three-interval}, an abstract three-interval method framework is presented;
\item In Section \ref{sec:prequantization}, we focus on the prequantization case and use
the three-interval machinery introduced in Section \ref{sec:three-interval} to prove exponential
convergence. The proof is divided into several steps which are organized into
different subsections.
\item In Section \ref{sec:general}, we prove exponential decay for general cases;
\item In Section \ref{sec:asymp-cylinder}, we explain how to apply this method to symplectic
manifolds with asymptotically cylindrical ends.

\end{itemize}

\section{Subsequence convergence on the adapted contact triad $(U_F,\lambda,J)$}
\label{sec:pseudo}

We first introduce the subsequence convergence result for Morse-Bott contact instantons of finite $\pi$-energy and finite gradient bound. The proofs are almost word-by-word the same as the  nondegenerate case considered in \cite{oh-wang2}. For readers' convenience, we include details here.

We fix a punctured Riemann surface  $(\dot\Sigma, j)$ with $l$-punctures and associate it with
a metric $h$ which is cylindrical at each end.
To be precise, it means that  there exists a compact set $K_\Sigma\subset \dot\Sigma$,
such that $\dot\Sigma-\text{\rm Int}(K_\Sigma)$ is the disjoint union of $l^+$-positive half cylinders and $l^-$-negative half cylinders with $l=l^++l^-$, i.e.
$$
\dot\Sigma-\text{\rm Int}(K_\Sigma)=(\sqcup_{i=1, \cdots, l^+}C^+_i)\sqcup (\sqcup_{i=1, \cdots, l^-}C^-_i),\quad l=l^++l^-
$$
where
$$
C^+_i=[0, \infty)\times S^1, \quad C^-_i=(-\infty, 0]\times S^1
$$
equipped with the cylindrical metric $\quad h|_{C^\pm_i}=d\tau^2+dt^2$ thereon.
For any smooth map
$
w: \dot\Sigma\to M,
$
the $\pi$-harmonic energy $E^\pi(w)$ is defined as
\be\label{eq:endenergy}
E^\pi(w):= E^\pi_{(\lambda,J;\dot\Sigma,h)}(w) = \frac{1}{2} \int_{\dot \Sigma} |d^\pi w|^2
\ee
where the norm is taken by $h$ and the triad metric on $M$.

We put the following hypotheses on the asymptotic study of Morse-Bott contact instantons:

\begin{hypo}\label{hypo:basic} Consider the contact triad $(M,\lambda,J)$ with
Morse-Bott contact form $\lambda$.
Let $w:\dot\Sigma\to M$ be a contact instanton, i.e., satisfy the contact instanton equations \eqref{eq:contact-instanton}
defined on a punctured Riemann surface with cylindrical ends $(\dot\Sigma, j, h)$. We assume that $w$ satisfies
\begin{enumerate}
\item $E^\pi(w):=E^\pi_{(\lambda,J;\dot\Sigma,h)}(w)<\infty$, i.e., finite $\pi$-energy;
\item $\|d w\|_{L^\infty(\dot\Sigma)} <\infty$, i.e., finite gradient bound.
\end{enumerate}
\end{hypo}

The following two asymptotic invariants associated to each puncture play essential roles in the study of asymptotic behavior of a Morse-Bott contact instanton satisfying Hypothesis \ref{hypo:basic}.

\begin{defn} Let $w: \dot \Sigma \to M$ be as in Hypothesis \ref{hypo:basic}.
At each puncture, we define the asymptotic contact action $\CT_{C^{\pm}_i}(w)$ and the asymptotic contact charge $\CQ_{C^{\pm}_i}(w)$ for a contact instanton $w$ satisfying Hypothesis \ref{hypo:basic} as
\bea
\CT_{C^{\pm}_i}(w) & := & \frac{1}{2}\int_{C^{\pm}_i} |d^\pi w|^2 + \int_{\{0\}\times S^1}(w|_{\{0\}\times S^1})^*\lambda\label{eq:TQ-T}\\
\CQ_{C^{\pm}_i}(w) & : = & \int_{\{0\}\times S^1}((w|_{\{0\}\times S^1})^*\lambda\circ j)\label{eq:TQ-Q}
\eea
where $C^\pm$ is the cylindrical end associated to the given puncture.
\end{defn}
The following remark shows that both $\CT$ and $\CQ$ are translation invariant.
\begin{rem}\label{rem:TQ}
For any contact instanton $w$ satisfying Hypothesis \ref{hypo:basic} at a puncture
$[0, \infty)\times S^1$, we have
$$
\CT(w) = \frac{1}{2}\int_{[s,\infty) \times S^1} |d^\pi w|^2 + \int_{\{s\}\times S^1}(w|_{\{s\}\times S^1})^*\lambda, \quad
\text{for any } s\geq 0,
$$
 which is due to $\frac{1}{2}|d^\pi w|^2\, dA=d(w^*\lambda)$ and  Stokes' formula;
and $$
\CQ(w)=\int_{\{s \}\times S^1}(w|_{\{s \}\times S^1})^*\lambda\circ j, \quad
\text{for any } s \geq 0,
$$
which is due to $d(w^*\lambda\circ j)=0$.
\end{rem}

Since our main interest lies on the asymptotic behavior of a fixed contact instanton $w$ at a given puncture,
we will assume the domain of $w$ is a positive half cylinder $[0, \infty)\times S^1$
without loss of generality. (The case of negative half cylinder can be treated in the same way.)
We simply denote the asymptotic contact action and charge at this puncture by $\CT$ and by $\CQ$
respectively.

\bigskip

\begin{thm}[Subsequence Convergence \cite{oh-wang2}]\label{thm:subsequence}
Let  $(M, \lambda, J)$ be any, not necessarily Morse-Bott, contact triad.
Assume  $w:[0, \infty)\times S^1\to M$ is a contact instanton, i.e. it satisfies the contact instanton equations \eqref{eq:contact-instanton}, and satisfies Hypothesis \ref{hypo:basic}.
Then for any sequence $s_k\to \infty$, there exists a subsequence, still denoted by $s_k$, and
a Reeb trajectory $\gamma$, not necessarily closed, such that
$$
\lim_{k\to \infty}w(s_k + \tau, t) = \gamma(-\CQ\tau+\CT t)
$$
in the $C^l(K \times S^1, M)$ sense for any $l\geq 0$, where $K\subset \R$ is an arbitrary compact set.

Furthermore when $(M, \lambda, J)$ is of Morse-Bott type and $w$ has
non-vanishing period $\CT\neq 0$, then there exists a connected  submanifold $Q$  foliated by closed
Reeb orbits of period $\CT$, so that the limit becomes $\gamma(-\CQ\tau+\CT\, t)$, where $z$ is a
closed Reeb orbit over $Q$.
\end{thm}
The first part of the theorem was proved in \cite[Section 6]{oh-wang2}. For reader's convenience,
we include its complete proof in Appendix \ref{appendix:subseqproof}.
Similar statement for Morse-Bott case in the context of symplectization
was proved in \cite[Proposition 2.1]{HWZ3} (see also \cite[Proposition 2.1]{HWZ1, HWZ2}).

\begin{cor}\label{cor:tangent-convergence}Assume  $w:[0, \infty)\times S^1\to M$ is a contact instanton, i.e. it satisfies the contact instanton equations \eqref{eq:contact-instanton} in a Morse-Bott contact triad $(M, \lambda, J)$, and satisfies
Hypothesis \ref{hypo:basic}.
Then
\beastar
&&\lim_{s\to \infty}\left|\pi \frac{\del w}{\del\tau}(s+\tau, t)\right|=0, \quad
\lim_{s\to \infty}\left|\pi \frac{\del w}{\del t}(s+\tau, t)\right|=0\\
&&\lim_{s\to \infty}\lambda(\frac{\del w}{\del\tau})(s+\tau, t)=-\CQ, \quad
\lim_{s\to \infty}\lambda(\frac{\del w}{\del t})(s+\tau, t)=\CT
\eeastar
and
$$
\lim_{s\to \infty}|\nabla^l dw(s+\tau, t)|=0 \quad \text{for any}\quad l\geq 1.
$$
All the limits are uniform for $(\tau, t)$ in $K\times S^1$ with compact $K\subset \R$.
\end{cor}

\bigskip

From now on, we consider $J$ as a $CR$-almost complex structure adapted to $Q$,
which in turn induces a $CR$-almost complex structures on a neighborhood $U_F$
of the zero section of $F$.
Denote by $(U_F,\lambda,J)$ the corresponding adapted contact triad.

When restricted to each connected component of the loci of closed Reeb orbits, there exists a uniform constant
$\tau_0>0$ such that the image of $w$ lies in a tubular neighborhood of $Q$ whenever $\tau>\tau_0$.
In other words, it is enough to restrict ourselves to study contact instanton maps from half cylinder $[0, \infty)\times S^1$
to the canonical neighborhood  $(U_F, \lambda, J)$ defined in Definition \ref{defn:normalneighborhood}
for the purpose of the study of asymptotic behaviour at the end.
\medskip

With the normal form
we developed in Part \ref{part:coordinate}, we express
$w$ as $w=(u, s)$
where $u:=\pi\circ w:[0, \infty)\times S^1\to Q$ and
$s=(\mu, e)$ is a section of the pull-back bundle $u^*(JT\CN)\oplus u^*E\to [0, \infty)\times S^1$.
Recall from Section \ref{sec:coord} and express
$$
dw = \left(\begin{matrix}du\\
\nabla_{du} s \end{matrix}\right)
=\left(\begin{matrix}du\\
\nabla_{du} \mu\\
\nabla_{du} e \end{matrix}\right),
$$
we reinterpret the convergence of $w$ stated in Theorem \ref{thm:subsequence}
in terms of the coordinate $w = (u,s)=(u, (\mu, e))$.

\begin{cor}\label{cor:convergence-ue} Let $w = (u,s)=(u, (\mu, e))$ satisfy the same assumption as in Theorem \ref{thm:subsequence}. Then for any sequence $s_k\to \infty$, there exists a subsequence, still denoted by $s_k$,
and a Reeb orbit $\gamma$ on $Q$ (may depend on the choice of subsequences) with action $\CT$
and charge $\CQ$, such that
$$
\lim_{k\to \infty} u(\tau+s_k, t)= \gamma(-\CQ\, \tau + \CT\, t)
$$
in $C^l(K \times S^1, M)$ sense for any $l$, where $K\subset [0,\infty)$ is an arbitrary compact set.
Furthermore,
we have
\bea
\lim_{s\to \infty}\left|\mu(s+\tau, t)\right|=0, &{}& \quad \lim_{s\to \infty}\left|e(s+\tau, t)\right|=0\nonumber\\
\lim_{s \to \infty} \left|d^{\pi_{\lambda}} u(s+\tau, t)\right|= 0, &{}& \quad
\lim_{s \to \infty} u^*\theta(s+\tau, t) = -\CQ d\tau + \CT\, dt\nonumber\\
\lim_{s \to \infty} \left|\nabla_{du} e(s+\tau, t)\right|  =  0,&{}&\nonumber
\eea
and
\bea
\lim_{s \to \infty} \left|\nabla^k d^{\pi_{\lambda}} u(s+\tau, t)\right| = 0, &{}& \quad
\lim_{s \to \infty} \left|\nabla^k u^*\theta(s+\tau, t)\right| =0\nonumber\\
\lim_{s \to \infty} \left|\nabla_{du}^k e(s+\tau, t)\right| = 0 &{}&\label{eq:highere}
\eea
for all $k \geq 1$, and all the limits are uniform for $(\tau, t)$ on $K\times S^1$ with compact $K\subset [0,\infty)$.
\end{cor}

In particular,
$$
\lim_{s \to \infty} du(s+\tau, t) = (-\CQ\, d\tau +\CT\, dt)\otimes X_\theta
$$
uniformly for $(\tau, t)$  in $C^\infty$ topology on
$K\times S^1$ for any given compact $K\subset [0,\infty)$.

\bigskip

In the rest of the present paper, we will add the following technical assumption of vanishing charge.

\begin{hypo}[Charge vanishing]\label{hypo:exact}
\be\label{eq:asymp-a=0}
\CQ:=\int_{\{0\}\times S^1}((w|_{\{0\}\times S^1})^*\lambda\circ j)=0.
\ee
\end{hypo}

Then the  uniform convergence proved in this section ensures all the basic requirements (including
the uniformly local tameness,
pre-compactness, uniformly local coercive property and the locally asymptotically cylindrical
property) of applying the three-interval method
to prove exponential decay of $w$ at the end which we will introduce in details in following  sections.

\section{Abstract framework of the three-interval method}
\label{sec:three-interval}

In this section, we introduce a new method in proving exponential decay using the abstract framework of the three-interval
method. In later Section \ref{sec:expdecayMB},   we will apply the scheme
to the normal bundle part.
We remark that the method can deal with the case with an exponentially decaying perturbation too (see Theorem \ref{thm:three-interval}).

The three-interval method is based on the following analytic lemma.
\begin{lem}[\cite{mundet-tian} Lemma 9.4]\label{lem:three-interval}
For a sequence of nonnegative numbers $\{x_k\}_{k=0, 1, \cdots, N}$, if there exists some constant $0<\gamma<\frac{1}{2}$ such that
$$
x_k\leq \gamma(x_{k-1}+x_{k+1})
$$
 for every $1\leq k\leq N-1$, then it follows
$$
x_k\leq x_0\xi^{-k}+x_N\xi^{-(N-k)},\quad k=0, 1, \cdots, N,
$$
where $\xi:=\frac{1+\sqrt{1-4\gamma^2}}{2\gamma}$.
\end{lem}

\begin{rem}\label{rem:three-interval}
\begin{enumerate}
\item If we write $\gamma=\gamma(c):=\frac{1}{e^c+e^{-c}}$ where $c>0$ is uniquely determined by $\gamma$, then the conclusion
can be written into the exponential form
$$
x_k\leq x_0e^{-ck}+x_Ne^{-c(N-k)}.
$$
\item
For an infinite nonnegative sequence $\{x_k\}_{k=0, 1, \cdots}$, if we have a uniform bound of
in addition,
then the exponential decay follows as
$$
x_k\leq x_0e^{-ck}.
$$
\end{enumerate}
\end{rem}

The analysis of proving the exponential decay will be carried on a Banach bundle $\CE\to [0, \infty)$ modelled by the Banach space $\E$, for which we mean every fiber $\E_\tau$ is identified with the Banach space $\E$ smoothly depending on $\tau$. We omit this identification if there is no way of confusion.

First we emphasize the base $[0,\infty)$ is non-compact and
carries a natural translation map for any positive number $r$,
which is $\sigma_r: \tau \mapsto \tau + r$.
We introduce the following definition
which ensures us to study the sections in local trivialization after taking a subsequence.

\begin{defn}\label{def:unif-local-tame} Let $\CE$ be a Banach bundle modelled with a Banach space $\E$ over $[0, \infty)$.
Let $[a,b] \subset [0,\infty)$ be any given bounded interval and let
$s_k \to \infty$ be any given sequence. A \emph{tame family of trivialization}
over $[a,b]$ relative to the sequence $s_k$ is defined to be
a sequence of trivializations $\{\Phi_k\}:\CE|_{[a,b]} \to [a,b] \times \E$
$$
\Phi_k: \sigma^*_{s_\cdot}\CE|_{[a+s_k, b+s_k]} \to [a,b] \times \E
$$
for $k\geq 0$ satisfying the following: There exists a sufficiently large $k_0 > 0$
such that for any $k \geq k_0$ the bundle map
$$
\Phi_{k_0+k} \circ \Phi_{k_0}^{-1}: [a,b] \times \E \to [a,b] \times \E
$$
satisfies
\be\label{eq:locallytame}
\|\nabla_\tau^l(\Phi_{k_0+k} \circ \Phi_{k_0}^{-1})\|_{\CL(\E,\E)} \leq C_l<\infty
\ee
for constants $C_l = C_l(|b-a|)$ depending only on $|b-a|$, $l=0, 1, \cdots$.

We call $\CE$ \emph{uniformly locally tame}, if it carries a tame family of
trivializations over $[a,b]$ relative to the sequence $s_k$ for any given bounded
interval $[a,b] \subset [0,\infty)$ and a sequence $s_k \to \infty$.
\end{defn}

\begin{defn} Suppose $\CE$ is uniformly locally tame. We say a connection $\nabla$ on
$\CE$ is \emph{uniformly locally tame} if the push-forward $(\Phi_k)_*\nabla_\tau$ can be written as
$$
(\Phi_k)_*\nabla_\tau = \frac{d}{d\tau} + \Gamma_k(\tau)
$$
for any tame family $\{\Phi_k\}$ so that
$\sup_{\tau \in [a,b]}\|\Gamma_k(\tau)\|_{\CL(\E,\E)} < C$ for some $C> 0$ independent of $k$'s.
\end{defn}

\begin{defn}
Consider a pair $\CE_2 \subset \CE_1$ of uniformly locally tame bundles, and a bundle map
$B: \CE_2 \to \CE_1$.
We say $B$ is \emph{uniformly locally bounded}, if for any compact set $[a,b] \subset [0,\infty)$ and
any sequence $s_k \to \infty$, there exists a subsequence, still denoted by $s_k$, a sufficiently large $k_0 > 0$ and tame families
$\Phi_{1,k}$, $\Phi_{2,k}$ such that for any $k\geq 0$
\be\label{eq:loc-uni-bdd}
\sup_{\tau \in [a,b]} \|\Phi_{i,k_0+k} \circ B \circ \Phi_{i,k_0}^{-1}\|_{\CL(\E_2, \E_1)} \leq C
\ee
where $C$ is independent of $k$.
\end{defn}

For a given locally tame pair $\CE_2 \subset \CE_1$, we denote by $\CL(\CE_2, \CE_1)$ the set
of bundle homomorphisms which are uniformly locally bounded.

\begin{lem}
If $\CE_1, \, \CE_2$ are uniformly locally tame, then so is $\CL(\CE_2, \CE_1)$.
\end{lem}

\begin{defn}\label{defn:precompact} Let $\CE_2 \subset \CE_1$ be as above and let $B \in \CL(\CE_2, \CE_1)$.
We say $B$ is \emph{pre-compact} on $[0,\infty)$ if for any locally tame families $\Phi_1, \Phi_2$,
there exists a further subsequence such that
$
\Phi_{1, k_0+k} \circ B \circ \Phi_{1, k_0}^{-1}
$
converges to some $B_{\Phi_1\Phi_2;\infty}\in \CL(\Gamma([a, b]\times \E_2), \Gamma([a, b]\times \E_1))$.
\end{defn}

Assume $B$ is a bundle map from $\CE_2$ to $\CE_1$ which is uniformly locally bounded,
where $\CE_1 \supset \CE_2$ are uniformly locally tame with tame families
$\Phi_{1,k}$, $\Phi_{2,k}$. We can write
$$
\Phi_{2,k_0+k} \circ (\nabla_\tau + B) \circ \Phi_{1,k_0}^{-1} = \frac{\del}{\del \tau} + B_{\Phi_1\Phi_2, k}
$$
as a linear map from $\Gamma([a,b]\times \E_2)$ to $\Gamma([a,b]\times \E_1)$, since $\nabla$ is uniformly
locally tame.

Next we introduce the following notion of coerciveness.

\begin{defn}\label{defn:localcoercive}
Let $\CE_1, \, \CE_2$ be as above and $B: \CE_2 \to \CE_1$ be a uniformly locally bounded bundle map.
We say the operator
$$
\nabla_\tau + B: \Gamma(\CE_2) \to \Gamma(\CE_1)
$$
is \emph{uniformly locally coercive}, if the following holds:
\begin{enumerate}
\item For any pair of bounded closed intervals $I, \, I'$ with $I \subset \operatorname{Int}I'$,
\be\label{eq:coercive}
\|\zeta\|_{L^2(I,\CE_2)} \leq C(I,I') (\|\nabla_\tau \zeta + B \zeta\|_{L^2(I',\CE_1)} + \| \zeta\|_{L^2(I',\CE_1)})
\ee
for a constant $C(I, I')$ depending only on $I, \, I'$ but independent of $\zeta$.
\item if for given bounded sequence $\zeta_k \in \Gamma(\CE_2)$ satisfying
$$
\nabla_\tau \zeta_k + B \zeta_k = L_k
$$
with $|L_k(\tau|_{\CE_1}$ bounded on a given compact subset $K \subset [0,\infty)$,
there exists a subsequence, still denoted by $\zeta_k$, that uniformly converges in $\CE_2$.
\end{enumerate}
\end{defn}

\begin{rem}
Let $E \to [0,\infty)\times S$ be a (finite dimensional) vector bundle and denote by
$W^{k,2}(E)$ the set of $W^{k,2}$-section of $E$ and $L^2(E)$ the set of $L^2$-sections. Let
$D: L^2(E)\to L^2(E)$ be a first order elliptic operator with cylindrical end.
Denote by $i_\tau: S \to [0,\infty)\times S$ the natural inclusion map. Then
there is a natural pair of Banach bundles $\CE_2 \subset \CE_1$ over $[0, \infty)$ associated to $E$, whose fiber
is given by $\CE_{1,\tau}=L^2(i_\tau^*E)$, $\CE_{2,\tau} = W^{1,2}(i_\tau^*E)$.
Furthermore assume $\CE_i$ for $i=1, \, 2$ is uniformly local tame if $S$ is a compact manifold (without boundary).
Then $D$ is uniformly locally coercive, which follows from the elliptic bootstrapping and the Sobolev's embedding.
\end{rem}

Finally we introduce the notion of asymptotically cylindrical operator $B$.
\begin{defn}\label{defn:asympcylinderical} We call $B$ \emph{locally asymptotically cylindrical} if the following holds:
Any subsequence limit $B_{\Phi_1\Phi_2;\infty}$ appearing in Definition \ref{defn:precompact}
is a \emph{constant} section,
and $\|B_{\Phi_1\Phi_2, k}-\Phi_{2,k_0+k} \circ  B \circ \Phi_{1,k_0}^{-1}\|_{\CL(\E_i, \E_i)}$ converges to zero
as $k\to \infty$ for both $i =1, 2$.
\end{defn}

Now we specialize to the case of Hilbert bundles $\CE_2 \subset \CE_1$ over $[0,\infty)$ and assume that
$\CE_1$ carries a connection {which is compatible with the Hilbert inner product of $\CE_1$. We denote by $\nabla_\tau$ the associated covariant
derivative. We assume that $\nabla_\tau$ is uniformly locally tame.

Denote by $L^2([a,b];\CE_i)$ the space of $L^2$-sections $\zeta$ of $\CE_i$ over
$[a,b]$, i.e., those satisfying
$$
\int_a^b |\zeta(\tau)|_{\CE_i}^2\, dt < \infty.
$$
where $|\zeta(\tau)|_{\CE_i}$ is the norm with respect to the given Hilbert bundle structure of $\CE_i$.

\begin{thm}[Three-Interval Method]\label{thm:three-interval}
Assume $\CE_2\subset\CE_1$ is a pair of Hilbert bundles over $[0, \infty)$ with fibers $\E_2$ and $\E_1$,
and $\E_2\subset \E_1$ is dense.
Let $B$ be a section of the associated bundle $\CL(\CE_2, \CE_1)$ and
$L \in \Gamma(\CE_1)$.
We assume the following:
\begin{enumerate}
\item There exists a covariant derivative $\nabla_\tau$ 
that preserves the Hilbert structure;
\item  $\CE_i$ for $i=1, \, 2$ are uniformly locally tame;
\item  $B$ is precompact, uniformly locally coercive and asymptotically cylindrical;
\item  Every subsequence limit $B_{\infty}$ is a self-adjoint unbounded operator on
$\E_1$ with its domain $\E_2$, and satisfies $\ker B_{\infty} = \{0\}$;
\item There exists some positive number $\delta$
such that any subsequence limiting operator $B_{\infty}$
of the above mentioned pre-compact family has all their eigenvalues $\lambda$
satisfying $|\lambda| >\delta$;
\item
There exists some $R_0 > 0$, $C_0 > 0$ and $\delta_0 > \delta$ such that
$$
|L(\tau)|_{\CE_{1,\tau}} \leq C_0 e^{-\delta_0 \tau}
$$
for all $\tau \geq R_0$.
\end{enumerate}
Then for any (smooth) section $\zeta \in \Gamma(\CE_2)$ with
\be\label{eq:sup-bound}
\sup_{\tau \in [R_0,\infty)} |\zeta(\tau,\cdot)|_{\CE_{2,\tau}} < \infty
\ee
and satisfying the equation
\be\label{eq:nabla=L}
\nabla_\tau \zeta + B(\tau) \zeta(\tau) = L(\tau),
\ee
there exist some constants $R$, $C>0$ such that for any $\tau>R$,
$$
|\zeta(\tau)|_{\CE_{1,\tau}}\leq C e^{-\delta \tau }.
$$
\end{thm}

\begin{proof}
We divide $[0, \infty)$ into the union of unit intervals $I_k:=[k, k+1]$ for
$k=0, 1, \cdots$.
We first prove the exponential decay of $\|\zeta\|^2_{L^2(I_k;\CE_1)}$ to zero as $k\to \infty$.
By Lemma \ref{lem:three-interval} and Remark \ref{rem:three-interval}, it is enough to prove that
for the function $\gamma(c) = \frac{1}{e^c + e^{-c}}$ as in Remark \ref{rem:three-interval} we have
\be
\|\zeta\|^2_{L^2(I_k;\CE_1)}\leq \gamma(2\delta)(\|\zeta\|^2_{L^2(I_{k-1};\CE_1)}+\|\zeta\|^2_{L^2(I_{k+1};\CE_1)}),\label{eq:3interval-ineq}
\ee
for every $k=1, 2, \cdots$ for some choice of $0 < \delta < 1$.
For the simplicity of notation and also because we use only the norms
$\|\zeta\|_{L^2([a,b];\CE_1}$ or $L^\infty([a,b];\CE_1)$ but for $\zeta \in \CE_2$
in the discussion below, we will just denote
$$
L^2([a,b]) := L^2([a,b];\CE_1), \quad L^\infty([a,b]): = L^\infty([a,b];\CE_1)
$$
for any given interval $[a,b]$.

If the inequality \eqref{eq:3interval-ineq} does not hold for every $k$,
we collect all the $k$'s that reverse the direction of the inequality.
If such $k$'s are finitely many, i.e., \eqref{eq:3interval-ineq} holds after some large $k_0$,
then we will still get the exponential estimate
as the theorem claims.

Otherwise, there are infinitely many such three-intervals, which we enumerate
by $I^{l_k}_{I}:=[l_k, l_k+1], \,I^{l_k}_{II}:=[l_k+1, l_k+2], \,I^{l_k}_{III}:=[l_k+2, l_k+3]$, $k=1, 2, \cdots$, such that
\be
\| \zeta\|^2_{L^2(I^k_{II}) }>
\gamma(2\delta)(\| \zeta\|^2_{L^2(I^k_{I}) }+\|\zeta\|^2_{L^2(I^k_{III}) }).\label{eq:against-3interval}
\ee
Before we deal with this case, we first remark that
this hypothesis in particular implies $ \zeta \not \equiv 0$ on
$I^{l_k}:=I^{l_k}_{I}\cup I^{l_k}_{II}\cup I^{l_k}_{III}$, i.e.,
$\|\zeta\|_{L^\infty(I^{l_k}) }\neq 0$.

\emph{If there exists some uniform constant $C_1>0$} such that
on each such three-intervals
\be
\|\zeta\|_{L^\infty(I^{l_k}) }<C_1e^{-\delta l_k}, \label{eq:expas-zeta}
\ee
it follows that
\be
\|\zeta\|_{L^2([l_k+1, l_k+2]) } \leq C_1e^{-\delta l_k}=C e^{-\delta(l_k+1)}.\label{eq:expest-zeta}
\ee
Here $C=C_1e^{-\delta}$ is purely a constant depending only on $\delta$ which will be determined at the end.
From now on, various constants $C$ appearing below may vary but be independent of $\zeta$.

Recall that under our assumption
we have infinitely many intervals that satisfy \eqref{eq:3interval-ineq},
and the exponential inequality \eqref{eq:expas-zeta}.
If the union of such intervals $[l_k+1,l_k+2]$ is connected after some point, then we already get our conclusion
 from \eqref{eq:expest-zeta} since every interval becomes a middle interval (of the form $[l_k+1, l_k+2]$)
in such three-intervals.

Otherwise the set of $k$'s that satisfy \eqref{eq:3interval-ineq} form a sequence of clusters,
$$
I^{l_k+1}, I^{l_k+2}, \cdots, I^{l_k+N_k}
$$
for the sequence $l_1, \, l_2, \cdots, l_k, \cdots $ such that $l_{k+1} > l_k +N_k$ and \eqref{eq:3interval-ineq}
holds on each element contained in each cluster.

We remark that each cluster has the farthest left interval $[l_k+1, l_k+2]$ as the middle interval in $I^{l_k}$,
and the farthest right interval $[l_k+N+2, l_k+N+3]$ as the middle interval in $I^{l_{k+1}}$.
(See Figure \ref{fig:three-interval}.)

\begin{figure}[h]
\setlength{\unitlength}{0.37in}
\centering 
\begin{picture}(32,6) 
\put(2,5){\line(1,0){9}}
\put(2,5){\line(0,1){0.1}}
\put(3,5){\line(0,1){0.1}}
\put(4,5){\line(0,1){0.1}}
\put(5,5){\line(0,1){0.1}}
\put(8,5){\line(0,1){0.1}}
\put(9,5){\line(0,1){0.1}}
\put(10,5){\line(0,1){0.1}}
\put(11,5){\line(0,1){0.1}}
\put(1,3){\red{\line(1,0){3}}}
\put(9,3){\red{\line(1,0){3}}}
\put(1,3){\red{\line(0,1){0.1}}}
\put(2,3){\red{\line(0,1){0.1}}}
\put(3,3){\red{\line(0,1){0.1}}}
\put(4,3){\red{\line(0,1){0.1}}}
\put(9,3){\red{\line(0,1){0.1}}}
\put(10,3){\red{\line(0,1){0.1}}}
\put(11,3){\red{\line(0,1){0.1}}}
\put(12,3){\red{\line(0,1){0.1}}}
\put(2.5, 4.8){\vector(0,-1){1.3}}
\put(10.5, 4.8){\vector(0,-1){1.3}}
\put(2, 5.2){$\overbrace{\qquad\qquad\qquad\qquad}^{I_{l_k+1}}$}
\put(8, 5.2){$\overbrace{\qquad\qquad\qquad\qquad}^{I_{l_k+N}}$}
\put(1, 2.8){$\underbrace{\qquad\qquad\qquad\qquad}_{I^{l_k}}$}
\put(9, 2.8){$\underbrace{\qquad\qquad\qquad\qquad}_{I_{l_{k+1}}=I_{l_k+N+1}}$}
\put(7.5, 4.7){$l_k+N$}
\put(1.5, 4.7){$l_k+1$}
\put(0.9, 3.2){$l_k$}
\put(8.8, 3.2){$l_{k+1}(=l_k+N+1)$}
\put(2,5){\circle*{0.07}}
\put(1,3){\red{\circle*{0.07}}}
\put(8,5){\circle*{0.07}}
\put(9,3){\red{\circle*{0.07}}}
\put(6,5.3){$\cdots\cdots$}
\put(2,5.02){\line(1,0){1}}
\put(10,5.02){\line(1,0){1}}

\put(2,3.02){\red{\line(1,0){1}}}
\put(10,3.02){\red{\line(1,0){1}}}

\put(2,5.03){\line(1,0){1}}
\put(10,5.03){\line(1,0){1}}

\put(2,3.03){\red{\line(1,0){1}}}
\put(10,3.03){\red{\line(1,0){1}}}

\put(2,1){\line(1,0){1}}
\put(3.5,1){\text{denotes the unit intervals that satisfy \eqref{eq:3interval-ineq}}}
\put(2,1){\line(0,1){0.1}}
\put(3,1){\line(0,1){0.1}}
\put(2,0){\red{\line(1,0){1}}}
\put(3.5,0){\text{denotes the unit intervals that satisfy \eqref{eq:against-3interval} and \eqref{eq:expas-zeta}}}
\put(2,0){\red{\line(0,1){0.1}}}
\put(3,0){\red{\line(0,1){0.1}}}
\end{picture}
\caption{} 
\label{fig:three-interval} 
\end{figure}

Then from \eqref{eq:expest-zeta}, we derive
\beastar
\|\zeta\|_{L^2([l_k+1, l_k+2]) } &\leq& Ce^{-\delta l_k}, \\
\|\zeta\|_{L^2([l_k+N+2, l_k+N+3]) } &\leq& Ce^{-\delta l_{k+1}}=Ce^{-\delta (l_k+N+1)}.
\eeastar

Combining them and Lemma \ref{lem:three-interval}, we get the following estimate for $l_k+1\leq l\leq l_k+N+2$,
\beastar
&{}&\|\zeta\|_{L^2([l, l+1]) }\\
&\leq&\|\zeta\|_{L^2([l_k+1, l_k+2) } e^{-\delta (l-(l_k+1))}
+\|\zeta\|_{L^2([l_k+N+2, l_k+N+3]) }e^{-\delta (l_k+N+2-l)}\\
&\leq& Ce^{-\delta l_k}e^{-\delta (l-(l_k+1))}+Ce^{-\delta (l_k+N+1)}e^{-\delta (l_k+N+2-l)}\\
&=&Ce^{\delta}(e^{-\delta l}+e^{-\delta(2l_k+2N+4-l)})
\leq (2Ce^{\delta})e^{-\delta l}.
\eeastar
Thus on each such cluster, we have exponential decay with the presumed rate $\delta$ as claimed in the theorem.

\emph{Now if there is no such uniform $C=C_1$ for which \eqref{eq:expas-zeta} holds},
then we can find a sequence of constants $C_k\to \infty$
and a subsequence of such three-intervals $\{I^{l_k}\}$, still denoted by $l_k$, such that
\be
\|\zeta\|_{L^\infty(I^{l_k}) }\geq C_ke^{-\delta l_k}.\label{eq:decay-fail-zeta}
\ee
We can further choose a subsequence, but still denoted by $l_k$,
so that $l_k+3<l_{k+1}$, i.e., the intervals do not intersect one another.

We translate the sections
$
\zeta_k:=\zeta|_{[l_k, l_k+3]}
$
and consider the sections ${\widetilde \zeta}_k$ defined on $[0,3]$ given by
$$
{\widetilde \zeta}_k(\tau, \cdot):=\zeta(\tau + l_k, \cdot).
$$
Then \eqref{eq:decay-fail-zeta} becomes
\be
\|{\widetilde\zeta}_k\|_{L^\infty([0,3]) }\geq C_ke^{-\delta l_k}.\label{eq:decay-fail-zetak}
\ee
If we consider the translations of $L$ given by $\widetilde L_k(\tau, t)=L(\tau+l_k, t)$,
then
\be
|\widetilde L_k(\tau, t)|<Ce^{-\delta l_k}e^{-\delta\tau}\leq Ce^{-\delta l_k}\label{eq:Lk}
\ee
for $\tau \geq 0$. It follows that ${\widetilde \zeta}_k$ satisfies the equation
\be
\nabla_\tau \widetilde \zeta_k + B(\tau+l_k,\cdot) \widetilde \zeta_k = \widetilde{L}_k(\tau, t).\label{eq:uktilde-zeta}
\ee

We now rescale \eqref{eq:uktilde-zeta} by dividing it by $\|{\widetilde\zeta}_k\|_{L^\infty([0,3])}$, which can not vanish
by the standing hypothesis as we remarked right below \eqref{eq:against-3interval}, and
consider the rescaled sequence
$$
\overline \zeta_k:={\widetilde\zeta}_k/\|{\widetilde\zeta}_k\|_{L^\infty([0,3])}.
$$
We have now
\bea
\|\overline \zeta_k\|_{{L^\infty([0,3])}}&=&1\nonumber\\
\nabla_\tau\overline\zeta_k+ B(\tau+l_k,t) \overline\zeta_k
&=&\frac{\widetilde{L}_k}{\|{\widetilde\zeta}_k\|_{L^\infty([0,3] )}}\label{eq:nablabarzeta}\\
\|\overline \zeta_k\|^2_{L^2([1,2] )}&\geq&\gamma(2\delta)(\|\overline \zeta_k\|^2_{L^2([0,1] )}+\|\overline \zeta_k\|^2_{L^2([2,3] )}).\nonumber
\eea
From \eqref{eq:decay-fail-zetak} and \eqref{eq:Lk}, we get
\beastar
\frac{\|\widetilde{L}_k\|_{L^\infty([0, 3])}}{\|\zeta_k\|_{L^\infty([0, 3] )}} \leq \frac{C}{C_k},
\eeastar
and then by our assumption that $C_k\to \infty$,
we prove that the right hand side of \eqref{eq:nablabarzeta} converges to zero as $k\to \infty$.

Since $B$ is assumed to be pre-compact, we
get a limiting operator $B_{\infty}$ after taking a subsequence (in a trivialization).

On the other hand, since $B$ is locally coercive,
there exists $\overline\zeta_\infty$ such that
$\bar\zeta_k \to
\overline\zeta _\infty$ uniformly converges in $\CE_2$ and $\overline\zeta _\infty$ satisfies
\be\label{eq:xibarinfty-zeta}
\nabla_\tau \overline\zeta_\infty+ B_\infty \overline\zeta_\infty=0 \quad \text{on $[0,3]$},
\ee
and
\be\label{eq:xibarinfty-ineq-zeta}
\|\overline \zeta_\infty\|^2_{L^2([1,2] )}\geq\gamma(2\delta)(\|\overline \zeta_\infty\|^2_{L^2([0,1] )}
+\|\overline \zeta_\infty\|^2_{L^2([2,3] )}).
\ee

Since $\|\overline\zeta_\infty\|_{L^\infty([0,3]\times S^1)} = 1$,
$\overline\zeta_\infty \not \equiv 0$. Recall that $B_\infty$ is assumed to be a
(unbounded) self-adjoint operator on $\E_1$ with its domain $\E_2$.
Let $\{e_i\}$ be its orthonormal eigen-basis of $\E_1$ with respect to $B_\infty$.
We consider the eigen-function expansion of
$\overline \zeta_\infty(\tau,\cdot)$ and write
$$
\overline \zeta_\infty(\tau) = \sum_{i=1}^\infty a_i(\tau)\, e_i
$$
for each $\tau \in [0,3]$, where $e_i$ are the eigen-functions of $B$ associated to the eigenvalue $\lambda_i$ with
$$
-\infty < \ldots \leq \lambda_{-k} \leq \lambda_{-k+1} \leq \cdots < 0 < \lambda_1 \leq \ldots \leq \lambda_i \leq \cdots < \infty.
$$
By plugging $\overline\zeta_\infty$ into \eqref{eq:xibarinfty-zeta}, we derive
$$
a_i'(\tau)+\lambda_ia_i(\tau)=0, \quad i \in \mathbb Z \setminus \{0\}
$$
It follows that
$$
a_i(\tau)=c_ie^{-\lambda_i\tau}, \quad i \in \mathbb Z \setminus \{0\}
$$
for some constants $c_i$ and hence
\be\label{eq:|ai|2}
\|a_i\|^2_{L^2([1,2])}=\gamma(2\lambda_i)(\|a_i\|^2_{L^2([0,1])}+\|a_i\|^2_{L^2([2,3])})
\ee
with the function determined by the function
$$
\gamma(2c) : =
\frac{\int_1^2 e^{-2c \tau} \, d\tau}
{\int_0^1 e^{-2c \tau} \, d\tau + \int_2^3 e^{-2c \tau} \, d\tau}.
$$
Equivalently, we obtain
$$
\gamma(c) = \frac{e^{-c} - e^{-2c}}{1- e^{-c} +
 e^{-2c} - e^{-3\lambda_i}} = \frac{1}{e^c + e^{-c}}.
$$
\emph{(This is how the function $\gamma$ becomes relevant to this three interval argument. We note
that $\gamma$ is an even function.})
We compute
\beastar
\|\overline\zeta_\infty\|^2_{L^2([k, k+1])}&=&\int_{[k,k+1]}\|\overline\zeta_\infty\|^2_{L^2(S^1)}\,d\tau\\
&=&\int_{[k,k+1]}\sum_i |a_i(\tau)|^2 d\tau = \sum_i\|a_i\|^2_{L^2([k,k+1])}.
\eeastar
By the monotonically decreasing property of $\gamma$ for $c> 0$, this and \eqref{eq:|ai|2} give rise to
$$
\|\overline \zeta_\infty\|^2_{L^2([1,2])}< \gamma(2\delta) (\|\overline \zeta_\infty\|^2_{L^2([0,1])}+\|\overline \zeta_\infty\|^2_{L^2([2,3])})
$$
for any $\delta$ satisfying $0 < \delta < \min\{|\gamma_{-1}|, \gamma_1\}$.
Since $\overline\zeta_\infty\not \equiv 0$, this contradicts to \eqref{eq:xibarinfty-ineq-zeta}, if we choose
$0 < \delta <  \min\{|\gamma_{-1}|, \gamma_1\}$ at the beginning. This finishes the proof of the exponential decay
\be\label{eq:L2|zeta|}
\|\zeta\|_{L^2(I_k;\CE_1)} \leq C e^{-\delta \tau}
\ee
as $k\to \infty$.

Now we show this indicates the exponential decay of $\|\zeta\|_{\CE_{1,\tau}}$.
Using \eqref{eq:nabla=L} and \eqref{eq:coercive}, we also derive
\bea\label{eq:nablazeta}
\|\nabla_\tau\zeta(\tau)\|_{L^2(I_k,\CE_1)} & \leq & \|B(\tau) \zeta(\tau)\|_{L^2(I_k,\CE_1)}
+ \|L(\tau)\|_{L^2(I_k,\CE_1)}\nonumber\\
& \leq & \sup_{\tau \in [R_0,\infty)}\|B(\tau)\|_{\L(\CE_2,\CE_1)} \|\zeta(\tau)\|_{L^2(I_k,\CE_2)}
+ \|L(\tau)\|_{L^2(I_k,\CE_1)}\nonumber\\
& \leq & C_2' C(I_k,I_k')\left(\|(\nabla + B) \zeta\|_{L^2(I_k',\CE_1)}  + \|\zeta\|_{L^2(I_k',\CE_1)} \right)\nonumber\\
&{}& +  \|L(\tau)\|_{L^2(I_k,\CE_1)})\nonumber\\
& \leq & (C_2'C(I_k,I_k') + 1) \|L(\tau)\|_{L^2(I_k',\CE_1)} + C_2' C(I_k,I_k') \| \zeta\|_{L^2(I_k,\CE_1)}\nonumber\\
&\leq & C_2'' e^{-\delta k}.
\eea
Here we have chosen $I_k' = [k-\frac{1}{3},k+\frac{4}{3}]$, $C_2'=\sup_{\tau \in [R_0,\infty)}\|B(\tau)\|_{\L(\CE_2,\CE_1)}$ and $C_2'' = (C_2'C(I_k,I_k')e^{\delta/3} + 1)$.
Combining \eqref{eq:L2|zeta|} and \eqref{eq:nablazeta}, we have derived $\|\zeta\|_{W^{1,2}(I_k,\CE_1)} \leq C_2 e^{-\delta k}$
for all $k$ with $C_2 = \max\{C,C_2''\}$.

By applying Sobolev's inequality for the section $I_k \to \CE_1$
\be\label{eq:Sobolev}
\max_{\tau \in I_k} |\zeta(\tau)|_{\CE_1,\tau} \leq C_3\|\zeta\|_{W^{1,2}(I_k,\CE_1)}
\ee
with $C_3$ the Sobolev constant on $I_k$.  This now finishes the proof.
\end{proof}

\begin{rem} Since $\CE_1$ may not be finitely dimensional, application of the Sobolev inequality \eqref{eq:Sobolev} may not
be standard to some readers. For readers' convenience, we give
a direct proof of this inequality \eqref{eq:Sobolev} in Appendix \ref{sec:Sobolev}.
\end{rem}

\section{Exponential convergence: the prequantization case}
\label{sec:prequantization}
		
To make the main arguments transparent in the scheme of our
exponential estimates, we start with the case of prequantization, i.e.,
the case without $\CN$ and the normal form contains $E$ only.
The general case will be dealt with in the next section.

We put the basic hypothesis that
\be\label{eq:Hausdorff}
|e(\tau,t)| < \delta
\ee
for all $\tau \geq \tau_0$ in our further study,
where $\delta$ is given as in Proposition \ref{lem:deltaforE}.
From Corollary \ref{cor:convergence-ue} and the remark after it,
we can locally work with everything in a neighborhood of zero section in the
normal form $(U_E, f\lambda_E, J)$.

\subsection{Computational preparation}
\label{subsec:prepare}

For a smooth function $h$, we can express its gradient vector field $\grad h$
with respect to the metric $g_{(\lambda_E,J_0)} = d\lambda_E(\cdot, J_0\cdot ) + \lambda_E \otimes \lambda_E$
in terms of the $\lambda_E$-contact Hamiltonian vector field $X^{d\lambda_E}_h$ and
the Reeb vector field $X_E$ as
\be\label{eq:gradh}
\grad h = -J_0 X^{d\lambda_E}_h + X_E[h]\, X_E.
\ee
Note the first term
$-J_0 X^{d\lambda_E}_h=:\grad h^\pi$ is the $\pi_{\lambda_E}$-component of $\grad h$.

Consider the vector field $Y$ along $u$ given by
$
Y(\tau,t) := \nabla^\pi_\tau e
$
where $w = (u,e)$ in the coordinates defined in section \ref{sec:thickening}.
The vector field $e_\tau = e(\tau,t)$
as a vector field along $u(\tau,t)$ is nothing but
the map $(\tau,t) \mapsto I_{w(\tau,t);u(\tau,t)}(\vec R(w(\tau,t))$ as a section of $u^*E$. In particular
$$
e(\infty,t) = I_{w(\infty,t);u(\infty,t)}(\vec R(w(\infty,t)))= I_{z(t);x(Tt)}(\vec R(o_{x(T t)}) = o_{x(T t)}.
$$
Obviously, $I_{w(\tau,t);u(\tau,t)}(\vec R(w(\tau,t))$ is pointwise perpendicular to $o_E \cong Q$.
In particular,
\be\label{eq:perpendicular}
(\Pi_{x(T\cdot)}^{u_\tau})^{-1} e_\tau \in (\ker D\Upsilon(z))^\perp
\ee
where $e_\tau$ is the vector field along the loop $u_\tau \subset o_E$ and we regard
$(\Pi_{x(T\cdot)}^{u_\tau})^{-1} e_\tau$ as a vector field along $z = (x(T\cdot),o_{x(T\cdot)})$.

For further detailed computations, one needs to decompose the contact instanton map equation
\be\label{eq:contact-instanton-E}
\delbar_J^{f\lambda_E} w = 0, \quad d(w^*(f\, \lambda_E)) \circ j) = 0.
\ee
The second equation does not depend on the choice of endomorphisms $J$ and becomes
\be\label{eq:dw*flambda}
d(w^* \lambda_E\circ j) = - dg \wedge (\lambda_E\circ j),\quad g = \log f
\ee
which is equivalent to
\be\label{eq:dw*flambda}
d(u^*\theta \circ j + \Omega(e, \nabla^E_{du\circ j} e))
= - dg \wedge (u^* \theta \circ j + \Omega(e, \nabla^E_{du\circ j} e)).
\ee
On the other hand, by the formula \eqref{eq:xi-projection}, the first equation $\delbar_J^{f\lambda_E} w = 0$ becomes
\be\label{eq:delbarwflambda}
\delbar^{\pi_{\lambda_E}}_{J_0} w = (w^*\lambda_E\, Y_{dg})^{(0,1)} + (J - J_0) d^{\pi_{\lambda_E}} w
\ee
where $(w^*\lambda_E\, Y_{dg})^{(0,1)}$ is the $(0,1)$-part of the one-form $w^*\lambda_E\, Y_{dg}$
with respect to $J_0$.

In terms of the coordinates, the equation can be re-written as
\beastar
&{}& \left(\begin{matrix} \delbar^{\pi_\theta} u - \left(\Omega(\vec R(u,e),\nabla^E_{du} e)\, X_E(u,e)\right)^{(0,1)}\\
(\nabla^E_{du} e)^{(0,1)} \end{matrix}\right)\\
& = & \left(\begin{matrix}
\left( \left(u^*\theta + \Omega(\vec R(u,e),\nabla^E_{du} e)\right) d\pi_E((Y_{dg})^h)\right)^{(0,1)} \\
\left( \left(u^*\theta + \Omega(\vec R(u,e),\nabla^E_{du} e)\right)  X_g^\Omega(u,e)\right)^{(0,1)}\\
\end{matrix}\right) + (J - J_0) d^{\pi_{\lambda_E}} w.
\eeastar
Here $\left(\Omega(\vec R(u,e),\nabla^E_{du} e)\, X_E(u,e)\right)^{(0,1)}$ is the $(0,1)$-part with respect to
$J_Q$ and $(\nabla^E_{du} e)^{(0,1)}$ is the $(0,1)$-part with respect to $J_E$.
From this, we have derived

\begin{lem}\label{lem:eq-in-(u,e)}In coordinates $w = (u,e)$, \eqref{eq:contact-instanton-E} is equivalent to
\bea
\nabla''_{du} e & = & \left(w^*\lambda_E\, X_g^\Omega(u,e)\right)^{(0,1)} + I_{w;u}\left((J - J_0) d^{\pi_{\lambda_E}} w)^v\right)\label{eq:CR-e} \\
\delbar^\pi u & = & \left(w^*\lambda_E\, (d\pi_E(Y_{dg})^h)\right)^{(0,1)}
+ \left(\Omega(e, \nabla^E_{du} e)\, X_\theta(u(\tau,e))^h\right)^{(0,1)}\nonumber\\
&{}& + d\pi_E\left((J - J_0) d^{\pi_{\lambda_E}} w\right)^h
\label{eq:CR-uxi}
\eea
and
\be\label{eq:CR-uReeb}
d(w^*\lambda_E \circ j)  = -dg \wedge w^*\lambda_E \circ j
\ee
with the insertions of
$$
w^*\lambda_E = u^*\theta  + \Omega(e,\nabla^E_{du}e).
$$
\end{lem}
Note that with insertion of \eqref{eq:CR-uReeb}, we obtain
\bea\label{eq:delbarpiu=}
\delbar^\pi u
& = & \left(u^*\theta\, d\pi_E((Y_{dg})^h)\right)^{(0,1)}
+ \left(\Omega(e, \nabla^E_{du} e)\, (X_\theta(u(\tau,e))+ d\pi_E(Y_{dg}))^h\right)^{(0,1)}\, \nonumber \\
&{}& + d\pi_E\left((J - J_0) d^{\pi_{\lambda_E}} w\right)^h.
\eea

Now let $w=(u,e)$ be a contact instanton in terms of the decomposition as above.

\begin{lem}\label{lem:J-J0} Let $e$ be an arbitrary section over a
smooth map $u:\Sigma \to Q$. Then
\bea\label{eq:J-J0}
I_{w;u}\left(((J - J_0) d^{\pi_{\lambda_E}} w)^v\right)& = & L_1(u,e)(e,(d^\pi u,\nabla_{du} e))\\
d\pi_E\left(((J - J_0) d^{\pi_{\lambda_E}} w)^h\right) & = & L_2(u,e)(e,(d^\pi u,\nabla_{du} e))
\eea
where $L_1(u,e)$ is a $(u,e)$-dependent bilinear map with values in $\Omega^0(u^*E)$ and
$L_2(u,e)$ is a bilinear map with values in $\Omega^0(u^*TQ)$. They also satisfy
\be\label{eq:|J-J0|}
|L_i(u,e)| = O(1).
\ee
\end{lem}
An immediate corollary of this lemma is
\begin{cor}\label{cor:|J-J0|}
\beastar
|I_{w;u}\left(((J - J_0) d^{\pi_{\lambda_E}} w)^v\right)| & \leq & O(1)|e|(|d^\pi u| + |\nabla e|) \\
\left|\nabla\left(I_{w;u}\left(((J - J_0) d^{\pi_{\lambda_E}} w)^v\right)\right)\right| & \leq &
O(1)((|du| + |\nabla e|)^2|e| \\
&{}& \quad + |\nabla e|(|du| + |\nabla e|) + |e||\nabla^2 e|).
\eeastar
\end{cor}

Next, we give the following lemmas whose proofs are straightforward from the definition of
$X_g^\Omega$.

\begin{lem}\label{lem:MN} Suppose $d_{C^0}(w(\tau,\cdot),z(\cdot)) \leq \iota_g$. Then
\beastar
X^{\Omega}_{g}(u,e) &= & D^v X^\Omega_g(u,0) e + M_1(u,e)(e,e) \\
d\pi_E(Y_{dg}) & = & M_2(u,e)(e)\\
\Omega(e,\nabla_{du}^E e)\, X_g^\Omega(u,e) & =  & N(u,e)(e,\nabla_{du} e,e)
\eeastar
where $M_1(u,e)$ is a smoothly $(u,e)$-dependent bi-linear map on $\Omega^0(u^*E)$
and $M_2: \Omega^0(u^*E) \to \Omega^0(u^*TQ)$ is a linear map, $N(u,e)$ is a
$(u,e)$-dependent tri-linear map on $\Omega^0(u^*E)$. They also satisfy
$$
|M_i(u,e)| = O(1), \quad |N(u,e)| = O(1).
$$
\end{lem}

\begin{lem}\label{lem:K}
\beastar
(X^{d\lambda_E}_g)^h(u,e) = K(u,e)\, e
\eeastar
where $K(u,e)$ is a $(u,e)$-dependent linear map
from $\Omega^0(u^*E)$ to $\Omega^0(u^*TQ)$ satisfying
$$
|M_1(u,e)| = O(1), \quad |N(u,e)| = O(1), \quad |K(u,e)| = o(1).
$$
\end{lem}

\medskip

\subsection{$L^2$-exponential decay of the normal bundle component $e$}
\label{sec:expdecayMB}

Combining Lemma \ref{lem:J-J0}, \ref{lem:MN} and \ref{lem:K}, we can write \eqref{eq:CR-e}
as
\be\label{eq:nabladue}
\nabla''_{du}e - \left(u^*\theta \, DX^\Omega_g(u)(e)\right)^{(0,1)} = K(e,\nabla_{du} e, d^\pi u).
\ee
By evaluating \eqref{eq:nabladue} against $\frac{\del}{\del \tau}$, we derive
\be\label{eq:e-reduction}
\nabla_\tau e + J_E(u) \nabla_t e - \theta\left(\frac{\del u}{\del \tau}\right)DX_g^\Omega(u)(e)
- J_E \theta\left(\frac{\del u}{\del t}\right) DX_g^\Omega(u)(e)
= K\left(e, \nabla_\tau e, \pi_\theta \frac{\del u}{\del \tau}\right)
\ee
First notice that
\begin{lem} \label{lem:|K|=e2}
$$
\left|K\left(e, \nabla_\tau e, \frac{\del u}{\del \tau}\right)\right|_{L^\infty} = o(|e|)
$$
\end{lem}
\begin{proof}
We consider \eqref{eq:e-reduction}
as an equation for $e$. Clearly this is a quasilinear elliptic equation of $e$ when $u$ is fixed.
We also recall $K(e,\nabla_{du} e, du)$ has the form
$$
L_1(u,e)\left(e,\left(\nabla_\tau e, \frac{\del u}{\del \tau}\right)\right)
$$
where $L_1(u,e)$ is a bilinear map with $|L_1(u,e)| = O(1)$ by \eqref{eq:J-J0}
which satisfies the inequality
$$
\left|L_1(u,e)\left(e,\left(\nabla_\tau e, \frac{\del u}{\del \tau}\right)\right)\right| \leq O(1)
|e|(|d^\pi u| + |\nabla_\tau e|)
$$
(See Corollary \ref{cor:|J-J0|}.) Now the lemma
immediately follows from the convergence  $|\pi \frac{\del u}{\del \tau}|, \, |\nabla_\tau e| \to 0$
established in Corollary \ref{cor:convergence-ue}.
\end{proof}

Denote by $B(\tau)$ a $\tau$-family of operators
$$
B(\tau): W^{1, 2}(u(\tau, \cdot)^*E)\to L^2(u(\tau, \cdot)^*E)
$$
defined by
\beastar
B(\tau)e& :=& J_E(u) \nabla_t e - \theta\left(\frac{\del u}{\del \tau}\right)DX_g^\Omega(u)(e)
- J_E \theta\left(\frac{\del u}{\del t}\right) DX_g^\Omega(u)(e)\\
&{}& \quad -K(e, \nabla_\tau e, \frac{\del u}{\del \tau}).
\eeastar
Then \eqref{eq:e-reduction} for $e$ with $u$ fixed can be rewritten as
$$
\nabla_\tau e(\tau)+B(\tau)e(\tau)=0.
$$

Once we know $u(\tau, \cdot) \to z_\infty$ as $\tau \to \infty$ for some Reeb orbit $z_\infty$,
we can use the exponential map from $z_\infty$ to $u(\tau, \cdot)$ for any sufficiently large $\tau$
and its associated parallel transport to regard $B(\tau)$ as a $\tau$-family of
linear operators
$$
W^{1, 2}(z_\infty^*E)\to L^2(z_\infty^*E)
$$
along the limiting closed Reeb orbit $z_\infty$. (See Section 8 \cite{oh-wang2} for
a detailed discussion on this process.)
\begin{lem}\label{lem:B}
Let $\tau_k$ be a sequence with $\tau_k \to \infty$, and also denote by $\tau_k$
a subsequence thereof appearing in Theorem \ref{thm:subsequence}.
Under the above mentioned identification, the operator $B(\tau_k)$ converges to the the linearized operator
$$
B_\infty = J_E (z_\infty(t))(\nabla_t - \CT DX_g^\Omega(u_\infty))
$$
as $k \to \infty$.
\end{lem}
\begin{proof} Reorganize \eqref{eq:e-reduction} into
\beastar
\nabla_{\frac{\del u}{\del\tau}}e&+&J_E \left(\nabla_{\frac{\del u}{\del t}}e
-\lambda_E(\frac{\del u}{\del t})X^\Omega_g(u, e)\right)\\
&-&\lambda_E(\frac{\del u}{\del\tau})X^\Omega_g(u, e)
- K\left(e, \nabla_\tau e, \frac{\del u}{\del \tau}\right) =0.
\eeastar
We first note that $\nabla_{\frac{\del u}{\del t}(\tau_k,)} \to \CT \nabla_{\dot z_\infty}$
in the operator norm under the above mentioned identification.
(See Proposition 8.2 \cite{oh-wang1} and its proof for the precise explanation on this
statement.)
We now estimate the two terms in the second line.
For the first term, we have
$$
|\lambda_E(\frac{\del u}{\del\tau})X^\Omega_g(u, e)|\leq |\lambda_E(\frac{\del u}{\del\tau})||X^\Omega_g(u, e)|
=o(1)|e|,
$$
where the last estimate follows from Corollary \ref{cor:convergence-ue} and Lemma \ref{cor:|J-J0|}.
For the second term, Lemma \ref{lem:|K|=e2} implies that is of order $o(|e|)$.
This now completes the proof.
\end{proof}

Note that so far this convergence can only be expected in the subsequence sense.
Fortunately this weak convergence is already enough to conduct our scheme of three-interval argument,
when combined with the uniform local a priori estimate from \cite{oh-wang2}.

Next we briefly explain how the current situation fits into the general framework set
up in Section \ref{sec:three-interval}. We refer readers to Section 8 \cite{oh-wang2} for
further details of this verification.

We first consider two Banach spaces $\CE_{1, \tau} \supset \CE_{2, \tau}$
defined by
$$
\CE_{1, \tau}:=L^2(\iota_\tau^* u^*E), \quad \CE_{2, \tau}:=W^{1, 2}(\iota_\tau^* u^*E),
$$
where the maps
$$
\iota_\tau^*: S^1\to [0, \infty)\times S^1, \quad t\mapsto (\tau, t)
$$
are embeddings at $\tau\in [0, \infty)$. This family defines the bundle
$\CE_i$ over $[0, \infty)$ whose fiber at $\tau$ is given by $\CE_{1, \tau}$:
Its local triviality can be again proved by the parallel transport over a sufficient
small interval $(-\e +\tau, \tau + \e)$ at each given $\tau \in [0,\infty)$.

We denote the translation map $\sigma_s(\tau)=\tau+s$ by $\sigma_s: [a, b]\to [0, \infty)$ for each $s \in [a,b]$
for any given bounded interval $[a,b] \subset [0,\infty)$.
Using the exponential map over the limiting Reeb orbit $z_\infty$ and the associated
parallel transport for all sufficiently large $k$'s the sequence of Banach bundles
$$\CE_{i;k}:=\sigma_{\tau_k}^*\CE_i\to [a,b]$$
have global trivializations
$$\Phi_{i;k}:\CE_{i;k}\to L^2(z_\infty^*E)\times [a,b].$$
The uniform convergence proved in Theorem \ref{thm:subsequence} and Corollary \ref{cor:convergence-ue} ensures
the uniformly local tameness of $\CE_i$ and also the  precompactness, uniformly local coerciveness of $B$. In particular, since
$$
\theta(\frac{\del u}{\del\tau})(\tau, \cdot)\to \CQ=0, \quad \theta(\frac{\del u}{\del t})(\tau, \cdot)\to \CT \quad \tau\to \infty
$$
uniformly over $[a, b]$, we also conclude that the operator $B_\infty$ defined as
$$
B_\infty = J_E (z_\infty(t))(\nabla_t -  \CT DX_g^\Omega(u_\infty)),
$$
is the limit of $B(\tau)$ for $\tau \in [a,b]$ with respect to the subsequence $\{\tau_k\}$, as shown in Lemma \ref{lem:B}.
Moreover, we also notice that $B_\infty$ is invariant under the $\tau$-translations, and it shows that
$B$ is  asymptotically cylindrical. Also, it follows from the Morse-Bott condition and Corollary \ref{cor:gap}
that the operator $B_\infty$ is an unbounded self-adjoint operator with trivial kernel.
The $C^1$-bound of $e$ from \eqref{eq:Hausdorff} and \eqref{eq:highere} guarantees that
the uniform bound of $W^{1, 2}(S^1)$-norm of $e(\tau)$.

The above discussion verifies that \eqref{eq:e-reduction} can be fit into
the general abstract framework of Theorem \ref{thm:three-interval} applied to $\zeta= e$.
Therefore we immediately obtain the following $L^2$-exponential estimate.
\begin{prop} \label{prop:expdecayvertical} There exists a sufficiently large $\tau_0 > 0$
and constants $C_0, \, \delta_0$ such that
$$
\|e(\tau)\|_{L^2(S^1)} < C_0 e^{-\delta_0 \tau}
$$
for all $\tau \geq \tau_0$.
\end{prop}

\subsection{$L^2$-exponential decay of the tangential component $du$ I}
\label{subsec:exp-horizontal}

We summarize previous geometric calculations, especially the equation
 \eqref{eq:delbarpiu=}, into the following basic equation which we will study
using the three-interval argument in this section.

\begin{lem} We can write the equation \eqref{eq:delbarpiu=} into the form
\be\label{eq:pidudtauL}
\pi_\theta \frac{\del u}{\del \tau}+J(u)\pi_\theta \frac{\del u}{\del t}=L(\tau, t),
\ee
so that $\|L(\tau, \cdot)\|_{L^2(S^1)}\leq C e^{-\delta \tau}$.
\end{lem}
\begin{proof} We recall $|du| \leq C$ which follows from Corollary
\ref{cor:convergence-ue}. Furthermore since $X_{dg}^{d\lambda_E}|_Q \equiv 0$, it follows
$$
\left|\left(u^*\theta\, d\pi_E((X_{dg}^{d\lambda_E})^h)\right)^{(0,1)}\right| \leq C |e|.
$$
Furthermore by the adaptedness of $J$ and by the definition of the associated $J_0$,
we also have $(J - J_0)|_{Q} \equiv 0$ and so
$$
\left|\left(d\pi_E((J - J_0) d^{\pi_{\lambda_E}} w\right)^h\right| \leq C |e|.
$$
It is manifest that the second term above also carries similar estimate. Combining them,
we have established that the right hand side is bounded by $C|e|$ from above.
Then the required exponential inequality follows from that
of $e$ established in Proposition \ref{prop:expdecayvertical}.
\end{proof}

In the rest of this section and Section \ref{subsec:centerofmass}, we give the proof of the following
\begin{prop}\label{prop:expdecayhorizontal}
There exists some constant $C_0>0$ and $\delta_0>0$ such that
$$
\left\|\pi_\theta\frac{\del u}{\del \tau}\right\|_{L^2}< C_0\, e^{-\delta_0 \tau}.
$$
\end{prop}

The proof basically follows the same  three-interval argument as in the proof of Theorem \ref{thm:three-interval}.
However, since the current case is much more subtle, we would like to highlight the following points before we start:
\begin{enumerate}
\item Unlike the normal component $e$ whose governing equation \eqref{eq:e-reduction} is a (inhomogeneous) quasi-linear elliptic
equation, \eqref{eq:pidudtauL} is only (inhomogeneous) quasi-linear \emph{degenerate elliptic}:
the limiting operator $B$ of its linearization contains non-trivial kernel.
\item  Nonlinearity of the
equation makes somewhat cumbersome to formulate the abstract framework of three-interval argument as in
Theorem \ref{thm:three-interval} although we believe it is doable. Since this is not the main interest of
ours, we directly deal with \eqref{eq:pidudtauL} in the present paper postponing such an abstract
framework elsewhere in the future.
\item For the normal component, we directly establish the exponential estimates of the
map $e$ itself. On the other hand, for the tangential component, partly due to the absence of
direct linear structure of $u$ and also due to the presence of nontrivial kernel of the asymptotic
operator, we prove the exponential decay of the \emph{derivative} $\pi_\theta\frac{\del u}{\del\tau}$ first and
then prove the exponential convergence to some Reeb orbit afterwards.
\item To obtain the exponential decay of the \emph{derivative} term,
we need to exclude the possibility of a kernel element for the limit obtained in the
three-interval argument.  In Section \ref{subsec:centerofmass} we use the techniques of the center of mass as an intrinsic geometric coordinates system to exclude the possibility of the vanishing of the limit.
This idea appears in \cite{mundet-tian} and \cite{oh-zhu11} too.
\item Unlike  \cite{HWZ3} and \cite{bourgeois}, our proof directly
obtains $L^2$-exponential decay instead of showing $C^0$ convergence first and
getting  exponential decay afterwards.
\end{enumerate}

\medskip
Starting from now until the end of Section \ref{subsec:centerofmass}, we give the proof of Proposition \ref{prop:expdecayhorizontal}.

Divide $[0, \infty)$ into the union of unit intervals $I_k=[k, k+1]$ for
$k=0, 1, \cdots$, and denote by $Z_k:=[k, k+1]\times S^1$. In the context below, we also denote by
$
Z^l
$
the union of three-intervals $Z^l_{I}:=[k,k+1]\times S^1$, $Z^l_{II}:=[k+1,k+2]\times S^1$ and
$Z^l_{III}:=[k+2,k+3]\times S^1$.

Consider $x_k: =\|\pi_\theta \frac{\del u}{\del \tau}\|^2_{L^2(Z_k)}$ as symbols in Lemma \ref{lem:three-interval}.
As in the proof of Theorem \ref{thm:three-interval}, we still use the three-interval inequality as the criterion and consider two situations:
\begin{enumerate}
\item If there exists some constant $\delta>0$ such that
\be
\|\pi_\theta \frac{\del u}{\del \tau}\|^2_{L^2(Z_{k})}\leq \gamma(2\delta)(\|\pi_\theta \frac{\del u}{\del \tau}\|^2_{L^2(Z_{k-1})}
+\|\pi_\theta \frac{\del u}{\del \tau}\|^2_{L^2(Z_{k+1})})\label{eq:3interval-ineq-II}
\ee
holds for every $k$, then from Lemma \ref{lem:three-interval}, we are done with the proof;
\item Otherwise, we collect all the three-intervals $Z^{l_k}$ against \eqref{eq:3interval-ineq-II}, i.e.,
\be
\|\pi_\theta \frac{\del u}{\del \tau}\|^2_{L^2(Z^{l_k}_{II})}>
\gamma(2\delta)(\|\pi_\theta \frac{\del u}{\del \tau}\|^2_{L^2(Z^{l_k}_{I})}+\|\pi_\theta \frac{\del u}{\del \tau}\|^2_{L^2(Z^{l_k}_{III})}).\label{eq:against-3interval-II}
\ee
In the rest of the proof, we deal with this case.
\end{enumerate}

First,  if there exists some uniform constant $C_1>0$ such that
on each such three-interval
\be
\|\pi_\theta \frac{\del u}{\del \tau}\|_{L^\infty([l_k+0.5, l_k+2.5]\times S^1)}<C_1e^{-\delta l_k}, \label{eq:expas-zeta-II}
\ee
then through the same estimates and analysis as for Theorem \ref{thm:three-interval},
we obtain the exponential decay of $\|\pi_\theta \frac{\del u}{\del \tau}\|$ with the presumed rate $\delta$ as claimed.

\begin{rem}Here we look at the $L^\infty$-norm on smaller intervals  $[l_k+0.5, l_k+2.5]\times S^1$
instead of the whole $Z^{l_k}=[l_k, l_k+3]\times S^1$
is out of consideration for the elliptic bootstrapping argument in Lemma \ref{lem:boot}.
However, the change here doesn't change any argument, since smaller ones are already enough to cover the middle intervals (see Figure \ref{fig:three-interval}).
\end{rem}
Following the same scheme as for Theorem \ref{thm:three-interval}, we are going to deal with the case
when there is no uniform bound $C_1$. Then there exists a sequence of constants $C_k\to \infty$
and a subsequence of such three-intervals $\{Z^{l_k}\}$ (still use $l_k$ to denote them) such that
\be
\|\pi_\theta \frac{\del u}{\del \tau}\|_{L^\infty([l_k+0.5, l_k+2.5]\times S^1)}\geq C_ke^{-\delta l_k}.\label{eq:decay-fail-zeta-II}
\ee

By Theorem \ref{thm:subsequence}
and the local uniform $C^1$-estimate,
we can take a subsequence, still denoted by $\zeta_k$, such that $u(Z^{l_k})$ lives in a neighborhood of
some closed Reeb orbit $z_\infty$.
Next, we translate the sequence $u_k:=u|_{Z^{l_k}}$ to $\widetilde{u}_k: [0,3]\times S^1\to Q$
by defining
$
\widetilde{u}_k(\tau, t)=u_k(\tau+l_k, t).
$
As before, we also define $\widetilde{L}_k(\tau, t)=L(\tau+l_k, t)$. From \eqref{eq:pidudtauL-t}, we now have
\be\label{eq:pidudtauL-t}
\pi_\theta \frac{\del \widetilde{u}_k}{\del \tau}+J(\widetilde{u}_k)\pi_\theta \frac{\del \widetilde{u}_k}{\del t}=\widetilde{L}_k(\tau, t).
\ee

Recalling that $Q$ carries a natural $S^1$-action induced from the Reeb flow,
we equip $Q$ with a $S^1$-invariant metric and its associated Levi-Civita connection.
In particular, the vector field $X_\lambda$ restricted to $Q$ is a Killing vector
field of the metric and satisfies $\nabla_{X_\lambda}X_\lambda = 0$.

Now since the image of $\widetilde{u}_k$ live in neighbourhood of a fixed Reeb orbit $z$ in $Q$,
we can express
\be\label{eq:normalexpN}
\widetilde u_k(\tau,t) = \exp_{z_k(\tau,t)}^Z \zeta_k(\tau,t)
\ee
for the normal exponential map $exp^Z: NZ\to Q$ of the locus $Z$ of $z$, where
$z_k(\tau,t) = \pi_N(\widetilde u_k(\tau,t))$ is the normal projection of $\widetilde u_k(\tau,t)$
to $Z$ and $\zeta_k(\tau,t) \in N_{z_k(\tau,t)}Z = \zeta_{z_k(\tau,t)} \cap T_{z_k(\tau,t)}Q$.
Then

\begin{lem}\label{lem:pithetadel}
\be
\pi_\theta\frac{\del \widetilde{u}_k}{\del \tau} =\pi_\theta (d_2\exp^Z)(\nabla^{\pi_{\theta}}_\tau\zeta_k)
\label{eq:utau-thetatau}
\ee
$$
\pi_\theta\frac{\del \widetilde{u}_k}{\del t}=\pi_\theta (d_2\exp^Z)(\nabla^{\pi_{\theta}}_t\zeta_k).
$$
\end{lem}
\begin{proof} To simplify notation, we omit $k$ here.
For each fixed $(\tau,t)$, we compute
$$
D_1 \exp^Z(z(\tau,t))(X_\lambda(z(\tau,t))
= \frac{d}{ds}\Big|_{s = 0} \exp^Z_{\alpha(s)} \Pi_{z(\tau,t)}^{\alpha(s)}(X_\lambda(z(\tau,t))
$$
for a curve $\alpha: (-\e,\e) \to Q$ satisfying $\alpha(0) = z(\tau,t), \, \alpha'(0) = X_\lambda(z(\tau,t))$.
For example, we can take $\alpha(s) = \phi_{X_\lambda}^s(z(\tau,t))$.

On the other hand, we compare the initial conditions of the two geodesics
$a \mapsto \exp^Z_{\alpha(s)} a \Pi_{x}^{\alpha(s)}(X_\lambda(x))$ and
$a \mapsto \phi_{X_\lambda}^s(\exp^Z_x a (X_\lambda(x))$ with $x = z(\tau,t)$.
Since $\phi_{X_\lambda}^s$ is an isometry, we derive
$$
\phi_{X_\lambda}^s(\exp^Z_x a (X_\lambda(x)) = \exp^Z_x a (d\phi_{X_\lambda}^s(X_\lambda(x))).
$$
Furthermore we note that $d\phi_{X_\lambda}^s(X_\lambda(x)) = X_\lambda(x) $ at $s = 0$
and the field $s \mapsto d\phi_{X_\lambda}^s(X_\lambda(x))$ is parallel along the curve $s \mapsto \phi_{X_\lambda}^s(x)$.
Therefore by definition of $\Pi_x^{\alpha(s)}(X_\lambda(z(\tau,t))$, we derive
$$
\Pi_x^{\alpha(s)}(X_\lambda(x))= d\phi_{X_\lambda}^s(X_\lambda(x)).
$$
Combining this discussion, we obtain
$$
\exp^Z_{\alpha(s)} \Pi_{z(\tau,t)}^{\alpha(s)}(X_\lambda(z(\tau,t)))
= \phi_{X_\lambda}^s(\exp^Z_{z(\tau,t)}(X_\lambda(z(\tau,t)))
$$
for all $s \in (-\e,\e)$. Therefore we obtain
$$
\frac{d}{ds}\Big|_{s = 0} \exp^Z_{\alpha(s)} \Pi_{z(\tau,t)}^{\alpha(s)}(X_\lambda(z(\tau,t)))
= X_\lambda(\exp^Z_{z(\tau,t)}(X_\lambda(z(\tau,t))).
$$
This shows
$
(D_1\exp^Z)(X_\lambda) = X_\lambda(\exp^Z_{z(\tau,t)}(X_\lambda(z(\tau,t)))
$.

To see $\pi_\theta (D_1\exp^Z)(\frac{\del z}{\del \tau}) = 0$, just note that $\frac{\del z}{\del \tau}
 = k(\tau,t) X_\lambda(z(\tau,t))$ for some function $k$, which is parallel to $X_\lambda$
and $z(\tau,t) \in Z$. Using the definition of $D_1\exp^Z(x)(v)$ for $v \in T_x Q$ at $x \in Q$, we
compute
\beastar
(D_1\exp^Z)(\frac{\del z}{\del \tau})(\tau,t) & = & D_1 \exp^Z(z(\tau,t))( k(\tau,t) X_\lambda(z(\tau,t))\\
& = & k(\tau,t) D_1 \exp^Z(z(\tau,t))(X_\lambda(z(\tau,t)),
\eeastar
and hence the $\pi_\theta$ projection vanishes.

At last write
\beastar
\pi_\theta \frac{\del \widetilde{u}}{\del \tau}= \pi_\theta (d_2\exp^Z)(\nabla^{\pi_{\theta}}_\tau\zeta)
+ \pi_\theta(D_1\exp^Z)(\frac{\del z}{\del \tau})
\eeastar
and we are done with the first identity claimed.

The second one is proved exactly the same way.
\end{proof}

Further noting that $\pi_\theta (d_2\exp^Z_{z_k(\tau,t)}): \zeta_{z_k(\tau,t)} \to \zeta_{z_k(\tau,t)}$ is invertible,
using this lemma and \eqref{eq:pidudtauL-t}, we now have the equation of $\zeta_k$
\be
\nabla^{\pi_{\theta}}_\tau\zeta_k+ \overline J(\tau,t) \nabla^{\pi_{\theta}}_t\zeta_k
= [\pi_\theta (d_2\exp^Z)]^{-1}\widetilde{L}_k.
\label{eq:nablaxi}
\ee
where we set  $[\pi_\theta (d_2\exp)]^{-1}J(\widetilde{u}_k)[\pi_\theta (d_2\exp^Z)] =: \overline J(\tau,t)$.

Next, we rescale this equation by the norm $\|{\zeta}_k\|_{L^\infty([0,3]\times S^1)}$:
\begin{lem} The norm $\|{\zeta}_k\|_{L^\infty([0,3]\times S^1)}$ is not zero.
\end{lem}
\begin{proof} Suppose to the contrary that
${\zeta}_k\equiv 0$. This then implies $\widetilde u_k(\tau,t) \equiv z_k(\tau, t)$ for all $(\tau, t)\in [0,3]\times S^1$.
Therefore $\frac{\del \widetilde u_k}{\del \tau}$ is parallel to $X_\theta$ on $[0,3] \times S^1$.
In particular $\pi_\theta \frac{\del \widetilde u_k}{\del \tau} \equiv 0$. This
violates the inequality
$$
\|\pi_\theta\frac{\del{\widetilde{u}_k}}{\del\tau}\|_{L^2([1,2]\times S^1)}>\gamma(\|\pi_\theta\frac{\del{\widetilde{u}_k}}{\del\tau}\|_{L^2([0,1]\times S^1)}+\|\pi_\theta\frac{\del{\widetilde{u}_k}}{\del\tau}\|_{L^2([2,3]\times S^1)}).
$$
Therefore the lemma holds.
\end{proof}

Now the rescaled sequence
$
\overline \zeta_k:={\zeta}_k/\|{\zeta}_k\|_{L^\infty([0,3]\times S^1)}
$
satisfies $\|\overline \zeta_k\|_{{L^\infty([0,3]\times S^1)}} = 1$, and
\bea\label{eq:nablabarxi}
\nabla^{\pi_{\theta}}_\tau\bar\zeta_k+  \overline J(\tau,t) \nabla^{\pi_{\theta}}_t\bar\zeta_k &=& \frac{[\pi_\theta (d_2\exp^Z)]^{-1}\widetilde{L}_k}{\|{\zeta}_k\|_{L^\infty([0,3]\times S^1)}}\\
\|\nabla^{\pi_{\theta}}_\tau\bar\zeta_k\|^2_{L^2([1,2]\times S^1)}&\geq&\gamma(2\delta)(\|\nabla^{\pi_{\theta}}_\tau\bar\zeta_k\|^2_{L^2([0,1]\times S^1)}+\|\nabla^{\pi_{\theta}}_\tau\bar\zeta_k\|^2_{L^2([2,3]\times S^1)}).\nonumber
\eea

The next step is to  focus on the right hand side of \eqref{eq:nablabarxi}.
\begin{lem}\label{lem:boot} The right hand side of \eqref{eq:nablabarxi} converges to zero as $k \to \infty$.
\end{lem}
\begin{proof}
Since the left hand side of \eqref{eq:nablaxi} is an elliptic (Cauchy-Riemann type) operator,
we have the elliptic estimates
\beastar
\|\nabla^{\pi_\theta}_\tau\zeta_k\|_{W^{l,2}([0.5, 2.5]\times S^1)}&\leq& C_1(\|\zeta_k\|_{L^2([0, 3]\times S^1)}+\|\widetilde{L}_k\|_{L^2([0, 3]\times S^1)})\\
&\leq& C_2(\|\zeta_k\|_{L^\infty([0, 3]\times S^1)}+\|\widetilde{L}_k\|_{L^\infty([0, 3]\times S^1)}).
\eeastar
The Sobolev's embedding theorem further gives
$$
\|\nabla^{\pi_\theta}_\tau\zeta_k\|_{L^\infty([0.5, 2.5]\times S^1)}\leq C(\|\zeta_k\|_{L^\infty([0, 3]\times S^1)}+\|\widetilde{L}_k\|_{L^\infty([0, 3]\times S^1)}).
$$
Hence we have
\beastar
\frac{\|\widetilde{L}_k\|_{L^\infty([0, 3]\times S^1)}}{\|\nabla^{\pi_\theta}_\tau\zeta_k\|_{L^\infty([0.5, 2.5]\times S^1)}}
&\geq& \frac{\|\widetilde{L}_k\|_{L^\infty([0, 3]\times S^1)}}{C(\|\zeta_k\|_{L^\infty([0, 3]\times S^1)}+\|\widetilde{L}_k\|_{L^\infty([0, 3]\times S^1)})}\\
&=&\frac{1}{C(\frac{\|\zeta_k\|_{L^\infty([0, 3]\times S^1)}}{\|\widetilde{L}_k\|_{L^\infty([0, 3]\times S^1)}}+1)}
\eeastar
We use our standing assumption \eqref{eq:decay-fail-zeta-II}
with $C_k \to \infty$ as $k \to \infty$ and Lemma \ref{lem:pithetadel},
the left hand side converges to zero as $k\to \infty$, so we get
$$
\frac{\|\widetilde{L}_k\|_{L^\infty([0, 3]\times S^1)}}{\|\zeta_k\|_{L^\infty([0, 3]\times S^1)}}\to 0, \quad \text{as} \quad k\to \infty.
$$
Thus the right hand side of \eqref{eq:nablabarxi} converges to zero.
\end{proof}

Then with the same argument as in the proof of Theorem \ref{thm:three-interval}, after taking a subsequence,
we obtain a limiting section $\overline\zeta_\infty$ of $z_\infty^*\zeta_\theta$ satisfying
\bea
\nabla_\tau \overline\zeta_\infty+ B_\infty \overline\zeta_\infty&=&0.\label{eq:xibarinfty-II}\\
\|\nabla_\tau\overline \zeta_\infty\|^2_{L^2([1,2]\times S^1)}&\geq&\gamma(2\delta)(\|\nabla_\tau\overline \zeta_\infty\|^2_{L^2([0,1]\times S^1)}+\|\nabla_\tau\overline \zeta_\infty\|^2_{L^2([2,3]\times S^1)}).
\nonumber\\
&{}&
\label{eq:xibarinfty-ineq-II}
\eea
Here to make it compatible with the notation used in Theorem \ref{thm:three-interval},
we denote by $B_\infty$ the limit operator of $B:=\overline J(\tau,t) \nabla^{\pi_{\theta}}_t$.
When applied to horizontal part as in the current case of study,
the operator is nothing but the linearization of Reeb orbit $z_\infty$ followed by action of $J$.

Write
$$
\overline\zeta_\infty=\sum_{j=0, \cdots k}a_j(\tau)e_j+\sum_{i\geq k+1}a_i(\tau)e_i,
$$
where $\{e_i\}$ is the basis consisting of the eigenfunctions associated to the eigenvalue
$\lambda_i$ for $j\geq k+1$ with
$$
0< \lambda_{k+1} \leq \lambda_{k+2} \leq \cdots \leq \lambda_i \leq \cdots \to \infty.
$$
and $e_j$ for $j=1, \cdots k$ are eigen-functions of eigen-value zero.
By plugging $\overline\zeta_\infty$ into \eqref{eq:xibarinfty-zeta}, we derive
\beastar
a_j'(\tau)&=&0, \quad  j=1, \cdots, k\\
a_i'(\tau)+\lambda_ia_i(\tau)&=&0, \quad  i=k+1,  \cdots
\eeastar
and it follows that
\beastar
a_j&=&c_j, \quad  j=1, \cdots, k\\
a_i(\tau)&=&c_ie^{-\lambda_i\tau}, \quad  i=k+1,  \cdots.
\eeastar
By the same calculation in the proof of Theorem \ref{thm:three-interval}, it follows
$$
\|\nabla^{\pi_\theta}_\tau\overline \zeta_\infty\|^2_{L^2([1,2]\times S^1)}< \gamma(2\delta) (\|\nabla^{\pi_\theta}_\tau\overline \zeta_\infty\|^2_{L^2([0,1]\times S^1)}
+\|\nabla^{\pi_\theta}_\tau\overline \zeta_\infty\|^2_{L^2([2,3]\times S^1)}).
$$

As a conclusion of this section, it remains to show
\begin{lem}\label{lem:nonzero}
$$
\nabla_\tau^{\pi_\theta} \overline \zeta_\infty \neq 0.
$$
\end{lem}
This lemma will then lead to contradiction and hence finish the proof of Proposition \ref{prop:expdecayhorizontal}.
The proof of this non-vanishing is given in the next section via the study of the center of mass.

\subsection{$L^2$-exponential decay of the tangential component $du$ II: study of center of mass}
\label{subsec:centerofmass}

We equip the submanifold with an $S^1$-invariant metric. Then its associated Levi-Civita connection $\nabla$
is also $S^1$-invariant.
Denote by $\exp: U_{o_Q} \subset TQ \to Q \times Q$ the exponential map,
and this defines a diffeomorphism between the open neighborhood $U_{o_Q}$ of the zero section $o_Q$ of $TQ$
and some open neighborhood $U_\Delta$ of the diagonal $\Delta \subset Q \times Q$. Denote its inverse by
$$
E: U_\Delta \to U_{o_Q}; \quad E(x,y) = \exp_x^{-1}(y).
$$
We refer readers to \cite{katcher} for the detailed study of the various basic derivative
estimates of this map.

The following lemma is a variation of the well-known center of mass techniques from Riemannian geometry
with the contact structure being taken into consideration (by introducing the reparameterization function $h$ in the following statement).

\begin{lem}\label{lem:centerofmass-Reeb}
Let $(Q, \theta=\lambda|_Q)$ be the submanifold foliated by closed Reeb orbits of period $T$.
Then there exists some $\delta > 0$ depending only
on $(Q, \theta)$ such that for any $C^{k+1}$ loop $\gamma: S^1 \to M$ with
$d_{C^{k+1}}(\gamma, \frak{Reeb}(Q, \theta)) < \delta $, there exists a unique point $m(\gamma) \in Q$,
and a reparameterization map $h:S^1\to S^1$ which is $C^k$ close to $id_{S^1}$,
such that
\bea
\int_{S^1} E(m, (\phi^{Th(t)}_{X_\theta})^{-1}(\gamma(t)))\,dt&=&0\label{eq:centerofmass-Reeb1}\\
E(m, (\phi^{Th(t)}_{X_\theta})^{-1}(\gamma(t)))&\in& \xi_\theta(m) \quad \text{for all $t\in S^1$}.\label{eq:centerofmass-Reeb2}
\eea
\end{lem}
\begin{proof}
Consider the functional
$$
\Upsilon: C^\infty(S^1, S^1)\times Q\times C^{\infty}(S^1, Q)\to TQ\times \mathcal{R}
$$
defined as
$$
\Upsilon(h, m, \gamma):=\left(\left(m, \int_{S^1} E(m, (\phi^{Th(t)}_{X_\theta})^{-1}(\gamma(t))dt )\right), \theta(E(m, (\phi^{Th(t)}_{X_\theta})^{-1}(\gamma(t))) ) \right),
$$
where $\mathcal{R}$ denotes the trivial bundle over $\R\times Q$ over $Q$.

If $\gamma$ is a Reeb orbit with period $T$, then
$h=id_{S^1}$ and $m(\gamma)=\gamma(0)$ will solve the equation
$$
\Upsilon(h, m, \gamma)=(o_Q, o_{\CR}).
$$
From straightforward calculations
\beastar
&{}& D_h\Upsilon\big|_{(id_{S^1}, \gamma(0), \gamma)}(\eta) \\
&=&\left(\left(m, \int_{S^1} d_2E\big|_{(\gamma(0), \gamma(0))}(\eta(t)TX_\theta)\,dt \right), \theta(d_2E\big|_{(\gamma(0), \gamma(0))}(\eta(t)TX_\theta) ) \right)\\
&=&\left(\left(m, (T\int_{S^1} \eta(t) \,dt)\cdot X_\theta(\gamma(0)) \right), T\eta(t) \right)
\eeastar
and
\beastar
D_m\Upsilon\big|_{(id_{S^1}, \gamma(0), \gamma)}(v)
&=&\left(\left(v, \int_{S^1} D_1E\big|_{(\gamma(0), \gamma(0))}(v)\,dt \right), \theta(D_1E\big|_{(\gamma(0), \gamma(0))}(v) ) \right)\\
&=&\left(\left(v, \int_{S^1} D_1E\big|_{(\gamma(0), \gamma(0))}(v)\,dt \right), \theta(D_1E\big|_{(\gamma(0), \gamma(0))}(v) ) \right)\\
&=&\left((v, v ), \theta(v)  \right),
\eeastar
we claim that
$D_{(h, m)}\Upsilon$ is transversal to $o_{TM}\times o_{\CR}$ at the point $(id_{S^1}, \gamma(0), \gamma(t))$, where $\gamma$ is a Reeb
orbit of period $T$.
To see this, notice that for any point in the set
$$
\Delta:=\{(aX_\theta+\mu, f)\in TQ\times C^\infty(S^1, \mathbb{R})\big| a=\int_{S^1}f(t)dt\},
$$
one can always find its pre-image as follows:
For any given $(aX_\theta+\mu, f)\in TQ\times C^\infty(S^1, \mathbb{R})$ with $a=\int_{S^1}f(t)dt$,
the pair
\beastar
v&=&a\cdot X_\theta+\mu\\
\eta(t)&=&\frac{1}{T}(f(t)-a)
\eeastar
lives in the preimage.
This proves surjectivity of the partial derivative $D_{(h,m)}\Upsilon\big|_{(id_{S^1}, \gamma(0), \gamma)}$.
Then applying the implicit function theorem, we have finished the proof.
\end{proof}

Using the center of mass, we can derive the following proposition which will be used to exclude the possibility of the vanishing of $\nabla_\tau^{\pi_\theta} \overline \zeta_\infty$.
\begin{prop}\label{prop:centerofmass-app} Recall the rescaling sequence $\frac{\widetilde \zeta_k}{L_k}$
we take in the proof Proposition \ref{prop:expdecayhorizontal} above,
and assume for a subsequence,
$$
\frac{\widetilde \zeta_k}{L_k} \to \overline \zeta_\infty
$$
in $L^2$. Then $\int_{S^1} (d\phi_{X_\theta}^{Tt})^{-1}(\overline \zeta_\infty(\tau,t)) \, dt = 0$
where $(d\phi_{X_\theta}^{t})^{-1}(\overline \zeta_\infty(\tau,t)) \in T_{z(0)}Q$.
\end{prop}
\begin{proof}
By the construction of the center of mass applying to maps $u_k(\tau, \cdot): S^1\to Q$
for $\tau\in [0,3]$, we have obtained
$$
\int_{S^1}E(m_k(\tau), (\phi_{X_\theta}^{Th_k(\tau, t)})^{-1}u_k(\tau, t))\,dt=0.
$$
If we write $u_k(\tau, t)=\exp_{z_\infty(\tau, t)}\zeta_k(\tau, t)$ where $z_\infty$ is the limit
of $z_k$ defined in \eqref{eq:normalexpN}, it follows that
\bea\label{eq:intEmk}
&{}& \int_{S^1}E(m_k(\tau), \exp_{(\phi_{X_\theta}^{Th_k(\tau, t)})^{-1}z_\infty(\tau, t)}d(\phi_{X_\theta}^{Th_k(\tau, t)})^{-1}\zeta_k(\tau, t))\,dt \nonumber\\
&= &\int_{S^1}E(m_k(\tau), (\phi_{X_\theta}^{Th_k(\tau, t)})^{-1}\exp_{z_\infty(\tau, t)}\zeta_k(\tau, t))\,dt = 0
\eea
Recall the following lemma whose proof is direct and we skip.

\begin{lem} Let $\Pi_y^x$ is the parallel transport along the short geodesic from
$y$ to $x$. Then there exists some sufficiently small $\delta > 0$ depending only on
the given metric on $Q$ and a constant $C = C(\delta) > 0$ such that $C(\delta) \to 1$ as $\delta \to 0$
and
$$
|E(x, \exp_y^Z(\cdot)) - \Pi_y^x| \leq C \, d(x, y).
$$
In particular $|E(x, \exp_y^Z(\cdot))| \leq |\Pi_y^x| + C \, d(x, y)$.
\end{lem}
Applying this lemma to \eqref{eq:intEmk}, we obtain
\beastar
&{}& \left|\int_{S^1}\Pi_{m_k(\tau)}^{(\phi_{X_\theta}^{Th_k(\tau, t)})^{-1}z_\infty(\tau, t)}(d\phi_{X_\theta}^{Th(t)})^{-1}\zeta_k(\tau, t)\,dt\right| \\
& \leq &
\int_{S^1} C \,d(m_k(\tau), (\phi_{X_\theta}^{Th_k(\tau, t)})^{-1}z_\infty(\tau, t))\left((d\phi_{X_\theta}^{Th(t)})^{-1}\zeta_k(\tau, t)\right)\,dt.
\eeastar

We rescale $\zeta_k$ by using $L_k$ and derive that
\beastar
&&{}\left|\int_{S^1}\Pi_{m_k(\tau)}^{(\phi_{X_\theta}^{Th_k(\tau, t)})^{-1}z_\infty(\tau, t)}
\left(d\phi_{X_\theta}^{Th_k(\tau, t)}\right)^{-1}\frac{\zeta_k(\tau, t)}{L_k}\,dt\right| \\
&\leq & \int_{S^1} C \,d(m_k(\tau), (\phi_{X_\theta}^{Th_k(\tau, t)})^{-1}z_\infty(\tau, t))
\left|\left(d\phi_{X_\theta}^{Th_k(\tau, t)}\right)^{-1}\frac{\zeta_k(\tau, t)}{L_k}\right| \,dt.
\eeastar
Take $k\to \infty$, and since that $m_k(\tau) \to z_{\infty}(0)$ and $h_k\to id_{S^1}$ uniformly,
we get $d(m_k(\tau), (\phi_{X_\theta}^{Th_k(\tau, t)})^{-1}z_\infty(\tau, t)) \to 0$
uniformly over $(\tau,t) \in [0,3] \times S^1$ as $k \to \infty$. Therefore the right hand
side of this inequality goes to 0. On the other hand by the same reason, we obtain
$$
\Pi_{m_k(\tau)}^{(\phi_{X_\theta}^{Th_k(\tau, t)})^{-1}z_\infty(\tau, t)}
\left(d\phi_{X_\theta}^{Th_k(\tau, t)}\right)^{-1}\frac{\zeta_k(\tau, t)}{L_k} \to
(d\phi_{X_\theta}^{Tt})^{-1}\overline \zeta_\infty
$$
uniformly and hence we obtain
$
\int_{S^1}(d\phi_{X_\theta}^{Tt})^{-1}\overline \zeta_\infty\,dt=0.
$
\end{proof}

Using this proposition, we now prove
$
\nabla_\tau^{\pi_\theta} \overline \zeta_\infty \neq 0,
$
which is the last piece of finishing the proof of Proposition \ref{prop:expdecayhorizontal}.

\begin{proof}[Proof of Lemma \ref{lem:nonzero}]
Suppose to the contrary, i.e., $\nabla_\tau^{\pi_\theta} \overline \zeta_\infty = 0$,
then we would have $J(z_\infty(t))\nabla_t \overline \zeta_\infty = 0$ from
\eqref{eq:xibarinfty-II} and the remark right after it.

Fix a basis $
\{v_1, \cdots, v_{2k}\}
$ of $\xi_{z(0)} \subset T_{z(0)}Q$
and define
$
e_i(t) = d\phi_{X_\theta}^t(v_i)$ for $i = 1, \cdots, 2k$.
Since $\nabla_t^{\pi_\theta}$ is nothing but the linearization of the Reeb orbit $z_\infty$, the Morse-Bott
condition implies
$
\ker B_\infty = \span \{e_i(t)\}_{i=1}^{2k}.
$
Then one can express
$
\overline \zeta_\infty(t) = \sum_{i=1}^{2k} a_i(t) e_i(t)$.
Moreover since $e_i$ are parallel (with respect to the $S^1$-invariant connection) by construction, it follows $a_i$ are constants, $i = 1, \cdots, 2k$.
Then we can write
$
\overline \zeta_\infty(t) = \phi_{X_\theta}^t(v)$, where $v= \sum_{i=1}^{2k} a_i v_i$,
and further it follows that $\int_{S^1}((d\phi_{X_\theta}^t)^{-1}(\overline \zeta_\infty(t))\, dt = v$.
On the other hand,  from Proposition \ref{prop:centerofmass-app},  $v=0$ and further
$\overline \zeta_\infty \equiv 0$, which contradicts with $\|\overline \zeta_\infty\|_{L^\infty([0,3]\times S^1)} = 1$.

Thus finally we conclude that $\nabla_\tau^{\pi_\theta}\overline \zeta_\infty$ can not be zero.
\end{proof}

This now concludes  the proof of Proposition \ref{prop:expdecayhorizontal}.

\subsection{$L^2$-exponential decay of the Reeb component of $dw$}
\label{subsec:Reeb}

We again consider the equation
\be\label{eq:nabladue2}
\nabla''_{du}e - \left(u^* \, DX^\Omega_g(u)(e)\right)^{(0,1)} = K(e,\nabla_{du} e, du)
\ee
as an equation for $e$. Clearly this is a quasi-linear elliptic equation of $e$ when $u$ is fixed.
Applying the uniform (local) elliptic estimates
to \eqref{eq:nabladue2}, the $L^2$-exponential decays of $e$
and  convergence of $\frac{\del u}{\del \tau}$ to $0$ then lead to the $L^2$-exponential decay of  $\nabla_{du} e$.
Combining these, we have obtained $L^2$-exponential estimates of the $\pi dw$.

Now we consider the original map $w=(u, e)$ which satisfies  \eqref{eq:contact-instanton-E}
$$
\delbar_J^{\lambda} w = 0, \quad d(w^*\lambda\circ j) = 0
$$
where $\lambda = f\lambda_E$.
We recall that this system is an elliptic system and the corresponding
uniform local a priori estimates was established in \cite{oh-wang2}.
Then by the elliptic bootstrapping argument using the local uniform a priori estimates
on the cylindrical region, we obtain higher order $W^{k,2}$-exponential decay of
$\pi dw$ for all $k\geq 0$ under the hypothesis, Hypothesis \ref{hypo:exact}.

Next,
in the rest of this subsection, we prove the exponential decay of the Reeb component $w^*\lambda$.
For this purpose, we define a complex-valued function
$$
\alpha(\tau,t) = \left(w^*\lambda(\frac{\del}{\del t})-\CT\right) + \sqrt{-1}\left(w^*\lambda(\frac{\del}{\del\tau})\right).
$$
The following lemma is easy to prove.
\begin{lem} Let $\zeta = \pi \frac{\del w}{\del \tau}$ on the cylindrical ends. Then
$$
*d(w^*\lambda) =|\zeta|^2.
$$
\end{lem}
Combining this lemma together with the equation $d(w^*\lambda\circ j)=0$, we
notice that $\alpha$ satisfies the equations
\be\label{eq:atatau-equation}
\delbar \alpha =\nu, \quad \nu
=\frac{1}{2}|\zeta|^2 + \sqrt{-1}\cdot 0,
\ee
where $\delbar=\frac{1}{2}\left(\frac{\del}{\del \tau}+\sqrt{-1}\frac{\del}{\del t}\right)$ the standard Cauchy-Riemann operator for the standard complex structure $J_0=\sqrt{-1}$.

Notice that from previous section we have already
established the $W^{1, 2}$-exponential decay of $\nu = \frac{1}{2}|\zeta|^2$.
The exponential decay of $\alpha$ follows from the following lemma,
whose proof can be proved again by the three-interval method in a much easier way
and so omitted.

\begin{lem}\label{lem:exp-decay-lemma}
Suppose the complex-valued functions $\alpha$ and $\nu$ defined on $[0, \infty)\times S^1$
satisfy
\beastar
\begin{cases}
\delbar \alpha = \nu \\
\|\nu\|_{L^2(S^1)}+\left\|\nabla\nu\right\|_{L^2(S^1)}\leq Ce^{-\delta \tau} \quad \text{ for some constants } C, \delta>0\\
\lim_{\tau\rightarrow +\infty}\alpha =0
\end{cases}
\eeastar
then $\|\alpha\|_{L^2(S^1)}\leq \overline{C}e^{-\delta \tau}$
for some constant $\overline{C}$.
\end{lem}

\subsection{$C^0$ exponential convergence}

Now we prove $C^0$-exponential convergence of $w(\tau,\cdot)$ to some Reeb orbit as $\tau \to \infty$
from the $L^2$-exponential estimates presented in previous sections.

\begin{prop}\label{prop:czero-convergence}
Under Hypothesis \ref{hypo:basic}, for any contact instanton $w$ with vanishing charge,
there exists a unique Reeb orbit $z(\cdot)=\gamma(T\cdot):S^1\to M$ with period $\CT>0$, such that
$$
\|d(w(\tau, \cdot), z(\cdot))
\|_{C^0(S^1)}\rightarrow 0,
$$
as $\tau\rightarrow +\infty$,
where $d$ denotes the distance on $M$ defined by the triad metric.
\end{prop}
\begin{proof}
We start with the following lemma

\begin{lem} Let $t \in S^1$ be given. Then
for any given $\epsilon > 0$, there exists sufficiently large $\tau_1 > 0$ such that
$$
d(w(\tau,t), w(\tau', t)) < \epsilon
$$
for all $\tau, \, \tau' \geq \tau_1$.
\end{lem}
\begin{proof}
Suppose to the contrary that
there exist some $t_0\in S^1$ and some constant $\epsilon>0$, sequences $\tau_k \to \infty$, $p_k>0$ such that
\be\label{eq:geqepsilon}
d(w(\tau_{k+p_k}, t_0), w(\tau_k, t_0))\geq\epsilon.
\ee
Then combining this with the continuity of $w$ in $t$, there exists some $l>0$ small such that
$$
d(w(\tau_{k+p_k}, t), w(\tau_k, t))\geq\frac{\epsilon}{2}, \quad |t-t_0|\leq l.
$$
Therefore
\beastar
&{}&\int_{S^1}d(w(\tau_{k+p_k}, t), w(\tau_k, t))\ dt\\
&=&\int_{|t-t_0|\leq l}d(w(\tau_{k+p_k}, t), w(\tau_k, t))\ dt+\int_{|t-t_0|>l}d(w(\tau_{k+p_k}, t), w(\tau_k, t))\ dt\\
&\geq&\int_{|t-t_0|\leq l}d(w(\tau_{k+p_k}, t), w(\tau_k, t))\ dt \geq \epsilon l.
\eeastar
On the other hand, we compute
\beastar
&{}&\int_{S^1}d(w(\tau_{k+p_k}, t), w(\tau_k, t))\ dt\\
&\leq&\int_{S^1}\int_{\tau_k}^{\tau_{k+p_k}}\left|\frac{\del w}{\del s}(s, t)\right|\ ds\ dt
= \int_{\tau_k}^{\tau_{k+p_k}}\int_{S^1}\left|\frac{\del w}{\del s}(s, t)\right|\ dt\ ds\\
&\leq&\int_{\tau_k}^{\tau_{k+p_k}}\left(\int_{S^1}\left|\frac{\del w}{\del s}(s, t)\right|^2\,
 dt\right)^{\frac{1}{2}}\ ds\\
&\leq&\int_{\tau_k}^{\tau_{k+p_k}}Ce^{-\delta s}\ ds
= \frac{C}{\delta}(1-e^{-(\tau_{k+p_k}-\tau_k)})e^{-\tau_k} \leq \frac{C}{\delta}e^{-\tau_k}.
\eeastar
When $\tau_k$ sufficiently large, this inequality gives rise to a contradiction to \eqref{eq:geqepsilon}.
Hence the proof.
\end{proof}

Now using the subsequence convergence from Theorem \ref{thm:subsequence},
we can pick a subsequence $\{\tau_k\}$ and a closed Reeb orbit $(\gamma, T)$ such that
$$
w(\tau_k, t)\to z(t) : = \gamma(Tt), \quad k\to \infty
$$
uniformly in $t$. Then the above lemma  immediately implies
$
w(\tau, t)
$
uniformly converges to $z(t)$ for any $t\in S^1$.

It remains to show that this convergence is uniform in $t$. Suppose to the contrary that
there exist some $\epsilon>0$ and some sequence $(\tau_k, t_k)$ such that
$$
d(w(\tau_k, t_k), z(t_k))\geq2\epsilon.
$$
Since $t_k\in S^1$, we can further take a subsequence, still denoted by $t_k$, such that $t_k\to t_0\in S^1$.
We can take $k$ so large that $d(z(t_k), z(t_0))\leq \frac{1}{2}\epsilon$. We also note
\beastar
d(w(\tau, t_k), w(\tau, t_0))  \leq  \int_{t_0}^{t_k}\left|\frac{\del w}{\del t}(\tau, s)\right|\, ds
\leq(t_k-t_0)\|d w\|_{C^0},
\eeastar
by which we can make the distance less than $\frac{1}{2}\epsilon$ by taking $k$ sufficiently large.

Combing these, we derive
\beastar
d(w(\tau_k, t_0), z(t_0))&\geq&
d(w(\tau_k, t_k), z(t_k))-d(w(\tau_k, t_k), w(\tau_k, t_0))\\
&{}&-d(z(t_k), z(t_0))\\
&\geq&2\epsilon-\frac{1}{2}\epsilon-\frac{1}{2}\epsilon = \epsilon
\eeastar
for all sufficiently large $k$'s.
This gives rise to contradiction to the pointwise convergence $w(\tau_k,t_k) \to z(t_0)$, which
finishes the proof of uniform convergence for $t \in S^1$ and hence completes the proof.
\end{proof}
Then the following $C^0$-exponential convergence immediately follows.
\begin{prop}
There exist some constants $C>0$, $\delta>0$ and $\tau_0$ large such that for any $\tau>\tau_0$,
\beastar
\|d\left( w(\tau, \cdot), z(\cdot) \right) \|_{C^0(S^1)} &\leq& C\, e^{-\delta \tau}
\eeastar
\end{prop}
\begin{proof}
For any $\tau<\tau_+$, similarly as in the previous proof,
\beastar
d(w(\tau, t), w(\tau_+, t))\leq \int^{\tau_+}_{\tau}\left| \frac{\del w}{\del \tau}(s,t) \right|\,ds
\leq \frac{C}{\delta}e^{-\delta \tau}.
\eeastar
Take $\tau_+\rightarrow +\infty$ and using the $C^0$ convergence of $w$ part, i.e.,
Proposition \ref{prop:czero-convergence}, we get
$$d(w(\tau, t), z(t))\leq\frac{C}{\delta}e^{-\delta \tau}.$$
This proves the first inequality.
\end{proof}

\subsection{$C^\infty$-exponential decay of $dw - X_\lambda(w) \, d\tau$}
\label{subsec:Cinftydecaydu}

We recall the coordinate expression of $w = (u,e)$ under the
identification of a tubular neighborhood of $Q$ with
a neighborhood of the zero section of
the normal bundle of $Q$. So far, we have established the following:
\begin{itemize}
\item $W^{1,2}$-exponential decay of the normal component $e$,
\item $L^2$-exponential decay of the derivative $du$ of the base component $u$,
\item $C^0$-exponential convergence of $w(\tau,\cdot) \to z(\cdot)$ as $\tau \to \infty$
for some closed Reeb orbit $z$.
\end{itemize}
Now we are ready to complete the proof of
$C^\infty$-exponential convergence $w(\tau,\cdot) \to z$
by establishing the $C^\infty$-exponential decay of $dw - X_\lambda(w)\, dt$.
The proof of the latter decay is now in order which will be
carried out by the bootstrapping arguments
applied to the system \eqref{eq:contact-instanton-E}.

Combining the above three, we have obtained $L^2$-exponential estimates of the full derivative $dw$.
As already used in Section \ref{subsec:Reeb}, we consider the equation
$$
\delbar_J^{\lambda} w = 0, \quad d(w^*\lambda \circ j) = 0
$$
where $\lambda = f \lambda_E$, under Hypothesis \ref{hypo:exact}.
By the bootstrapping argument using the local uniform a priori estimates
on the cylindrical region (see \cite{oh-wang2} for the details),
we obtain higher order $W^{k,2}$-exponential decays of the term
$$
\frac{\del w}{\del t} - \CT X_\lambda(z), \quad \frac{\del w}{\del\tau}
$$
for all $k\geq 0$, where $w(\tau,\cdot)$ converges to $z$ as $\tau \to \infty$ in $C^0$ sense.
This, combined with the Sobolev's embedding, then completes proof of $C^\infty$-convergence of
$w(\tau,\cdot) \to z$ as $\tau \to \infty$.

\section{Exponential decay: general Morse-Bott case}
\label{sec:general}

In this section, we consider the general case of the Morse-Bott submanifold.
For this one, it is enough to consider the normalized contact triad
$(F, f\lambda_F, J)$ where $J$ is adapted to the zero section $Q$.

Write $w=(u, s)=(u, \mu, e)$, where $\mu\in u^*JT\CN$ and $e\in u^*E$.
By the calculations in Section \ref{sec:coord},
and with similar calculation of Section \ref{sec:prequantization},
the $e$ part can be dealt with exactly the same as in the prequantization case,
whose details are skipped here.

After the $e$-part is taken care of, for the $(u, \mu)$ part, we derive
$$
\left(\begin{matrix} \pi_\theta\frac{\del u}{\del\tau} \\
\nabla_\tau \mu \end{matrix}\right)
+J\left(\begin{matrix} \pi_\theta\frac{\del u}{\del t} \\
\nabla_t \mu \end{matrix}\right)=L,
$$
where $|L|\leq C e^{-\delta \tau}$ similarly as for the prequantization case.

Then we apply the three-interval argument
whose details are similar to the prequantization case and so are omitted.
We only need to establish as in the prequantization case
for the limiting $(\overline{\zeta}_\infty, \overline{\mu}_\infty)$ is not in kernel of $B_\infty$.

If $(\overline{\zeta}_\infty, \overline{\mu}_\infty)$ is in the kernel of $B_\infty$, then by the Morse-Bott
condition, we have $\overline{\mu}_\infty=0$. With the same procedure for introducing the
center of mass, we can use the same argument to prove that $\overline{\zeta}_\infty$ must vanish
if it is contained in the kernel of $B_\infty$.
This will them prove  the following proposition.
\begin{prop} \label{prop:expdecaygeneral}
For any $k=0, 1, \cdots$, there exists some constant $C_k>0$ and $\delta_k>0$
\beastar
\left|\nabla^k \left(\pi\frac{\del u}{\del \tau}\right)\right|<C_k\, e^{-\delta_k \tau}, \quad
\left|\nabla^k \mu\right|<C_k\, e^{-\delta_k \tau}
\eeastar
for each $k \geq 0$.
\end{prop}

\section{The case of asymptotically cylindrical symplectic manifolds}
\label{sec:asymp-cylinder}

In this section, we explain how we can apply the three-interval method and our tensorial scheme
to non-compact symplectic manifolds with \emph{asymptotically cylindrical ends}.
Here we use Bao's precise definition \cite{bao} of the asymptotically cylindrical ends but
restricted to the case where the asymptotical manifold is a contact manifold $(V,\xi)$.
In this section, we will denote a contact manifold by $V$, instead of $M$ which is
what we used in the previous sections,to make comparison of our definition with Bao's transparent.

Let $(V,\xi)$ be a closed contact manifold of dimension $2n+1$ and let $J$ be an almost complex
structure on $W = [0,\infty) \times V$. We denote
\be\label{eq:bfR}
{\bf R}: = J\frac{\del}{\del r}
\ee
a smooth vector field on $W$, and let $\xi \subset TW$ be a subbundle defined by
\be\label{eq:xi}
\xi_{(r,v)} = J T_{(r,v)}(\{r\} \times V) \cap  T_{(r,v)}(\{r\} \times V).
\ee
Then we have splitting
\be\label{eq:splitting-asymp}
TW = \R\{\frac{\del}{\del r}\} \oplus \R\{{\bf R}\} \oplus \xi_{(r,v)}
\ee
and denote by $i: \R\{\frac{\del}{\del r}\} \oplus \R\{{\bf R}\} \to \R\{\frac{\del}{\del r}\} \oplus \R\{{\bf R}\}$
the almost complex structure
$$
i \frac{\del}{\del r} = {\bf R}, \quad i{\bf R} = - \frac{\del}{\del r}.
$$
We denote by $\lambda$ and $\sigma$ the dual 1-forms of $\frac{\del}{\del r}$ and
${\bf R}$ such that $\lambda|_\xi = 0 = \sigma|_\xi$. In particular,
$$
\lambda({\bf R}) = 1 = \sigma(\frac{\del}{\del r}), \quad
\lambda(\frac{\del}{\del r}) = 0 = \sigma({\bf R}).
$$
We denote by $T_s: [0,\infty) \times V \to [-s,\infty) \times V$ the translation $T_s(r,v) = (r+s,v)$
and call a tensor on $W$ is translational invariant if it is invariant under the translation.

The following definition is the special case of the one in \cite{bao}
restricted to the contact type asymptotical boundary.

\begin{defn}[Asymptotically Cylindrical $(W,\omega, J)$ \cite{bao}]
The almost complex structure is called $C^\ell$-\emph{asymptotically cylindrical} if there exists a
2-form $\omega$ on $W$ such that the pair $(J,\omega)$ satisfies the following:
\begin{itemize}
\item[{(AC1)}] $\frac{\del}{\del r} \rfloor \omega = 0 = {\bf R} \rfloor \omega$,
\item[{(AC2)}] $\omega|_\xi(v, J\, v) \geq 0$ and equality holds iff $v = 0$,
\item[{(AC3)}] There exists a smooth translational invariant almost complex structure
$J_\infty$ on $\R \times V$ and constants $R_\ell > 0$ and $C_\ell, \, \delta_\ell > 0$
$$
\|(J - J_\infty)|_{[r,\infty) \times V}\|_{C^\ell} \leq C_\ell e^{-\delta_\ell r}
$$
for all $r \geq R_\ell$. Here the norm is computed in terms of the translational invariant
metric $g_\infty$ and a translational invariant connection.
\item[{(AC4)}] There exists a smooth translational invariant closed 2-form $\omega_\infty$ on
$\R \times V$ such that
$$
\|(\omega - \omega_\infty)|_{[r,\infty) \times V}\|_{C^\ell} \leq C_\ell e^{-\delta_\ell r}
$$
for all $r \geq R_\ell$.
\item[{(AC5)}] $(J_\infty,\omega_\infty)$ satisfies (AC1) and (AC2).
\item[{(AC6)}] ${\bf R}_\infty\rfloor d\lambda_\infty = 0$, where ${\bf R}_\infty: = \lim_{s \to \infty}T_s^*{\bf R}$,
$\lambda_\infty : = \lim_{s \to \infty} T_s^*\lambda$ where both limit exist on $\R \times V$ by (AC3).
\item[{(AC7)}] ${\bf R}_\infty(r,v) = J_\infty\left(\frac{\del}{\del r}\right) \in T_{(r,v)}(\{r\} \times V)$.
\end{itemize}
\end{defn}

For the purpose of current paper, we will restrict ourselves to the case when $\lambda_\infty$ is
a contact form of a contact manifold $(V,\xi)$ and ${\bf R}$ the translational
invariant vector field induced by the Reeb vector field  on $V$
associated to the contact form $\lambda_\infty$ of $(V,\xi)$. More precisely, we have
$$
{\bf R}(r,v) = (0, X_{\lambda_\infty}(v))
$$
with respect to the canonical splitting $T_{(r,v)}W = \R \oplus T_vV$.
Furthermore we also assume that $(V,\lambda_\infty, J_\infty)$ is a contact triad.

Now suppose that $Q \subset V$ is a Morse-Bott submanifold of closed Reeb orbits of $\lambda_\infty$
and that $\widetilde u:[0,\infty) \times S^1 \to W$ is a $\widetilde J$-holomorphic
curve for which the Subsequence Theorem given in section 3.2 \cite{bao} holds.
We also assume that $J_\infty$ is adapted to $Q$ in the sense of Definition \ref{defn:adapted}.
Let $\tau_k \to \infty$ be a sequence such that $a(\tau_k,t) \to \infty$ and
$w(\tau_k,t) \to z$ uniformly as $k \to \infty$ where $z$ is a closed Reeb orbit
whose image is contained in $Q$. By the local uniform elliptic estimates, we may
assume that the same uniform convergence holds on the intervals
$$
[\tau_k, \tau_k+3] \times S^1
$$
as $k \to \infty$. On these intervals, we can write the equation $\delbar_{\widetilde J} \widetilde u = 0$ as
$$
\delbar_{J_\infty} \widetilde u\left(\frac{\del}{\del \tau}\right)
= (\widetilde J - J_\infty) \frac{\del \widetilde u}{\del t}.
$$
We can write the endomorphism $(\widetilde J - J_\infty)(r,\Theta) =: M(r,\Theta)$ where
$(r,\Theta) \in \R \times V$ so that
\be\label{eq:expdkM}
|\nabla^k M(r,\Theta)| \leq C_k e^{-\delta r}
\ee
for all $r \geq R_0$. Therefore $u = (a,w)$ with $a = r\circ \widetilde u$, $w = \Theta \circ \widetilde u$
satisfies
$$
\delbar_{J_\infty} \widetilde u\left(\frac{\del}{\del \tau}\right)  = M(a,w)\left(\frac{\del \widetilde u}{\del t}\right)
$$
Decomposing $\delbar_{J_\infty} \widetilde u$ and $\frac{\del \widetilde u}{\del t}$ with respect to the decomposition
$$
TW = \R \oplus TV = \R\cdot\frac{\del}{\del r} \oplus \R \cdot X_{\lambda_\infty} \oplus \xi
$$
we have derived
\bea
\delbar^{\pi_\xi} w\left(\frac{\del}{\del \tau}\right) & = & \pi_\xi
\left(M(a,w)\left(\frac{\del \widetilde u}{\del t}\right)\right) \label{eq:delbarpiw1}\\
(dw^*\circ j - da)\left(\frac{\del}{\del \tau}\right) & = & \pi_{\C} \left(M(a,w)\left(\frac{\del \widetilde u}{\del t}\right)\right)\label{eq:delbarpiw2}
\eea
where $\pi_\xi$ is the projection to $\xi$ with respect to the contact form $\lambda_\infty$ and
$\pi_\C$ is the projection to $\R\cdot\frac{\del}{\del r} \oplus \R \cdot X_{\lambda_\infty}$
with respect to the cylindrical $(W,\omega_\infty,J_\infty)$.
Then we obtain from \eqref{eq:expdkM}
$$
|\delbar^{\pi_\xi} w| \leq Ce^{-\delta a}
$$
as $a \to \infty$. By the subsequence convergence theorem assumption and local a priori estimates on $\widetilde u$,
we have immediately obtained the following
$$
|\nabla''_\tau e(\tau,t)| \leq C\, e^{\delta_1 \tau},\,
|\nabla''_\tau \xi_\CF(\tau,t)| \leq C\, e^{\delta_1 \tau},\,
|\nabla''_\tau \xi_G(\tau,t)| \leq C\, e^{\delta_1 \tau}
$$
where $w = \exp_Z(\xi_G + \xi_\CF + e)$ is the decomposition similarly as before.
Now we can apply exactly the same proof as the one given in the previous section to establish
the exponential decay property of $dw$.

For the component $a$, we can use \eqref{eq:delbarpiw2} and the argument used in
\cite{oh-wang2} and obtain the necessary exponential property as before.
\smallskip

\appendix

\section{Proof of Proposition \ref{prop:adapted}}
\label{sec:appendix-adapted}

In this appendix, we prove contractibility of the set of $Q$-adapted $CR$-almost complex structures
postponed from the proof of Proposition \ref{prop:adapted}.

We first notice that for any $d\lambda$-compatible $CR$-almost complex structure $J$,
$(TQ\cap JTQ)\cap T\CF=\{0\}$: This is because for any $v\in (TQ\cap JTQ)\cap T\CF$,
\beastar
|v|^2=d\lambda(v, Jv)=0,
\eeastar
since $Jv\in TQ$ and $v\in T\CF=\ker \omega_Q$.
Therefore $(TQ\cap JTQ)$ and $T\CF$ are linearly independent.

We now give the following lemma.

\begin{lem}\label{lem:J-identify-G} $J$ satisfies
the condition $JTQ\subset TQ+JT\CN$ if and only if it satisfies
$TQ =(TQ\cap JTQ)\oplus T\CF$.
\end{lem}
\begin{proof}
It is obvious to see that $TQ=(TQ\cap JTQ)\oplus T\CF$ indicates $J$ is $Q$-adapted.

It remains to prove the other direction.
For this, we only need to prove that $TQ \subset (TQ\cap JTQ) + T\CF$ by the discussion
right in front of the statement of the lemma.

Let $v \in TQ$. By the definition of the adapted condition,
$Jv\in TQ +  JT\CN$. Therefore we can write
$$
Jv=w+Ju,
$$
for some $w\in TQ$ and $u\in T\CN$.
Then it follows that $v=-Jw+u$,
Noting that $Jw\in TQ\cap JTQ$, we derive $v \in (TQ\cap JTQ) + T\CF$ and so
we have finished the proof.
\end{proof}

This lemma shows that any $Q$-adapted $J$ naturally defines a splitting
\be\label{eq:split1}
T\CF \oplus G_J = TQ, \quad G_J:= TQ \cap JTQ.
\ee
We also note that such $J$ preserves the subbundle $TQ + JT\CF \subset TM$ and so defines an
invariant splitting
\be\label{eq:split2}
TM = TQ \oplus JT\CF \oplus E_J; \quad E_J = (TQ \oplus JT\CF)^{\perp_{g_J}}.
\ee
Conversely, for a given splittings \eqref{eq:split1}, \eqref{eq:split2},
we can always choose $Q$-adapted $J$ so that $TQ \cap JTQ = G$
but the choice of such $J$ is not unique.

It is easy to see that the set of such splittings forms a contractible manifold
(see Lemma 4.1 \cite{oh-park} for a proof). We also note that the 2-form
$d\lambda$ induces nondegenerate (fiberwise) bilinear 2-forms on $G$ and $E$ which we denote by $\omega_G$ and
$\omega_E$. Now we denote by $\CJ_{G,E}(\lambda;Q)$ the subset of $\CJ(\lambda;Q)$ consisting of $J \in \CJ(\lambda;Q)$
that satisfy \eqref{eq:split1}, \eqref{eq:split2}. Then $\CJ(\lambda;Q)$ forms a fibration
$$
\CJ(\lambda;Q) = \bigcup_{G,E}\CJ_{G,E}(\lambda;Q).
$$
Therefore it is enough to prove that $\CJ_{G,E}(\lambda;Q)$ is contractible for each fixed $G, \, E$.

We denote each $J: TM \to TM$ as a block $4 \times 4$ matrix in terms of the splitting
$$
TM = T\CF \oplus G \oplus JT\CF \oplus E.
$$
Then one can easily check that the $Q$-adaptedness of $J$ implies $J$ must have the form
$$
\left(\begin{matrix} 0 & 0 & Id & 0 \\
0 & J_G & 0 & 0 \\
-Id & 0 & 0 & 0 \\
0 & B & 0 & J_E
\end{matrix} \right)
$$
where $J_G:G \to G$ is $\omega_G$-compatible and $J_E:E \to E$ is $\omega_E$-compatible, and $B$ satisfies the relation
$
BJ_G = 0
$
which in turn implies $B = 0$.
Since each set of such $J_G$'s or of such $J_E$'s is contractible,
it follows that $\CJ_{G,E}(\lambda;Q)$ is contractible.
This finishes the proof of contractibility of $\CJ(\lambda;Q)$.

\section{Proof of Theorem \ref{thm:subsequence}}
\label{appendix:subseqproof}

In this appendix, we provide the proof of Theorem \ref{thm:subsequence} borrowing
the exposition from \cite{oh-wang2}.

For a given contact instanton $w: [0, \infty)\times S^1\to M$, we define  maps
$w_s: [-s, \infty) \times S^1 \to M$ by
$
w_s(\tau, t) = w(\tau + s, t)$.
For any compact set $K\subset \R$, there exists sufficiently large $s_0$ such that for every $s\geq s_0$,
$K\subset [-s, \infty)$. For such $s\geq s_0$, we also get an $[s_0, \infty)$-family of maps by defining $w^K_s:=w_s|_{K\times S^1}:K\times S^1\to M$.

The asymptotic behavior of $w$ at infinity can be understood by studying the limit of the sequence of maps
$\{w^K_s:K\times S^1\to M\}_{s\in [s_0, \infty)}$, for any compact set $K\subset \R$.

First of all,
it is easy to check that under Hypothesis \ref{hypo:basic}, the family
$\{w^K_s:K\times S^1\to M\}_{s\in [s_0, \infty)}$ satisfies the following
\begin{enumerate}
\item $\delbar^\pi w^K_s=0$, $d((w^K_s)^*\lambda\circ j)=0$, for every $s\in [s_0, \infty)$
\item $\lim_{s\to \infty}\|d^\pi w^K_s\|_{L^2(K\times S^1)}=0$
\item $\|d w^K_s\|_{C^0(K\times S^1)}\leq \|d w\|_{C^0([0, \infty)\times S^1)}<\infty$.
\end{enumerate}

From (1) and (3) together with the compactness of the target manifold $M$ (which provides the uniform $L^2(K\times S^1)$ bound)
and the coercive estimate for contact instanton equation derived in \cite[Theorem 5.7]{oh-wang2}, we obtain
$$
\|w^K_s\|_{W^{3,2}(K\times S^1)}\leq C_{K;(3,2)}<\infty,
$$
for some constant $C_{K;(3,2)}$ independent of $s$.
Then it follows from the compactness of the embedding of $W^{3,2}(K\times S^1)$ into $C^2(K\times S^1)$ that the
set $\{w^K_s:K\times S^1\to M\}_{s\in [s_0, \infty)}$ is sequentially compact.
Therefore, for any sequence $s_k \to \infty$, there exists a subsequence, still denoted by $s_k$,
that converges to a map $w^K_\infty\in C^2(K\times S^1, M)$ in $C^2(K\times S^1, M)$ as $k\to \infty$.

Combined with (2), we derive the convergence
$$
dw^K_{s_k}\to dw^K_{\infty} \quad \text{and} \quad dw^K_\infty=(w^K_\infty)^*\lambda\otimes X_\lambda.
$$
Finally by taking (1) into consideration,
we also derive that both $(w^K_\infty)^*\lambda$ and $(w^K_\infty)^*\lambda\circ j$ are harmonic $1$-forms.

Recall that these limiting maps $w^K_\infty$ have common extension $w_\infty: \R\times S^1\to M$
by the nature of the diagonal argument which takes a sequence of compact sets $K$
in the way one including another and exhausting full $\R$ as $k \to \infty$.
Then $w_\infty$ is $C^2$ (actually $C^\infty$) and satisfies
$$
\|d w_\infty\|_{C^0(\R\times S^1)}\leq \|d w\|_{C^0([0, \infty)\times S^1)}<\infty
$$
and
$
dw_\infty=(w_\infty)^*\lambda \otimes X_\lambda.
$
We also note that both $(w_\infty)^*\lambda$ and $(w_\infty)^*\lambda\circ j$
are bounded harmonic one-forms on $\R\times S^1$.
Therefore they must be written into the forms
$$
(w_\infty)^*\lambda=a\,d\tau+b\,dt, \quad (w_\infty)^*\lambda\circ j=b\,d\tau-a\,dt,
$$
where $a$, $b$ are some constants.
Now we show that such $a$ and $b$ are actually related to $\CT$ and $\CQ$ as
follows
\begin{lem}
$$
a=-\CQ, \quad b= \CT.
$$
\end{lem}
\begin{proof} Take an arbitrary point $r\in K$. Using the $C^2$-convergence of some sequence $w_{s_k}|_{\{r\}\times S^1}$
to $w_\infty|_{\{r\}\times S^1}$, we derive
\beastar
b=\int_{\{r\}\times S^1}(w_\infty|_{\{r\}\times S^1})^*\lambda
&=&\int_{\{r\}\times S^1}\lim_{k\to \infty}(w_{s_k}|_{\{r\}\times S^1})^*\lambda\\
&=&\lim_{k\to \infty}\int_{\{r\}\times S^1}(w_{s_k}|_{\{r\}\times S^1})^*\lambda\\
&=&\lim_{k\to \infty}\int_{\{r+s_k\}\times S^1}(w|_{\{r+s_k\}\times S^1})^*\lambda.
\eeastar
On the other hand, recalling $w^*d\lambda = \frac{1}{2} |d^\pi w|^2)$ and applying Stokes'
formula and finiteness of the $\pi$-energy on $[0,\infty) \times S^1$, the latter becomes
$$
\lim_{k\to \infty}(T-\frac{1}{2}\int_{[r+s_k, \infty)\times S^1}|d^\pi w|^2)
=\CT-\lim_{k\to \infty}\frac{1}{2}\int_{[r+s_k, \infty)\times S^1}|d^\pi w|^2
=\CT
$$
whjch proves $a = \CT$.
On the other hand, using the closeness of $w^*\lambda \circ j$
and Stokes' formula, we easily compute
\beastar
-a=\int_{\{r\}\times S^1}(w_\infty|_{\{r\}\times S^1})^*\lambda\circ j
&=&\int_{\{r\}\times S^1}\lim_{k\to \infty}(w_{s_k}|_{\{r\}\times S^1})^*\lambda\circ j\\
&=&\lim_{k\to \infty}\int_{\{r\}\times S^1}(w_{s_k}|_{\{r\}\times S^1})^*\lambda\circ j\\
&=&\lim_{k\to \infty}\int_{\{r+s_k\}\times S^1}(w|_{\{r+s_k\}\times S^1})^*\lambda\circ j
=\CQ.
\eeastar
Here in our derivation, we used Remark \ref{rem:TQ}. This proves the lemma.
\end{proof}

By the connectedness of $[0,\infty) \times S^1$, the image of $w_\infty$ is contained in
a single leaf of the Reeb foliation. If $\gamma: \R \to M$ is a parametrization of
the leaf so that it satisfies $\dot \gamma = X_\lambda(\gamma)$,
then we can write $w_\infty(\tau, t)=\gamma(s(\tau, t))$, where
$s:\R\times S^1\to \R$ and $s=-\CQ\,\tau+\CT\,t+c_0$ since $ds=-\CQ\,d\tau+\CT\,dt$, where $c_0$ is some constant.

This implies that if $\CT\neq 0$
$\gamma$ defines a closed Reeb orbit of period $T$. On the other hand,
If $T=0$ but $\CQ\neq 0$, we can only conclude that $\gamma$ is some Reeb
trajectory parameterized by $\tau\in \R$.

\begin{rem} Of course, if both $\CT$ and $\CQ$ vanish, $w_\infty$ is a constant map.
In \cite{oh:energy}, it is shown that such a puncture is a removable singularity
under the finiteness of a suitably defined Hofer-type energy.
\end{rem}

\section{Sobolev's inequality for the sections of $\CE_1 \to \R$}
\label{sec:Sobolev}

In this section, we give the proof of \eqref{eq:Sobolev} for the sections of
the bundle $\CE_1 \to \R$ whose fiber is a Hilbert space possibly with infinite dimension.

As in the main text, we assume $\CE_2 \subset \CE_1$ a pair of Hilbert bundles that satisfies
all the properties and  is equipped with a compatible connection $\nabla$.
We denote
by $\Pi_s^\tau$ the parallel transport from the fiber $\CE_s$ to $\CE_\tau$.

\begin{prop}\label{prop:Sobolev} Let $I \subset \R$ be a closed
interval and $\zeta: I \to \CE_2$ be a smooth section.
Then there exists $C_3 = C_3(I) > 0$ depending only on the length $|I|$ of the interval
but independent of $\zeta$ such that
$$
\|\zeta(\tau)\|_{L^\infty(I,\CE_1)} \leq C_3 \|\zeta\|_{W^{1,2}(I,\CE_1)}.
$$
\end{prop}
\begin{proof}
Thanks to \eqref{eq:L2|zeta|}, there must be a point
$\tau_0 \in I_k$ such that
\be\label{eq:suptau0}
|\zeta(\tau_0)|_{\CE_1,\tau} \leq \frac{1}{\sqrt{|I|}}\|\zeta\|_{L^2(I,\CE_1)}
\ee
where $|I|$ is the length of the interval $I$.
Then for any $\tau \in I$, we write
$$
\zeta(\tau) - \Pi_{\tau_0}^\tau \zeta(\tau_0) = \int_{\tau_0}^\tau \Pi_{s}^\tau \nabla_s \zeta(s) \, ds
$$
Therefore we obtain
$$
|\zeta(\tau)|_{\CE_1,\tau} \leq |\zeta(\tau_0)|_{\CE_1,\tau_0} + \int_{\tau_0}^\tau |\nabla_s \zeta(s)|_{\CE_1,s} \, ds.
$$
Applying the H\"older's inequality, we derive
\beastar
\int_{\tau_0}^\tau |\nabla_s \zeta(s)|_{\CE_1,\tau} \, ds
&\leq & \sqrt{|I|} \sqrt{\int_{\tau_0}^\tau  |\nabla_s \zeta(s)|_{\CE_1,\tau}^2 \, ds}\\
&\leq & \sqrt{|I|}  \sqrt{\int_I |\nabla_s \zeta(s)|_{\CE_1,\tau}^2 \, ds}
 \leq  \sqrt{|I|}\,  \|\nabla_\tau \zeta\|_{L^2(I,\CE_1)}
\eeastar
since $\tau_0,\, \tau \in I$.
Combining the two, we have obtained
$$
|\zeta(\tau)|_{\CE_1,\tau} \leq \frac{1}{\sqrt{|I|}}\,\|\zeta\|_{L^2(I,\CE_1)} +
\sqrt{|I|}\, \|\nabla_\tau \zeta\|_{L^2(I,\CE_1)}
$$
for all $\tau \in I$.
By setting $C_3 = 2 \max\{\sqrt{|I|}, \frac{1}{\sqrt{|I|}}\}$, we have finished the proof.
\end{proof}

\bigskip

\textbf{Acknowledgements:}
Rui Wang would like to thank Erkao Bao and Ke Zhu for useful discussions.
Both authors are extremely grateful to the anonymous referee, who read the paper
carefully, pointed out many typos and errors and provided valuable
suggestions to improve the presentation of this paper.

\end{document}